\documentclass[10pt,oneside,reqno]{amsart}

\usepackage[utf8]{inputenc}
\usepackage[english]{babel}

\usepackage{color}
\usepackage{enumerate}
\usepackage{enumitem}
\usepackage{fancyhdr}
\usepackage{upgreek}
\usepackage[T1]{fontenc}

\usepackage{accents}
\usepackage{centernot}
\usepackage{stmaryrd}
\usepackage[super]{nth}

\usepackage{amsmath}
\usepackage{amscd}
\usepackage{amsfonts}
\usepackage{amsthm}
\usepackage{amstext}
\usepackage{amssymb}
\usepackage{amsbsy}
\usepackage{mathrsfs}
\usepackage{mathtools}
\usepackage{array}
\usepackage{tikz-cd}
\usepackage{bm}

\usepackage{float}
\usepackage{caption}
\usepackage{subcaption}
\usepackage{graphicx}

\usepackage{tikz}
\usetikzlibrary{decorations.markings}
\usetikzlibrary{decorations.pathreplacing}
\usetikzlibrary{calligraphy}
\usetikzlibrary{positioning}
\usetikzlibrary{arrows}
\usetikzlibrary{fadings}
\usetikzlibrary{math}

\usepackage[nobysame]{amsrefs}

\usepackage{hyperref}
\hypersetup{
	colorlinks=true,
	linkcolor=blue,
	citecolor=red,
	filecolor=cyan,    
	urlcolor=magenta,
}

\newtheorem{theorem}[subsection]{Theorem}

\newtheorem{proposition}[subsection]{Proposition}

\theoremstyle{definition}
\newtheorem{definition}[subsection]{Definition}

\theoremstyle{remark}
\newtheorem{remark}[subsection]{Remark}

\title[Non-uniqueness \& inadmissibility of the vanishing viscosity limit]{Non-uniqueness \& inadmissibility of the vanishing viscosity limit of the passive scalar transport equation}
\author[1]{Lucas Huysmans}
\author[2]{Edriss S. Titi}
\address[1]{Lucas Huysmans \hfill\break Department of Pure Mathematics \& Mathematical Statistics, University of Cambridge, Cambridge CB3 0WB. Current address: Max Planck Institute for Mathematics in the Sciences, Leipzig 04103, Germany.}
\email{lucas.huysmans@mis.mpg.de}
\address[2]{Edriss S. Titi \hfill\break Department of Mathematics, Texas A\&M University, College Station, TX 77840, USA. Department of Applied Mathematics and Theoretical Physics, University of Cambridge, Cambridge CB3 0WA, UK. Department of Computer Science and Applied Mathematics, Weizmann Institute of Science, Rehovot 76100, Israel.}
\email{titi@math.tamu.edu and Edriss.Titi@maths.cam.ac.uk}
\date{\today}
\thanks{{\it Keywords}: Passive scalar transport equation; non-uniqueness of vanishing viscosity 
	limit; nonphysical and entropy inadmissibility of vanishing viscosity 
	limit; vanishing viscosity limit is not a selection principle for weak 
	solutions of passive scalar transport equation; peculiar energy cascade 
	scenario of turbulent transport; inviscid non-uniqueness; inviscid mixing and unmixing}

\begin{document}
	
	\begin{abstract}
		We study selection by vanishing viscosity for the transport of a passive scalar $f(x,t)\in\mathbb{R}$ advected by a bounded, divergence-free vector field $u(x,t)\in\mathbb{R}^2$. This is described by the initial value problem to the PDE $\frac{\partial f}{\partial t} + \nabla\cdot (u f) = 0$, or with positive viscosity/diffusivity $\nu>0$, to the PDE $\frac{\partial f}{\partial t} + \nabla\cdot (u f) -\nu\Delta f = 0$. We demonstrate the failure of the vanishing viscosity limit to select (a) unique solutions or (b) physically admissible solutions in the sense of non-increasing energy/entropy.
	\end{abstract}
	\maketitle
	\setcounter{tocdepth}{1}
	\tableofcontents
	\newpage
	\section{Introduction}\label{Introduction}
	
	\subsection{Background}
	We consider the selection problem among non-unique weak solutions to the passive scalar transport equation. Consider a divergence-free vector field $u:\mathbb{T}^d\times[0,T]\to\mathbb{R}^d$ on a $d$-dimensional spatial torus $\mathbb{T}^d=\mathbb{R}^d/\mathbb{Z}^d$ and a time interval $[0,T]$. For a passive scalar $f:\mathbb{T}^d\times[0,T]\to\mathbb{R}$, the passive transport of $f$ by $u$ is given in conservation-form by the following transport equation:
	\begin{equation*}
		\frac{\partial f}{\partial t} + \nabla\cdot (u f) = 0. \tag{TE}
	\end{equation*}
	This partial differential equation is understood in the distributional sense and admits non-unique weak solutions to the initial value problem for specific non-smooth vector fields outside the Sobolev class $L_t^1W_x^{1,1}$, see \cites{aizenman1978vector,diperna1989ordinary,depauw2003non}, and also recent constructions with sharper regularity \cites{modena2018non,cheskidov2021nonuniqueness}.
	
	When $f$ represents a physical scalar such as the advection of heat, or the concentration of a tracer transported by an incompressible fluid $u$, many of these non-unique weak solutions exhibit thermodynamically inadmissible behaviour. For example, in the construction of \cites{depauw2003non}, one finds that the passive scalar `perfectly unmixes' from a constant initial datum to an inhomogeneous final state, exhibiting some kind of separation of particles akin to a decrease in the thermodynamic entropy of the system.
	
	When studied at the level of particle trajectories, one has, by the superposition principle in \cite{Ambrosio2014}, that such a perfectly unmixing weak solution as in \cite{depauw2003non} corresponds to particle paths which fail the Markov property. On the other hand, in physical systems, molecular chaos at the kinetic level should enforce such Markovianality. Hence, one may view the `perfect unmixing' phenomenon as a physically inadmissible failure of the second law of thermodynamics. The incorporation of such thermodynamic restrictions on weak solutions is well studied in the form of entropy inequalities for hyperbolic conservation laws, see for instance \cite{dafermos2005hyperbolic}.
	
	One approach to recover physically admissible behaviour is introducing diffusivity into the model. This advection-diffusion of $f$ by $u$ with viscosity/diffusivity $\nu>0$ is given by the initial value problem to the following advection-diffusion equation:
	\begin{equation*}
		\frac{\partial f}{\partial t} + \nabla \cdot (uf) - \nu\Delta f = 0. \tag{$\nu$-ADE}
	\end{equation*}
	In contrast to the inviscid case $\nu=0$, such parabolic equations exhibit additional smoothness and remain well-posed under lower regularity of the coefficients, see for example \cites{evans2010partial,flandoli2010well}. Indeed, as shown in \cite{bonicatto2021advection}, the inclusion of time-irreversible viscosity/diffusion removes the previous zoo of non-physical solutions, in line with the philosophy that diffusion enforces molecular chaos at the level of particle trajectories. Viewing \eqref{eqADE} as the Fokker-Planck equation for the SDE on particle paths $dX_t = u(X_t,t) dt + \sqrt{2\nu} dW_t$, see for instance \cite{figalli2008existence}, the unique path measure for $X_t$ now satisfies the strong Markov condition under suitable integrability assumptions on $u$, \cite{krylov2005strong}.
	
	Consequently, it is relevant to study if the vanishing viscosity limit also recovers such physically admissible behaviour. Indeed, one observes that for constant initial data, the vanishing viscosity limit must uniquely select the corresponding constant solution, avoiding, in particular, the solutions as mentioned earlier which perfectly unmix, as in \cite{depauw2003non}.
	
	Also, more generally, selection by vanishing viscosity is well studied for other fluid models, \cite{dafermos2005hyperbolic,Bianchini2005,bardos2012vanishing,bardos2013stability,Nussenzveig_Lopes_2021}. In these examples the vanishing viscosity limit does select a unique solution to the inviscid model, therefore referred to as the physical solution. This literature indicates the importance of vanishing viscosity for the physical admissibility of weak solutions, more generally in fluid mechanics. In contrast to this status quo, we give examples wherein the vanishing viscosity limit fails to restore the uniqueness or physical admissibility of the inviscid system. The precise statement of these results is stated in Theorem \ref{vvtheorem} and \ref{perfmixvv} in the following section.
	
	
	Theorem \ref{perfmixvv}, which is the main contribution of this work, is the statement that there exists a bounded, divergence-free vector field $u$, such that the vanishing viscosity limit of the advection-diffusion equation converges \textit{uniquely} to the inadmissible `mixing-unmixing' solution of the inviscid problem, for all initial data. This solution perfectly mixes to its spatial average, remains this constant average for a fixed time, and later perfectly unmixes back to its initial state. In particular, any arrow of time (or Markovianality) enforced by diffusion is lost in its vanishing limit.
	
	Meanwhile, our first result Theorem \ref{vvtheorem} gives the existence of a bounded, divergence-free vector field $u$, and two different vanishing viscosity subsequences, along each of which the viscous solution converges strongly to different inviscid solutions for all initial data. These inviscid solutions are, moreover, renormalised solutions, see Definition \ref{renormalised} or \cite{diperna1989ordinary}, and entropy preserving, see Section \ref{sec:introentropy}. From the viewpoint of Lagrangian trajectories, this failure of the renormalised condition to select unique solutions can be seen as a consequence of the non-uniqueness of a corresponding Lagrangian flow map, see Definition \ref{lagrangian}, which is injective for all except a null set of times.
	
	\subsection{Main results}
	We introduce in Section \ref{Notation} standard notation of function spaces, and rigorous definitions of weak solutions of the transport equation (Definition \ref{TE}), weak renormalised solutions of the transport equation (Definition \ref{renormalised}), and weak solutions of the advection-diffusion equation with viscosity $\nu>0$ (Definition \ref{ADE}). We also give our own proof of the well-posedness of the advection-diffusion equation in Section \ref{Background}, Theorem \ref{ADEwellposed}.
	{\renewcommand{\thesubsection}{\ref{vvtheorem}}
		\begin{theorem}[Non-unique renormalised vanishing viscosity limit]
			There exists a divergence-free vector field $u \in L^\infty([0,1];L^\infty(\mathbb{T}^2;\mathbb{R}^2))$, and a sequence $\{\nu_n\}_{n\in\mathbb{N}}$ with $\nu_n > 0$ and $\nu_n\xrightarrow{n\to\infty}0$, such that for any initial data $f_0\in  L^\infty(\mathbb{T}^2)$, and for $f^\nu$ the unique solution to \eqref{eqADE} along $u$ with initial data $f_0$, one has
			\begin{gather*}
				f^{\nu_{2n}}\xrightarrow{n\to\infty}f^\mathrm{even}, \\
				f^{\nu_{2n+1}}\xrightarrow{n\to\infty}f^\mathrm{odd},
			\end{gather*}
			with the above convergence in weak-$*$ $L^\infty([0,1];L^\infty(\mathbb{T}^2))$, and strong in $L^p([0,1];L^p(\mathbb{T}^2))$ for all $p \in [1,\infty)$. Furthermore, the limits $f^\mathrm{even}, f^\mathrm{odd}$ are renormalised weak solutions to \eqref{eqTE} along $u$ with initial data $f_0$.
			
			If $f_0$ is not constant, then $f^\mathrm{even}\ne f^\mathrm{odd}$, and moreover the set of weak-$*$ limit points of $f^\nu\in L^\infty([0,1];L^\infty(\mathbb{T}^2))$ as $\nu\to0$ is uncountable.
	\end{theorem}}
	
	{\renewcommand{\thesubsection}{\ref{perfmixvv}}
		\begin{theorem}[Inadmissible vanishing viscosity limit]
			There exists a divergence-free vector field $u \in L^\infty([0,100];L^\infty(\mathbb{T}^2;\mathbb{R}^2))$, such that for any initial data $f_0\in L^\infty(\mathbb{T}^2)$, and for $f^\nu$ the unique solution to \eqref{eqADE} along $u$ with initial data $f_0$, one has
			\begin{equation*}
				f^\nu\xrightarrow{\nu\to0}f,
			\end{equation*}
			with the above convergence in weak-$*$ $L^\infty([0,100];L^\infty(\mathbb{T}^2))$, strong in $L^p([0,42];L^p(\mathbb{T}^2))$, $C^0([0,42-\epsilon];L^p(\mathbb{T}^2))$, $L^p([58,100];L^p(\mathbb{T}^2))$, and $C^0([58+\epsilon,100];L^p(\mathbb{T}^2))$ for all $p \in [1,\infty)$, and all $\epsilon >0$. The limit function $f \in C_{\mathrm{weak-}*}^0([0,100];L^\infty(\mathbb{T}^2))$ is a weak solution to \eqref{eqTE} along $u$ with initial data $f_0$.
			
			Moreover, for all $t \in [42,58]$
			\begin{equation*}
				f(\cdot, t) \equiv \int_{\mathbb{T}^2}f_0(y)\;dy,
			\end{equation*}
			is perfectly mixed to its spatial average.
			
			Furthermore, for all $t \in [0,100]$, $f(\cdot, t)=f(\cdot,100-t)$ and in particular,
			\begin{equation*}
				f(\cdot, 100) = f_0,
			\end{equation*}
			is perfectly unmixed. In particular, if $f_0$ is not constant, any $L^p(\mathbb{T}^2)$ norms of $f(\cdot, t)$ (for $p\in(1,\infty]$) increase after $t=58$, contrary to the entropy-admissibility criterion in \cite{dafermos2005hyperbolic}.
	\end{theorem}}
	
	\setcounter{subsection}{2}
	\subsection{Entropy of inviscid transport}\label{sec:introentropy}
	It is possible to formulate a physical condition akin to the Markovianalilty of the path measure at the level of weak solutions to the partial differential equation. By formulating an appropriate entropy for the passive scalar $f$, see \cite{dafermos2005hyperbolic}, one may impose the condition of non-decreasing (or, depending on the convention, non-increasing) entropy. In particular, \eqref{eqTE} when written in conservation form is a (linear) scalar hyperbolic conservation law, and so any scalar function $\eta(f)$ (for $\eta:\mathbb{R}\to\mathbb{R}$) is an entropy, \cite{kruvzkov1970first,dafermos2005hyperbolic}. As in \cite{dafermos2005hyperbolic}, one argues that the second law of thermodynamics for physical weak solutions to the fluid model should ensure that the weak solution $f$ of \eqref{eqTE} satisfies
	\begin{equation*}
		\frac{\partial \eta(f)}{\partial t} + \nabla\cdot(u\eta(f))\le 0,
	\end{equation*}
	for any convex function $\eta:\mathbb{R}\to\mathbb{R}$, with the above understood in the sense of distribution. We note that for a renormalised weak solution, Definition \ref{renormalised}, this is an equality, and so the solution preserves entropy.
	
	$\eta(f)$ is often taken to be the `energy' $f^2$. This is because for any symmetric hyperbolic conservation law (see again \cite{dafermos2005hyperbolic}), the energy of the vector field $|v(x,t)|^2$ is an appropriate choice of entropy. The above entropy condition is often re-termed the `local energy inequality' in this context. We mention that in the context of the incompressible Euler equations, the local energy inequality has a separate history of study coming from anomalous dissipation in turbulence and point to the historical works \cite{kolmogorov1941local,constantin1994onsager,de2010admissibility}.
	
	In particular, our construction in Theorem \ref{perfmixvv} violates entropy-admissibility since for example the `bulk energy' $\int_{\mathbb{T}^d} |f(x,t)|^2 \; dx$ is strictly increasing at $t=58$ for any (non-constant) initial datum.
	
	This also highlights that vanishing viscosity does not produce a Markovian path measure along $u$. Instead, it remains an open problem to construct a suitable such measure (specifically, a Markov regular Lagrangian flow, \cite{Ambrosio2014}), or even entropy-admissible weak solutions to the initial value problem, for an arbitrary velocity field $u$. Indeed, even the existence of a solution whose bulk energy is non-increasing remains beyond the current literature for a general bounded, divergence-free vector field $u \in L^\infty(\mathbb{T}^d\times[0,T];\mathbb{R}^d)$.
	
	\subsection{Previous work on anomalous dissipation}
	That neither the bulk, nor local, energy inequalities are enforced in the vanishing viscosity limit of the passive scalar transport equation is also significant in the context of the recent work on anomalous dissipation.
	
	The limiting solution in Theorem \ref{perfmixvv} exhibits anomalously, i.e. without viscosity, a dissipation of energy (and later a reverse of that dissipation). In contrast, as originally stated by Kolmogorov as the fundamental assumption used to derive the energy cascade of hydrodynamic turbulence (see \cite{kolmogorov1941local}), anomalous dissipation is the assumption that the vector field $v_\nu$ (which should solve the incompressible Navier-Stokes equation with viscosity $\nu>0$) should satisfy
	\begin{equation*}
		\lim_{\nu\to0} \nu \int_{\mathbb{T}^d\times[0,T]} |\nabla v_\nu(x,t)|^2\;dxdt > 0.
	\end{equation*}
	
	This is motivated by the energy balance (for classical solutions) to the Navier-Stokes equations,
	\begin{equation*}
		\int_{\mathbb{T}^d} |v_\nu(x,T)|^2 \;dx = \int_{\mathbb{T}^d} |v_\nu(x,0)|^2 \;dx - 2\nu \int_{\mathbb{T}^d\times[0,T]} |\nabla v_\nu(x,t)|^2\;dxdt,
	\end{equation*}
	and so in particular if there is anomalous dissipation in the sense of Kolmogorov then the inviscid limit cannot conserve energy.
	
	There has been a great deal of recent work on the analogue of this phenomenon for the passive scalar transport equation, see for example \cites{corrsin1951spectrum,obukhov1949,drivas2019anomalous,armstrong2023anomalous,scoop}. 
	
	Similarly, in the context of passive scalars, if one can ensure that
	\begin{equation*}
		\lim_{\nu\to0} \nu \int_{\mathbb{T}^d\times[0,T]} |\nabla f_\nu(x,t)|^2\;dxdt > 0,
	\end{equation*}
	then the inviscid limit must dissipate energy. It is important to note that, unlike the unmixing construction of Theorem \ref{perfmixvv}, energy lost by such enhanced diffusion (see for example \cites{constantin2008diffusion,bedrossian2021almost,elgindi2023optimal} on this topic) may not be recovered, and indeed anomalous dissipation gives an arrow of time on inviscid solutions. However, if the inviscid limit fails to converge strongly, additional energy will dissipate. None of the examples in the current literature exhibiting anomalous dissipation for passive scalar transport \cites{drivas2019anomalous,scoop,armstrong2023anomalous} show that strong convergence of the inviscid limit fails, and whether further dissipation occurred or is even possible was unknown. Theorem \ref{perfmixvv} answers these questions, proving in particular that mixing does not imply anomalous dissipation.
	
	Correspondingly, whether the dissipation of energy in the infinite Reynolds number limit of full hydrodynamic turbulence is due solely to anomalous dissipation, or whether there is additional dissipation through weak convergence as above, and whether the arrow of time given by a local energy inequality for the incompressible Euler equations should hold in the vanishing viscosity limit of Navier-Stokes, remain critical open questions. In particular, wild solutions (those violating the local energy inequality) to the Euler equations by now have a long history of study \cites{scheffer1993inviscid,shnirelman1997nonuniqueness,constantin1994onsager,de2013dissipative,buckmaster2017onsager,buckmaster2019nonuniqueness}. We highlight that it remains an open question whether such solutions arise in the vanishing viscosity limit of strong solutions to the Navier-Stokes equations. Our result answers the analogue of this question for the passive scalar transport equation in the affirmative.
	
	\subsection{Previous work on non-uniqueness}
	Part of this recent work on anomalous dissipation, see in particular \cite{scoop,armstrong2023anomalous}, similarly deduces non-uniqueness of the vanishing viscosity limit as we show in Theorem \ref{vvtheorem}. The authors of \cite{scoop} show the non-uniqueness of weak limit points of vanishing viscosity for a single initial datum. This result is intriguing since one viscosity subsequence converges weakly to a solution dissipating energy and the other to a solution that preserves energy. Meanwhile, the authors of \cite{armstrong2023anomalous} deduce non-uniqueness of weak limit points for all initial data in $H_x^1$, with the required subsequences depending on the $H_x^1$-norm of the initial datum. In contrast, in Theorem \ref{vvtheorem}, we show non-uniqueness for viscosity subsequences independent of the initial datum, and so also for initial datum outside of $H_x^1$, and with the limits converging strongly to renormalised/entropy-preserving weak solutions.
	
	\subsection{Further comments}
	In this paper, we develop a simple framework (different to homogenisation, \cite{armstrong2023anomalous}, or stochastic calculus, \cite{scoop}) to give control of the vanishing viscosity limit, see Section \ref{vvLimit} and specifically Theorem \ref{vcontrol}. Despite diffusion not acting directly on the background velocity field $u$, diffusion of $f$ gives an indirect emergent smoothing of $u$. Therefore, we can approximate a weak solution $f$ to \eqref{eqADE} by inviscid transport \eqref{eqTE} along the vector field $\bar{u}$ which smooths spatial scales of $u$ below a dissipation length.
	
	The construction of Theorem \ref{vvtheorem} relies on an arrangement of alternating shear flows obeying a specific commutation property, Proposition \ref{shearcanc}. By playing with a second length scale independent of this commutation property, specific viscosity turns on only some of these shear flows. By design, the global effect switches back and forth between null and a large-scale shear as viscosity vanishes. We mention that alternating shear flows have been the basis of numerous recent constructions for perfect mixing and enhanced and anomalous dissipation \cites{alberti2019exponential,drivas2019anomalous,scoop,armstrong2023anomalous,elgindi2023optimal,elgindi2023norm}.
	
	The construction of Theorem \ref{perfmixvv} relies instead on moving slabs to mix the passive scalar for any initial datum, reminiscent of work on non-uniqueness of trajectories, \cite{aizenman1978vector}. To ensure unique convergence of the vanishing viscosity limit in-between the well-adapted viscosity subsequence given by Theorem \ref{vcontrol}, one introduces a third intermittency length scale to control the error. This technique is perhaps reminiscent of the intermittency of small-scale structures in hydrodynamic turbulence.
	
	The analysis of Section \ref{vvLimit} is non-quantitative, relying only on norm convergence of certain limits, and so hold for other regularisation of the passive scalar transport equation such as mollification of the velocity field, or hyperviscosity on the passive scalar. Thus, while some of the details of Theorem \ref{vvtheorem} and Theorem \ref{perfmixvv} are specific to vanishing viscosity, analogous results hold for other vanishing regularisations.
	
	Such vanishing regularisation methods are the only tools in the literature for solving the initial value problem outside of classical well-posedness classes, \cite{diperna1989ordinary,ambrosio2004transport}. Meanwhile, our results show that this is insufficient to give a unique selection or even the existence of `physically relevant' solutions. This raises the necessity of alternative approaches to solving the issues of selection and the existence of physical solutions.

	\subsection{Outline of paper}
	Section \ref{Notation} introduces the standard notation and definitions used throughout the remainder of this paper.
	
	Section \ref{Background} contains necessary background uniqueness, existence, and regularity results for the transport equation and advection-diffusion equation.
	
	In Section \ref{vvLimit}, we give the general framework (and intuition) used to control the vanishing viscosity limit in the proofs of Theorems \ref{vvtheorem}, \ref{perfmixvv}.
	
	Section \ref{Nonuniqueness} is then devoted to the construction and proof of Theorem \ref{vvtheorem}.
	
	Section \ref{Inadmissibility} is instead devoted to the construction and proof of Theorem \ref{perfmixvv}, and may be read independent of Section \ref{Nonuniqueness}.
	
	\section{Notation and definitions}\label{Notation}
	
	\subsection{Notation}
	Denote by $\mathbb{T}^d = \mathbb{R}^d / \mathbb{Z}^d$ the $d$-dimensional unit torus. Throughout this paper, we work on the spatial domain $\mathbb{T}^d$, or on the spatiotemporal domain $\mathbb{T}^{d}\times [0,T]$ for some $T>0$. The gradient operator $\nabla$ and Laplacian $\Delta$ will act on spatial coordinates only.
	
	If the dimension $d$ and time interval $[0,T]$ are implicit we use the shorthand notation $L^pL^q$ for $L^p\left([0,T];L^q(\mathbb{T}^d)\right)$ and similar spaces. Meanwhile, $L^q(\mathbb{T}^d)$ is shorthand for scalar functions $L^q(\mathbb{T}^d;\mathbb{R})$.
	
	For integrability exponents $p,q,... \in [1,\infty]$ (which are allowed to be infinite unless otherwise stated).
	
	For a positive integer $n\in\mathbb{N}$ we denote the homogeneous Sobolev space
	\begin{equation*}
		\dot{H}^n(\mathbb{T}^d)=\left\{f\in L^2(\mathbb{T}^d):\int_{\mathbb{T}^d}f(x)\;dx=0, \hbox{ and } (-\Delta)^{\frac{n}{2}}f\in L^2(\mathbb{T}^d)\right\},
	\end{equation*}
	where the fractional Laplacian $(-\Delta)^{\frac{n}{2}}$ is defined in terms of the Fourier transform/series $\mathcal{F}$ by $\mathcal{F}((-\Delta)^{\frac{n}{2}}f) = |\xi|^n \mathcal{F}(f)$. We denote the norm in $\dot{H}^n(\mathbb{T}^d)$ by
	\begin{equation*}
		\left\|f\right\|_{\dot{H}^n(\mathbb{T}^d)} = \left\|(-\Delta)^{\frac{n}{2}}f\right\|_{L^2(\mathbb{T}^d)},
	\end{equation*}
	for all $f \in \dot{H}^n(\mathbb{T}^d)$. Moreover, we define the Sobolev space
	\begin{equation*}
		H^n(\mathbb{T}^d) = \left\{f\in L^2(\mathbb{T}^d):(-\Delta)^{\frac{n}{2}}f\in L^2(\mathbb{T}^d)\right\},
	\end{equation*}
	with the norm
	\begin{equation*}
		\left\|f\right\|_{H^n(\mathbb{T}^d)} = \left\|f\right\|_{L^2(\mathbb{T}^d)} + \left\|(-\Delta)^{\frac{n}{2}}f\right\|_{L^2(\mathbb{T}^d)}.
	\end{equation*}
	Furthermore, we denote by $H^{-n}(\mathbb{T}^d)$ the dual space of $H^n(\mathbb{T}^d)$.
	
	For each $p \in [1,\infty]$ we define also the space
	\begin{equation*}
		C^0L^p = \left\{f:[0,T]\to L^p(\mathbb{T}^d) \text{ such that $f$ is continuous}\right\} \subset L^\infty L^p,
	\end{equation*}
	with the norm
	\begin{equation*}
		\left\|f\right\|_{C^0L^p} = \sup_{0\le t \le T} \left\|f(\cdot,t)\right\|_{L^p(\mathbb{T}^d)}.
	\end{equation*}
	Moreover, we define the linear space
	\begin{equation*}
		C_{\mathrm{weak}}^0L^p = \left\{f:[0,T]\to L^p(\mathbb{T}^d) \text{ such that $f$ is continuous in $L^p_{\mathrm{weak}}$}\right\} \subset L^\infty L^p,
	\end{equation*}
	and for $p \in (1,\infty]$ the linear space
	\begin{equation*}
		C_{\mathrm{weak-}*}^0L^p = \left\{f:[0,T]\to L^p(\mathbb{T}^d) \text{ such that $f$ is continuous in $L^p_{\mathrm{weak-}*}$}\right\} \subset L^\infty L^p.
	\end{equation*}
	For $p \in (1,\infty)$, $L^p(\mathbb{T}^d)$ is reflexive and so these two definitions coincide. However, we commonly take $p=\infty$. Hence, we choose the notation $C_{\mathrm{weak-}*}^0L^p$ and take care of the case when $p=1$.
	\vspace{\baselineskip}
	
	All results will be stated and proved for the compact domains $\mathbb{T}^d$ and $[0,T]$. The exact statements do not necessarily hold when we replace $\mathbb{T}^d$ with $\mathbb{R}^d$, or $[0,T]$ with $[0,\infty)$, though analogues can indeed be found.
	\begin{definition}[Weak solution to the transport equation \eqref{eqTE}]\label{TE}
		Consider a vector field $u\in L^1\left([0,T];L^1(\mathbb{T}^d;\mathbb{R}^d)\right)$ with $\nabla\cdot u=0$ in the distributional sense.
		
		We say $f\in L^1\left([0,T];L^1(\mathbb{T}^d;\mathbb{R})\right)$ with $uf \in L^1L^1$ is a weak solution to the transport equation along $u$
		\begin{equation}\label{eqTE}
			\frac{\partial f}{\partial t} + \nabla \cdot (u f) = 0, \tag{TE}
		\end{equation}
		with initial data $f_0 \in L^1(\mathbb{T}^d)$ if, for any $\phi \in C_c^\infty(\mathbb{T}^d\times[0,T))$,
		\begin{equation*}
			\int_{\mathbb{T}^d\times[0,T)} f \left(\frac{\partial \phi}{\partial t} + u\cdot\nabla \phi\right) \; dxdt = - \int_{\mathbb{T}^d} f_0 \phi_0 \; dx,
		\end{equation*}
		where $\phi_0(x) = \phi(x, 0)$.
		
		Meanwhile, we say the transport equation is satisfied on an \textit{open} interval $I\subset(0,T)$ if, for any $\phi \in C_c^\infty(\mathbb{T}^d\times I)$,
		\begin{equation*}
			\int_{\mathbb{T}^d\times I} f \left(\frac{\partial \phi}{\partial t} + u\cdot\nabla \phi\right) \; dxdt = 0.
		\end{equation*}
	\end{definition}
	
	\begin{remark}
		In the above definition, we may equivalently take the test function $\phi$ to be Lipschitz in time and space with compact support in $\mathbb{T}^d \times [0,T)$. This is done by finding a sequence of smooth functions $\phi_n \in C_c^\infty(\mathbb{T}^d\times[0,T))$ such that $\phi_n$ are bounded in $C^1$ by the Lipschitz norm of $\phi$, and $\frac{\partial \phi_n}{\partial t} \xrightarrow{n\to\infty} \frac{\partial \phi}{\partial t}$, $\nabla \phi_n \xrightarrow{n\to\infty} \nabla \phi$ pointwise almost everywhere.
	\end{remark}
	
	\begin{definition}[Renormalised weak solutions to \eqref{eqTE}]\label{renormalised}
		Following the definition introduced in \cite{diperna1989ordinary}, suppose $f$ is a weak solution to \eqref{eqTE} along $u$ with initial data $f_0$.
		
		If, for any $\beta \in C_b^0(\mathbb{R})$, $\beta(f)$ is a weak solution to \eqref{eqTE} along $u$ with initial data $\beta(f_0)$, then we say $f$ is a renormalised weak solution of \eqref{eqTE}.
	\end{definition}
	
	\begin{remark}
		This definition is well motivated by the expression $\left( \frac{\partial}{\partial t} + u \cdot \nabla \right) \beta(f) = \beta'(f)\left( \frac{\partial}{\partial t} + u \cdot \nabla \right) f$ when $\beta$ is a differentiable function. Indeed, some authors require $\beta \in C^1(\mathbb{R})$ as in \cite{Ambrosio2014}, or even with decay at infinity as in \cite{diperna1989ordinary}. In our case ($\nabla \cdot u = 0$), it is straightforward to show that these give equivalent definitions.
	\end{remark}
	
	\begin{definition}[Lagrangian flows and Lagrangian solutions to \eqref{eqTE}]\label{lagrangian}
		For a divergence-free vector field $u \in L^1([0,T];L^1(\mathbb{T}^d;\mathbb{R}^d))$, we say a family ($t \in [0,T])$ of Lebesgue-measure preserving bijections $y_t : \mathbb{T}^d \to \mathbb{T}^d$ is a Lagrangian flow along $u$ if for a.e. $x\in \mathbb{T}^d$ the map $t \mapsto y_t(x)$ is absolutely continuous and the derivative satisfies $\frac{d y_t(x)}{dt} = u(y_t(x),t)$ as a function class in $L^1([0,T];L^1(\mathbb{T}^d;\mathbb{R}^d))$.
		
		We say that a function $f \in L^1([0,T]; L^1(\mathbb{T}^d;\mathbb{R}))$ is a Lagrangian solution to \eqref{eqTE} along $u$ (with initial data $f_0 \circ y_0^{-1}$) if $f(\cdot, t) = f_0 \circ y_t^{-1}$ for $f_0\in L^1(\mathbb{T}^d)$, and $\{y_t\}_{t\in[0,T]}$ a Lagrangian flow along $u$. It is often convenient to take $y_0=\mathrm{Id}$ without loss of generality.
	\end{definition}
	
	\begin{remark}\label{lagrenorm}
		One may show that, if additionally $uf \in L^1 L^1$, then a Lagrangian solution to \eqref{eqTE} along $u$ with initial data $f_0$ is also a weak solution in the sense of Definition \ref{TE}. We reduce the problem to bounded $f$ and $f_0$ by the point-wise approximation with $f 1_{|f|\le k}$, so that $u f 1_{|f|\le k} \xrightarrow{k\to\infty} uf$ in $L^1L^1$ by dominated convergence. Observe then that $f 1_{|f|\le k}$ are already Lagrangian solutions for the initial data $f_0 1_{|f_0|\le k}$, since $(f_0 \circ y_t^{-1}) 1_{|f_0 \circ y_t^{-1}| \le k} = (f_0 1_{|f_0|\le k}) \circ y_t^{-1}$. The result for bounded $f$ and $f_0$ follows from changing variables in the integral \[\int_{\mathbb{T}^d\times[0,T)} f_0(y_t^{-1}(x)) \left(\frac{\partial \phi}{\partial t} + u\cdot\nabla \phi\right) \; dxdt,\] and using that $\left(\frac{\partial \phi}{\partial t} + u\cdot\nabla \phi\right)(y_t(x)) = \frac{\partial}{\partial t}(\phi(y_t(x), t))$ by chain rule for absolutely continuous functions. Moreover, since for any $\beta \in C_b^0(\mathbb{R})$ we can rewrite $\beta(f_0 \circ y_t^{-1}) = \beta(f_0) \circ y_t^{-1}$, these solutions are then also renormalised weak solutions in the sense of Definition \ref{renormalised}.
	\end{remark}
	
	\begin{definition}[Weak solution to the advection-diffusion equation \eqref{eqADE}]\label{ADE}
		Consider a vector field $u\in L^1\left([0,T];L^1(\mathbb{T}^d;\mathbb{R}^d)\right)$ with $\nabla\cdot u=0$ in the distributional sense, and some positive constant viscosity $\nu > 0$ (also called diffusivity).
		
		We say $f\in L^1\left([0,T];L^1(\mathbb{T}^d;\mathbb{R})\right)$ with $uf \in L^1L^1$ is a weak solution to the advection-diffusion equation along $u$
		\begin{equation}\label{eqADE}
			\frac{\partial f}{\partial t} + \nabla \cdot (u f) - \nu \Delta f = 0 \tag{$\nu$-ADE},
		\end{equation}
		with initial data $f_0 \in L^1(\mathbb{T}^d)$, if for any $\phi \in C_c^\infty(\mathbb{T}^d\times[0,T))$,
		\begin{equation*}
			\int_{\mathbb{T}^d\times[0,T]} f \left(\frac{\partial \phi}{\partial t} + u\cdot\nabla \phi + \nu \Delta \phi\right) \; dxdt = - \int_{\mathbb{T}^d} f_0 \phi_0 \; dx,
		\end{equation*}
		where $\phi_0(x) = \phi(x,0)$.
		
		Meanwhile, we say the advection-diffusion equation \eqref{eqADE} is satisfied on an \textit{open} interval $I\subset(0,T)$ if, for any $\phi \in C_c^\infty(\mathbb{T}^d\times I)$,
		\begin{equation*}
			\int_{\mathbb{T}^d\times I} f \left(\frac{\partial \phi}{\partial t} + u\cdot\nabla \phi + \nu\Delta\phi \right) \; dxdt = 0.
		\end{equation*}
	\end{definition}
	
	\begin{definition}[Vanishing viscosity solution to \eqref{eqTE}]\label{vvADE}
		
		Consider a vector field $u\in L^1\left([0,T];L^1(\mathbb{T}^d;\mathbb{R}^d)\right)$ with $\nabla\cdot u=0$ in the distributional sense.
		
		We say a weak solution $f$ to \eqref{eqTE} along $u$ with initial data $f_0 \in L^1(\mathbb{T}^d)$ is a vanishing viscosity solution if there exists a positive sequence $\nu_n\to0$ and corresponding weak solutions $f^{(n)}$ to \eqref{eqADE} with viscosity $\nu_n$ along $u$ with initial data $f_0$, such that
		\begin{equation*}
			f^{(n)} \xrightharpoonup{n \to \infty} f,
		\end{equation*}
		converges as distributions in $\mathcal{D}'(\mathbb{T}^d\times[0,T))$.
	\end{definition}
	
	\begin{remark}
		In the literature $f^{(n)}$ is often taken to have, in addition, non-constant initial data $f_0^{(n)}$, where $f_0^{(n)}$ converges to $f_0$ in a suitable topology. We do not consider such a more general definition, but we note that under strong convergence of $f_0^{(n)}\to f_0$, the appropriate analogue of our main results, Theorem \ref{vvtheorem} and Theorem \ref{perfmixvv}, follow by comparing between the viscous solutions with initial data $f^{(n)}_0$ and $f_0$. This can be done by equality \eqref{eqlpineq} in Theorem \ref{ADEwellposed}.
	\end{remark}
	
	\section{Necessary background}\label{Background}
	The following is by no means a complete exposition of standard existence, uniqueness, and regularity theory for \eqref{eqTE} and \eqref{eqADE}, but contains some of the more salient points, and in particular those relevant to this paper.
	
	The exact statements are adapted for this paper, with similar results, and similar methods of proof, found in the literature.
	
	We fix a finite time interval $[0,T]$ throughout.
	
	\begin{theorem}[Weak continuity of transport]\label{weakcont}
		Suppose $f$ is a weak solution to \eqref{eqADE} or \eqref{eqTE} (i.e. $\nu=0$) along $u \in L^1 L^1$ with initial data $f_0$.
		
		Then for any $\phi\in C^\infty(\mathbb{T}^d\times[0,T])$, for a.e. $t \in [0,T]$,
		\begin{multline}
			\text{(Trace Formula)\quad} \int_{\mathbb{T}^d} f(\cdot, t) \phi(\cdot, t) \;dx \\ = \int_{\mathbb{T}^d} f_0 \phi_0 \; dx + \int_{\mathbb{T}^d\times[0,t]} f\left(\frac{\partial\phi}{\partial t}+u\cdot\nabla\phi +\nu\Delta \phi\right) dxdt, \label{eqtrace}
		\end{multline}
		
		Suppose further that $f \in L^\infty L^p$ for $p \in (1, \infty]$, then (there is a representation of $f(x,t)$ with) $f \in C_{\mathrm{weak-}*}^0L^p$, such that \eqref{eqtrace} holds for all $t \in [0,T]$.
		
		In particular $f(\cdot, 0) = f_0$ in $L^p(\mathbb{T}^d)$.
	\end{theorem}
	\begin{remark}
		We remark without proof that the analogous result holds when $p=1$ if we additionally assume $f(\cdot,t): [0,T] \to L^1(\mathbb{T}^d)$ is uniformly integrable in the indexing variable $t$ (for a.e. $t\in[0,T]$).
	\end{remark}
	\begin{proof}
		For any weak solution of \eqref{eqTE} or \eqref{eqADE}, we have $f, uf \in L^1L^1$.
		
		Fix some $\phi \in C^\infty\left(\mathbb{T}^d \times [0,T]\right)$ and consider the following $L^1([0,T])$ function of $t \in [0,T]$,
		\begin{equation}\label{eq23}
			\int_{\mathbb{T}^d} f(\cdot, t) \phi(\cdot, t) \; dx.
		\end{equation}
		
		Then for all $\psi \in C_c^\infty([0,T))$, by Definitions \ref{TE}, \ref{ADE},
		\begin{multline*}
			\int_{\mathbb{T}^d \times [0,T)} f \phi \frac{d\psi}{ds} \; dxds \\
			\begin{aligned}[t]
				& = \int_{\mathbb{T}^d \times [0,T)} f \phi \frac{d\psi}{ds} \; dxds - \left(\int_{\mathbb{T}^d} f_0 \phi_0 \psi_0 \; dx + \int_{\mathbb{T}^d\times[0,T)} f\left(\frac{\partial}{\partial s}+u\cdot\nabla + \nu \Delta \right)(\phi \psi(s)) \; dxds \right) \\
				& = - \int_{\mathbb{T}^d} f_0 \phi_0 \psi_0 \; dx - \int_{\mathbb{T}^d\times[0,T)} f\left(\frac{\partial\phi}{\partial s}+u\cdot\nabla\phi + \nu \Delta \phi\right) \psi(s) \; dxds,
			\end{aligned}
		\end{multline*}
		where $\psi_0=\psi(0)$ and $\phi_0(x)=\phi(x,0)$.
		
		And so the function defined in \eqref{eq23} is an absolutely continuous function of $t \in [0,T]$, with derivative $\int_{\mathbb{T}^d} f(x, t)\left(\frac{\partial\phi}{\partial s}+u\cdot\nabla\phi + \nu \Delta \phi\right)(x, t) \; dx$, and the initial value at $t=0$ is $\int_{\mathbb{T}^d} f_0 \phi_0 \; dx$, see Chapter 3, Lemma 1.1 in \cite{temam1979}. That is \eqref{eqtrace} holds for a.e. $t \in [0,T]$.
		
		Suppose now, for some $p \in (1,\infty]$, that $f \in L^\infty L^p$. Define for each $t \in [0,T]$ the distribution $F_t\in\mathcal{D}'(\mathbb{T}^d)$ acting on test functions $\chi \in C^\infty(\mathbb{T}^d)$, given by
		\begin{equation*}
			\left\langle F_t,\chi\right\rangle = \int_{\mathbb{T}^d} f_0 \chi \; dx + \int_{\mathbb{T}^d\times[0,t]} f\left(u\cdot\nabla\chi +\nu\Delta \chi\right) dxdt.
		\end{equation*}
		
		Then, thanks to \eqref{eqtrace}, for a.e. $t \in [0,T]$, $F_t = f(\cdot, t)$ as distributions in $\mathcal{D}'(\mathbb{T}^d)$, which is assumed uniformly bounded in $L^p(\mathbb{T}^d)$. We therefore have for a.e. $t \in [0,T]$, for all $\chi \in C^\infty(\mathbb{T}^d)$, the bound $|F_t(\chi)| \le \left\|f\right\|_{L^\infty L^p}\left\|\chi\right\|_{L^q}$ with $\frac{1}{p}+\frac{1}{q}=1$. By the absolute continuity of $t\mapsto \left\langle F_t,\chi\right\rangle$, the bound, in fact, holds for all $t \in [0,T]$. Since $p\in(1,\infty]$ this implies further that $F_t$ can be extended to a function $\bar{f}(\cdot, t) \in L^p(\mathbb{T}^d)$ for all $t \in [0,T]$, and the absolute continuity of $\left\langle F_t, \chi \right\rangle$ implies $\bar{f} \in C_{\mathrm{weak-}*}^0L^p$ with \eqref{eqtrace} holding for the representative $\bar{f}$ now for all $t \in [0,T]$, as required.
	\end{proof}
	
	\begin{remark}
		An important application of this is that for $t \in [0, T]$, and any $\phi \in C_c^\infty\left(\mathbb{T}^d \times [t,T)\right)$,
		\begin{equation*}
			\int_{\mathbb{T}^d\times[t,T]} f\left(\frac{\partial\phi}{\partial t}+u\cdot\nabla\phi +\nu\Delta \phi\right) dxdt = - \int_{\mathbb{T}^d} f(\cdot, t) \phi(\cdot, t) \;dx.
		\end{equation*}
		So if $f\in C_{\mathrm{weak-}*}^0L^p$ is a weak solution to \eqref{eqTE}/\eqref{eqADE} along $u$ with initial data $f_0$, then it is also a weak solution to \eqref{eqTE}/\eqref{eqADE} along $u$ on the time interval $[t, T]$ with initial data $f(\cdot, t)$.
		
		This in particular allows us to say that $f\in C_{\mathrm{weak-}*}^0L^p$ is a weak solution to \eqref{eqTE}/\eqref{eqADE} with initial data $f_0$ if and only if it is a weak solution to \eqref{eqTE}/\eqref{eqADE} on $[0,t]$ with initial data $f_0$, and on $[t, T]$ with initial data $f(\cdot, t)$. That is, we have transitivity in the transport equation.
	\end{remark}
	
	\begin{theorem}[Existence for \eqref{eqTE} and \eqref{eqADE}]\label{TEexistence}
		Suppose $f_0 \in L^\infty(\mathbb{T}^d)$, and $u\in L^1([0,T];L^1(\mathbb{T}^d;\mathbb{R}^d))$ is divergence-free in the distributional sense, then there exists a weak solution $f \in C_{\mathrm{weak-}*}^0 L^\infty$ to \eqref{eqTE}/\eqref{eqADE} along $u$ with initial data $f_0$. Moreover, for all $p \in [1,\infty]$, $t \in [0,T]$,
		\begin{gather}
			\text{(Initial $L^p$-Inequality)\quad} \left\|f(\cdot, t)\right\|_{L^p(\mathbb{T}^d)} \le \left\|f_0\right\|_{L^p} \label{eqinitineq}, \\
			\text{(Conservation of Mass)\quad} \int_{\mathbb{T}^d} f(\cdot, t) \; dx = \int_{\mathbb{T}^d} f_0 \; dx \label{eqconsmass}.
		\end{gather}
	\end{theorem}
	\begin{proof}
		More generally, for $f_0 \in L^q(\mathbb{T}^d)$, $u \in L^1([0,T];L^r(\mathbb{T}^d;\mathbb{R}^d))$ with $\frac{1}{q}+\frac{1}{r}=1$, existence of a weak solution $f$ with $\left\|f\right\|_{L^\infty L^p} \le \left\|f_0\right\|_{L^p}$ (for all $p\in[1,\infty]$, permitting infinite values of the norms) follows from a standard approximation scheme (regularisation of $u$ and $f_0$), see \cite{diperna1989ordinary} for details; though the proof considers only $\nu = 0$ (i.e. \eqref{eqTE}), and the spatial domain $\mathbb{R}^d$ instead of $\mathbb{T}^d$, an identical argument goes through here.
		
		By assumption $u \in L^1L^1$, and so we may apply the above result with $q=\infty$.
		
		By Theorem \ref{weakcont} the solution is in $C_{\mathrm{weak-}*}^0 L^\infty$, and so the bound $\left\|f\right\|_{L^\infty L^p} \le \left\|f_0\right\|_{L^p}$ implies the Initial $L^p$-Inequality \eqref{eqinitineq}.
		
		Conservation of Mass \eqref{eqconsmass} follows from the Trace Formula \eqref{eqtrace} in Theorem \ref{weakcont} with $\phi \equiv 1$ on $\mathbb{T}^d \times [0,T]$.
	\end{proof}
	
	\begin{remark}
		In fact, Conservation of Mass \eqref{eqconsmass} will hold for a.e. $t \in [0,T]$ for every weak solution on the torus, and not only those constructed in Theorem \ref{TEexistence}. This does not hold in $\mathbb{R}^d$.
	\end{remark}
	
	\begin{remark}
		Though a significant topic of interest, we do not mention further regularity results for the transport equation, as they require further assumptions on $u$. Instead, we point to \cite{Ambrosio2014} for an exposition of the Cauchy-Lipschitz theory, and other more standard results, including the DiPerna-Lions theory of renormalised weak solutions (introduced in \cite{diperna1989ordinary}), and Ambrosio's extension of this well-posedness class to $u \in L^1([0,T];BV)$ (originally in \cite{ambrosio2004transport}).
	\end{remark}
	
	Next, we give the following well-posedness and regularity result for \eqref{eqADE}. To illustrate the stark contrast between \eqref{eqTE} and its regularisation \eqref{eqADE} we show well-posedness for any $f_0 \in L^1(\mathbb{T}^d)$. When, in addition, $f_0$ is more integrable, we may obtain further regularity. Therefore, when stating our main results, we shall later assume $f_0\in L^\infty$.
	
	\begin{theorem}[Well-posedness of \eqref{eqADE}]\label{ADEwellposed}
		Suppose $u\in L^\infty L^\infty$, then for any initial data $f_0 \in L^1(\mathbb{T}^d)$ any weak solution (in the class $L^1 L^1$) to \eqref{eqADE} along $u$ with initial data $f_0$ is unique. Moreover this solution exists, $f\in C^0L^1$ with $\left\|f\right\|_{C^0L^1} \le \left\|f_0\right\|_{L^1(\mathbb{T}^d)}$ and becomes immediately bounded, $f \in C^0([\epsilon,T];C^0(\mathbb{T}^d))$ for all $\epsilon\in(0,T]$.
		
		Furthermore, when $f_0 \in L^\infty$ we have the following additional regularity (for all $p\in[1,\infty)$)
		\begin{gather}
			f \in (C^0 L^p) \cap (L^2 H^1), \label{eqADEcont} \\
			f \in C_{\mathrm{weak-}*}^0L^\infty, \label{eqADEcont2} \\
			\text{($L^p$-Inequality)\quad} 0 \le s \le t \implies \left\|f(\cdot, t)\right\|_{L^p(\mathbb{T}^d)} \le \left\|f(\cdot, s)\right\|_{L^p(\mathbb{T}^d)}, \label{eqlpineq} \\
			\text{(Energy Identity)\quad} \int_{\mathbb{T}^d} \left|f(x, t)\right|^2 \;dx + 2\nu \int_{\mathbb{T}^d \times [0,t]} \left|\nabla f(x,s)\right|^2\;dxds = \int_{\mathbb{T}^d} \left|f_0(x)\right|^2\;dx, \label{eqenergy} \\
			\text{(Equicontinuity)\quad} \left\|\frac{\partial f}{\partial t}\right\|_{L^2H^{-1}} \le \left(\left\|u\right\|_{L^2 L^\infty} + \sqrt{\frac{\nu}{2}}\right)\left\|f_0\right\|_{L^2}. \label{eqequicont}
		\end{gather}
	\end{theorem}
	\begin{proof}
		The standard well-posedness result is for $f_0\in L^2$, and is done via energy estimates, see for example Chapter 7 in \cite{evans2010partial}. However, if we wish to obtain uniqueness in the class $L^1 L^1$, we must be more careful. Given by the spatial Fourier transform on $\mathbb{R}^d$, $\mathcal{F}(K_\nu(\cdot,t))(\xi) = e^{-\nu\xi^2 t}$, denote by $K_\nu \in C^\infty(\mathbb{R}^d\times(0,\infty))\cap L^\infty((0,\infty); L^1(\mathbb{R}^d))$ the heat kernel for the heat equation with diffusivity $\nu$, that is for any $\phi \in C_c^\infty(\mathbb{T}^d \times (-\infty, \infty))$ we have
		\begin{equation*}
			\left(\frac{\partial}{\partial t} - \nu\Delta\right)(\phi *_{x,t} (K_\nu 1_{t>0})) = \phi,
		\end{equation*}
		where $*_{x,t}$ denotes convolution is space and time. Denoting by $\bar{K}_\nu \in C^\infty(\mathbb{R}^d\times(-\infty,0))$ the backwards heat kernel $\bar{K}_\nu(x,t) = K_\nu(-x,-t)$, then we have for the backwards heat equation, for any $\phi \in C_c^\infty(\mathbb{T}^d \times (-\infty, \infty))$
		\begin{equation*}
			\left(\frac{\partial}{\partial t} + \nu\Delta\right)(\phi *_{x,t} (\bar{K}_\nu 1_{t<0})) = -\phi.
		\end{equation*}
		
		For $\phi \in C_c^\infty(\mathbb{T}^d\times[0,T))$ we may take $\phi *_{x,t} (\bar{K}_\nu 1_{t<0})$ as a test function in \eqref{eqADE}, which (by expanding out all convolutions) can be rewritten as
		\begin{align*}
			\int_{\mathbb{T}^d\times[0,T]} f\phi \; dx dt & = \int_{\mathbb{T}^d} f_0(x) \left(\phi *_{x,t} (\bar{K}_\nu 1_{t<0})\right)(x,0) \; dx + \int_{\mathbb{T}^d\times[0,T]} fu\cdot \left(\phi *_{x,t} (\nabla\bar{K}_\nu 1_{t<0})\right) \; dx dt \\
			& = \int_{\mathbb{T}^d} \phi(x,t) \left(f_0 *_x (K_\nu(\cdot, t)) - (fu 1_{t \in [0,T]}) *_{x,t} (\nabla K_\nu 1_{t>0})\right)(x,t) \; dx dt,
		\end{align*}
		where $*_x$ denotes convolution over the spatial variable $x \in \mathbb{T}^d$ only.
		
		We have shown, indeed for any weak solution $f$ to \eqref{eqADE} along $u$ with initial data $f_0 \in L^1(\mathbb{T}^d)$,
		\begin{equation}
			f = f_0 *_x K_\nu(\cdot, t) - (fu1_{t\in[0,T]}) *_{x,t} (\nabla K_\nu 1_{t>0}). \label{eqduhamel}
		\end{equation}
		
		When $f_0 = 0$, we have by Young's convolution inequality,
		\begin{equation*}
			\left\|f1_{t\in[0,\epsilon]}\right\|_{L^1 L^1} \le \left\|f1_{t\in[0,\epsilon]}\right\|_{L^1 L^1} \left\|u\right\|_{L^\infty L^\infty} \left\| \nabla K_\nu 1_{t\in(0,\epsilon]}\right\|_{L^1 L^1}.
		\end{equation*}
		
		It is straightforward to check that $\nabla K_\nu 1_{t \in (0,T]} \in L^1 L^1$, and so for $\epsilon$ small enough (depending only on $\left\|u\right\|_{L^\infty L^\infty}$ and $\nu$) the above implies that $f1_{t\in[0,\epsilon]} = 0$. Repeating the argument then shows that for all $n \in \mathbb{N}$, $f1_{t\in[0,n\epsilon]}=0$, and so indeed $f=0$, proving uniqueness.
		
		Existence of a weak solution $f \in L^\infty L^1$ with $\left\|f\right\|_{L^\infty L^1} \le \left\|f_0\right\|_{L^1}$ follows by solving the equation for mollified $u$ and $f_0$ as in Theorem \ref{TEexistence}. It can be checked that this produces a sequence of smooth functions $f_n$ converging to the solution $f$ weakly, and moreover, by the argument in \cite{Carlen1995}, with an a priori bound on $\left\|f_n\right\|_{L^\infty([\epsilon,T];L^\infty(\mathbb{T}^d))}$ depending only on $\left\|f_0\right\|_{L^1(\mathbb{T}^d)}$, $\left\|u\right\|_{L^1L^\infty}$, $\epsilon > 0$, and $\nu$. We therefore have $f \in L^\infty L^1$ and $f\in L^\infty([\epsilon,T];L^\infty(\mathbb{T}^d))$ for any $\epsilon>0$. Though we do not make use of it in the remainder of the paper, the improved regularity $f \in C^0 L^1$, and $f \in C^0([\epsilon,T];C^0(\mathbb{T}^d))$ for all $\epsilon\in(0,T]$, follow from the formula \eqref{eqduhamel} and the regularity of $K_\nu$.
		
		When in addition $f_0 \in L^2(\mathbb{T}^d)$ the regularised sequence $f_n$ can further be shown to converge in $(C^0L^2) \cap (L^2H^1)$. Statements \eqref{eqlpineq} and \eqref{eqenergy} then follow from their counterparts for smooth $f_0$ and $u$. If $f_0 \in L^\infty(\mathbb{T}^d)$ we have, say by \eqref{eqlpineq}, $f \in L^\infty L^\infty$, and hence by Theorem \ref{weakcont} we prove \eqref{eqADEcont2}. The continuity $f \in C^0L^p$ then follows from $f\in C^0L^2$, for $p\in [1,2)$ by compactness of $\mathbb{T}^d$, and for $p\in(2,\infty)$ by interpolation with $f\in L^\infty L^\infty$. It remains to show \eqref{eqequicont}. From \eqref{eqenergy} we have the bounds $\left\|f\right\|_{L^\infty L^2}, \sqrt{2\nu}\left\|\nabla f\right\|_{L^2L^2} \le \left\|f_0\right\|_{L^2}$, and hence for any $\phi \in C_c^\infty(\mathbb{T}^d\times(0,T))$, by the equation \eqref{eqADE},
		\begin{align*}
			\left|\int_{\mathbb{T}^d\times(0,T)}f\frac{\partial \phi}{\partial t} \; dxdt\right| & = \left|\int_{\mathbb{T}^d\times(0,T)}f \left(u\cdot\nabla\phi + \nu \Delta \phi\right) \; dxdt\right| \\
			& \le \left\|f\right\|_{L^\infty L^2}\left\|u\right\|_{L^2L^\infty}\left\|\nabla\phi\right\|_{L^2L^2} + \nu\left\|\nabla f\right\|_{L^2L^2}\left\|\nabla\phi\right\|_{L^2L^2} \\
			& \le \left\|f_0\right\|_{L^2}\left\|u\right\|_{L^2L^\infty}\left\|\nabla\phi\right\|_{L^2L^2} + \sqrt{\frac{\nu}{2}}\left\|f_0\right\|_{L^2}\left\|\nabla\phi\right\|_{L^2L^2},
		\end{align*}
		which proves \eqref{eqequicont}.
	\end{proof}
	
	\begin{remark}
		It is likely possible to extend the first part of Theorem \ref{ADEwellposed} to the Prodi-Serrin class \cite{Prodi1959, Serrin1962}, $u \in L^pL^q$, $\frac{2}{p} + \frac{d}{q} \le 1$ ($d<q$), though we do not attempt to prove such a result here as it is immaterial to this paper.
		
		In fact, such a result is already available in the stochastic setting due to Flandoli, see \cite{fedrizzi2013noise}.
	\end{remark}
	
	\begin{remark}
		When $f_0 \in L^1$ explicit decay rates of the $L^\infty(\mathbb{T}^d)$ norm (and other $L^p$ norms) are provable. See \cite{Carlen1995, Maekawa2008} for details when working on the whole space $\mathbb{R}^d$.	For general existence and regularity results concerning similar parabolic PDEs, we point to \cite{LeBris2019} and also Chapter 7 in \cite{evans2010partial}.
	\end{remark}
	
	\section{Vanishing viscosity}\label{vvLimit}
	The primary purpose of this section is to prove Theorem \ref{vcontrol} below, which allows us to construct bounded divergence-free vector fields in a way that permits control of the corresponding vanishing viscosity limit of \eqref{eqADE}. This relies on two Propositions.
	
	The first, Proposition \ref{ADEtoTE} below, gives a general criterion for which the vanishing viscosity limit of \eqref{eqADE} converges strongly to \eqref{eqTE}. That is, for a suitable divergence-free vector field $u:\mathbb{T}^d\times[0,T]\to\mathbb{R}^d$, and small viscosity $\nu>0$, that solutions of \eqref{eqADE} along $u$ are well-approximated by a weak solution of \eqref{eqTE} along $u$. This result is a generalisation of the Selection Theorem in \cite{bonicatto2021advection}.
	
	The second, Proposition \ref{ADEtoADE} below, uses a similar argument to show, for fixed viscosity $\nu>0$, how solutions of \eqref{eqADE} depend little on the small spatial scales of the vector field $u$. We quantify these scales through the weak-$*$ topology of vector fields in $L^\infty([0,T];L^\infty(\mathbb{T}^d;\mathbb{R}^d))$. The intuition is that the viscosity `blurs' these small spatial scales.
	
	The key idea of Theorem \ref{vcontrol} is then that solving \eqref{eqADE} with reduced viscosity $\nu>0$ is akin to adding small spatial scales while solving \eqref{eqTE}.
	
	The advantage of solving \eqref{eqTE} is then that Lagrangian solutions (Definition \ref{lagrangian}) can be designed rather explicitly.
	
	\begin{proposition}\label{ADEtoTE}
		Consider a vector field $u \in L^1\left( [0,T]; L^1(\mathbb{T}^d;\mathbb{R}^d)\right)$ with $\nabla\cdot u=0$ in the distributional sense. Fix some initial data $f_0 \in L^\infty(\mathbb{T}^d)$.
		
		Suppose that there is a unique weak solution $f$ (in the class $L^\infty L^\infty$) to \eqref{eqTE} along $u$ with initial data $f_0$, and that additionally $f$ is a renormalised weak solution (Definition \ref{renormalised}).
		
		For each $\nu > 0$, denote by $f^\nu$ any weak solution to \eqref{eqADE} along $u$ with initial data $f_0$. Suppose in addition that $f^\nu \in C_{\mathrm{weak-}*}^0L^\infty$ and satisfies the Initial $L^p$-Inequality \eqref{eqinitineq} for all $p \in [1,\infty]$. (This would be the case if $u \in L^\infty L^\infty$, see Theorem \ref{ADEwellposed}.)
		
		Then, for each $p \in [1,\infty)$, $f^\nu\xrightarrow{\nu \to 0}f$ converges strongly in $L^p L^p$ and also in weak-$*$ $L^\infty L^\infty$.
		
		If additionally $f \in C^0 L^1$, then for each $p \in [1,\infty)$, $f^\nu\xrightarrow{\nu \to 0}f$ also converges strongly in $L^\infty L^p$.
	\end{proposition}
	\begin{proof}
		Suppose that $f^\nu$ does not converge in weak-$*$ $L^\infty L^\infty$ to $f$ as $\nu\to0$. Then there exists some $g \in L^1 L^1$, and a sequence $\nu_i \xrightarrow{i\to\infty}0$ and $c>0$ such that for all $i\in\mathbb{N}$
		\begin{equation}
			\left|\int_{\mathbb{T}^d\times[0,T]}(f^{\nu_i}-f)g\;dxdt\right|\ge c. \label{eqcontr1}
		\end{equation}
		
		By the Initial $L^p$-Inequality \eqref{eqinitineq} $f^{\nu_i}$ is uniformly bounded for all $i\in\mathbb{N}$ in $L^\infty L^\infty$, and so by taking a subsequence if necessary, we may assume $f^{\nu_i}\xrightharpoonup{i\to\infty}\bar{f}$ converges in weak-$*$ $L^\infty L^\infty$ to some $\bar{f} \in L^\infty L^\infty$. Then, for any $\phi \in C_c^\infty(\mathbb{T}^d \times [0,T))$,
		\begin{align*}
			& \int_{\mathbb{T}^d\times[0,T)} \bar{f} \left(\frac{\partial \phi}{\partial t} + u\cdot\nabla \phi\right) \; dxdt \\
			& = \lim_{i\to\infty} \int_{\mathbb{T}^d\times[0,T)} f^{\nu_i} \left(\frac{\partial \phi}{\partial t} + u\cdot\nabla \phi + \nu_i \Delta \phi\right) \; dxdt \\
			& = - \int_{\mathbb{T}^d} f_0 \phi_0 \; dx,
		\end{align*}
		and so the limit $\bar{f}$ is a weak solution to \eqref{eqTE} along $u$ with initial data $f_0$. Moreover, it is in $L^\infty L^\infty$, so by assumption must be the unique weak solution $f$, contradicting \eqref{eqcontr1}. Therefore $f^\nu \xrightarrow{\nu\to0}f$ converges in weak-$*$ $L^\infty L^\infty$.
		
		Recall by Theorem \ref{weakcont} that $f, f^\nu \in C_{\mathrm{weak-}*}^0L^\infty$.
		
		Since by assumption $f^\nu$ satisfies the Initial $L^p$-Inequality \eqref{eqinitineq} for all $p \in [1,\infty]$, we may bound
		\begin{equation}\label{eq10.5}
			\left\|f^\nu\right\|_{L^p L^p} \le T^{\frac{1}{p}}\left\|f_0\right\|_{L^p}.
		\end{equation}
		
		If $p\in(1,\infty]$, weak-$*$ convergence in $L^\infty L^\infty$ implies weak-$*$ convergence in $L^p L^p$. Whenever $p \in (1,\infty)$, weak-$*$ convergence in $L^pL^p$ is also strong in $L^pL^p$ if and only if $\underset{\nu\to0}{\limsup}\left\|f^\nu\right\|_{L^pL^p} \le \left\|f\right\|_{L^pL^p}$. This is a standard result following from uniform convexity of $L^p(\mathbb{T}^d)$ for $p \in (1,\infty)$.
		
		To show that this is satisfied, we use that $f \in C_{\mathrm{weak-}*}^0L^\infty$ is a renormalised weak solution to \eqref{eqTE} (Definition \ref{renormalised}). Denoting by $a\wedge b = \min\left\{a,b\right\}$, let $M \in \mathbb{N}$ and $\beta(x) = M\wedge|x|^p$, then $\beta(f)$ is a weak solution to \eqref{eqTE} along $u$ with initial data $\beta(f_0) \in L^\infty$. Taking $\phi \equiv 1$ in the Trace Formula \eqref{eqtrace} shows that there exists a subset $E_M\subset[0,T]$ with zero Lebesgue-measure, such that for all $t \in [0,T]\setminus E_M$ we have,
		\begin{equation*}
			\int_{\mathbb{T}^d} M \wedge |f(\cdot, t)|^p \; dx = \int_{\mathbb{T}^d} M \wedge |f_0|^p \; dx.
		\end{equation*}
		
		In particular, the above holds for all $t \in [0,T]\setminus \bigcup_{M\in\mathbb{N}}E_M$. By the Lebesgue monotone convergence Theorem, taking $M\to\infty$ shows that $\left\|f(\cdot, t)\right\|_{L^p(\mathbb{T}^d)} = \left\|f_0\right\|_{L^p}$ for all $t \in [0,T]\setminus \bigcup_{M\in\mathbb{N}}E_M$, which implies $\left\|f\right\|_{L^p L^p} = T^{\frac{1}{p}}\left\|f_0\right\|_{L^p}$. Combined with \eqref{eq10.5} this implies that $\underset{\nu\to0}{\limsup}\left\|f^\nu\right\|_{L^pL^p} \le \left\|f\right\|_{L^pL^p}$, and hence $f^\nu \xrightarrow{\nu\to0} f$ in $L^pL^p$ as required.
		
		Convergence of $f^\nu \xrightarrow{\nu\to0} f$ in $L^1 L^1$ follows from convergence in $L^p L^p$ for any $p \in (1,\infty)$ and the compactness of the domain $\mathbb{T}^d \times [0,T]$. \\
		\linebreak
		We now assume that $f \in C^0 L^1$ and wish to upgrade to convergence in $L^\infty L^p$ for each $p \in [1,\infty)$. The idea is to use the Trace Formula \eqref{eqtrace} to show for each $t \in [0,T]$ that $f^\nu(\cdot, t) \xrightharpoonup{\nu\to0} f(\cdot, t)$ converges in weak-$*$ $L^\infty$, and then upgrade to (uniform in time) strong convergence in $L^p(\mathbb{T}^d)$ by convergence of the norm $\left\|f^\nu(\cdot,t)\right\|_{L^p(\mathbb{T}^d)}$ for $p \in (1,\infty)$.
		
		Since the $L^2(\mathbb{T}^d)$-inner product makes life easier, we will only prove convergence in $L^\infty L^2$, and notice that convergence in $L^\infty L^p$ follows for $p \in [1,2)$ by compactness of $\mathbb{T}^d$, and for $p\in(2,\infty)$ by interpolation with the existing uniform bound in $L^\infty L^\infty$. Though we do not elaborate on it, convergence in $L^\infty L^p$ for $p \in (1,\infty)$ can also be shown directly, if say $f_0 \notin L^2(\mathbb{T}^d)$.
		
		We have already shown that $f^\nu \xrightarrow{\nu\to0} f$ in $L^1L^1$, with uniform bound in $L^\infty L^\infty$. Suppose that $uf^\nu$ does not converge strongly in $L^1 L^1$ to $uf$ as $\nu\to0$. Then there exists a sequence $\nu_i \xrightarrow{i\to\infty}0$ and $c>0$ such that for all $i\in\mathbb{N}$
		\begin{equation}
			\left\|uf^{\nu_i} - uf\right\|_{L^1 L^1} \ge c \label{eqcontr2}.
		\end{equation}
		
		Now $f^{\nu_i} \xrightarrow{i\to\infty} f$ strongly in $L^1 L^1$, and so by taking a further subsequence if necessary, we may assume that $f^{\nu_i} \xrightarrow{i\to\infty} f$ point-wise a.e. in $\mathbb{T}^d\times[0,T]$. But then $uf^{\nu_i} \xrightarrow{i\to\infty} uf$ converge point-wise a.e. in $\mathbb{T}^d\times[0,T]$, and are also uniformly bounded for all $i\in\mathbb{N}$ in $L^1 L^1$ by $\left\|u\right\|_{L^1 L^1}\left\|f_0\right\|_{L^\infty(\mathbb{T}^d)}$, and so the dominated convergence Theorem yields a contradiction to \eqref{eqcontr2}. Therefore the product $uf^\nu \xrightarrow{\nu\to0} uf$ converges strongly in $L^1 L^1$.
		
		We now use that $f \in C^0 L^1$, and therefore $f\in C^0 L^2$ (by interpolation with the existing bound $f \in L^\infty L^\infty$), to take a smooth approximation $\phi_\epsilon \in C^\infty(\mathbb{T}^d\times[0,T])$, such that $\left\|\phi_\epsilon-f\right\|_{L^\infty L^2} \le \epsilon$. Then by the Trace Formula \eqref{eqtrace}, and that $\left\|f^\nu(\cdot, t)\right\|_{L^2(\mathbb{T}^d)} \le \left\|f_0\right\|_{L^2} = \left\|f(\cdot, t)\right\|_{L^2(\mathbb{T}^d)}$ for a.e. $t\in[0,T]$, we may write for a.e. $t \in [0,T]$,
		\begingroup
		\allowdisplaybreaks
		\begin{align*}
			& \left\|f^\nu(\cdot, t) - f(\cdot, t)\right\|_{L^2(\mathbb{T}^d)}^2 \\
			& \qquad = \left\|f^\nu(\cdot, t)\right\|_{L^2(\mathbb{T}^d)}^2 + \left\|f(\cdot, t)\right\|_{L^2(\mathbb{T}^d)}^2 - 2 \int_{\mathbb{T}^d} f^\nu(\cdot, t) f(\cdot, t)\;dx \nonumber \\
			& \qquad \le 2\left\|f(\cdot, t)\right\|_{L^2(\mathbb{T}^d)}^2 + 2\epsilon\left\|f_0\right\|_{L^2(\mathbb{T}^d)} - 2 \int_{\mathbb{T}^d} f^\nu(\cdot, t) \phi_\epsilon(\cdot, t)\;dx \nonumber \\
			& \qquad \le
			\begin{aligned}[t]
				& 2\left\|f(\cdot, t)\right\|_{L^2(\mathbb{T}^d)}^2 + 2\epsilon\left\|f_0\right\|_{L^2(\mathbb{T}^d)} - 2\int_{\mathbb{T}^d}f_0\phi_\epsilon(\cdot, 0)\;dx \\
				& - 2\int_{\mathbb{T}^d \times [0,t]}f^\nu \left(\frac{\partial \phi_\epsilon}{\partial t} + u \cdot \nabla \phi_\epsilon + \nu \Delta \phi_\epsilon \right) \;dxdt
			\end{aligned} \nonumber \\
			& \qquad \le
			\begin{aligned}[t]
				& 2\left\|f(\cdot, t)\right\|_{L^2(\mathbb{T}^d)}^2 + 2\epsilon\left\|f_0\right\|_{L^2(\mathbb{T}^d)} - 2\int_{\mathbb{T}^d}f_0\phi_\epsilon(\cdot, 0)\;dx \\
				& - 2\int_{\mathbb{T}^d \times [0,t]}f \left(\frac{\partial \phi_\epsilon}{\partial t} + u \cdot \nabla \phi_\epsilon\right) \;dxdt \\
				& + 2\left\|f^\nu - f\right\|_{L^1L^1}\left\|\frac{\partial \phi_\epsilon}{\partial t}\right\|_{L^\infty L^\infty} + 2\nu\left\|f^\nu\right\|_{L^1L^1}\left\|\Delta \phi_\epsilon\right\|_{L^\infty L^\infty} \\
				& + 2\left\|uf^\nu - uf\right\|_{L^1L^1}\left\|\nabla \phi_\epsilon\right\|_{L^\infty L^\infty}
			\end{aligned} \nonumber \\
			& \qquad = 
			\begin{aligned}[t]
				& 2\left\|f(\cdot, t)\right\|_{L^2(\mathbb{T}^d)}^2 + 2\epsilon\left\|f_0\right\|_{L^2(\mathbb{T}^d)} - 2\int_{\mathbb{T}^d}f(\cdot, t)\phi_\epsilon(\cdot, t)\;dx \\
				& + 2\left\|f^\nu - f\right\|_{L^1L^1}\left\|\frac{\partial \phi_\epsilon}{\partial t}\right\|_{L^\infty L^\infty} + 2\nu\left\|f^\nu\right\|_{L^1L^1}\left\|\Delta \phi_\epsilon\right\|_{L^\infty L^\infty} \\
				& + 2\left\|uf^\nu - uf\right\|_{L^1L^1}\left\|\nabla \phi_\epsilon\right\|_{L^\infty L^\infty}
			\end{aligned} \nonumber \\
			& \qquad \le 4\epsilon\left\|f_0\right\|_{L^2(\mathbb{T}^d)} + C_\epsilon\left(\left\|f^\nu - f\right\|_{L^1L^1} + \nu\left\|f^\nu\right\|_{L^1L^1} + \left\|uf^\nu - uf\right\|_{L^1L^1}\right),
		\end{align*}
		\endgroup
		with $\epsilon>0$ arbitrary, and $C_\epsilon$ a constant that depends only on $\epsilon$ and $f$, and not on $t \in [0,T]$. Then since $f^\nu \xrightarrow{\nu \to 0} f$, and $uf^\nu \xrightarrow{\nu \to 0} uf$ in $L^1 L^1$ we see that $f^\nu \xrightarrow{\nu \to 0} f$ in $L^\infty L^2$ as required.
	\end{proof}
	
	\begin{proposition}\label{ADEtoADE}
		Consider a sequence of uniformly bounded, divergence-free vector fields $u_n \in L^\infty \left( [0,T];L^\infty(\mathbb{T}^d;\mathbb{R}^d)\right)$, such that $u_n \xrightharpoonup{n\to\infty} u$ converges in weak-$*$ $L^\infty L^\infty$ to some $u \in L^\infty L^\infty$. Fix some initial data $f_0 \in L^\infty(\mathbb{T}^d)$.
		
		For each $n \in\mathbb{N}$, and $\nu > 0$ denote by $f^{n,\nu}$ the unique (by Theorem \ref{ADEwellposed}) weak solution to \eqref{eqADE} along $u_n$ with initial data $f_0$. Similarly, denote by $f^\nu$ the unique weak solution to \eqref{eqADE} along $u$ with initial data $f_0$. Then for each $0<a\le b$, $p \in [1,\infty)$
		\begin{equation*}
			\sup_{\nu\in[a,b]} \left\|f^{n,\nu}-f^\nu\right\|_{(L^2H^1) \cap (L^\infty L^p)} \xrightarrow{n \to \infty} 0.
		\end{equation*}
	\end{proposition}
	\begin{proof}
		The case $p \in (2,\infty)$ follows from the same result for $p=2$ by interpolation with the existing uniform bound on $f^{n,\nu}$, $f^\nu$ in $L^\infty L^\infty$ from the $L^p$-Inequality \eqref{eqlpineq}, while for the case $p \in (1,2)$ it follows from the case for $p=2$ by compactness of $\mathbb{T}^d$. Therefore, we may restrict to $p=2$.
		
		Assume to the contrary that for some $0<a\le b$ that there exists $c > 0$ and sequences $\{n_i\}_{i \in \mathbb{N}}$, $\{\nu_i\}_{i \in \mathbb{N}}$ with $n_i \in \mathbb{N}$ increasing, $a\le\nu_i \le b$, such that for all $i \in\mathbb{N}$
		\begin{equation}\label{eqcontr3}
			\left\|f^{n_i,\nu_i} - f^{\nu_i}\right\|_{L^2H^1} + \left\|f^{n_i,\nu_i} - f^{\nu_i}\right\|_{L^\infty L^2} \ge c.
		\end{equation}
		
		By taking a subsequence if necessary we may assume $\nu_i \xrightarrow{i\to\infty} \nu$ for some $\nu\in[a,b]$.
		
		Since $\nu_i$ are bounded above and below, by Energy Identity \eqref{eqenergy} $f^{n_i, \nu_i}$, $f^{\nu_i}$ are uniformly bounded for all $i\in\mathbb{N}$ in $L^2 H^1$, and by Equicontinuity \eqref{eqequicont} $\frac{\partial}{\partial t}f^{n_i, \nu_i}$, $\frac{\partial}{\partial t}f^{\nu_i}$ are uniformly bounded for all $i\in\mathbb{N}$ in $L^2 H^{-1}$.
		
		Since $H^1 \Subset L^2 \subset H^{-1}$ we may apply the Aubin-Lions compactness Lemma (see Chapter 3, Theorem 2.1 in \cite{temam1979}) to deduce that the set $\{f \in L^2H^1 : \left\|\frac{\partial f}{\partial t}\right\|_{L^2 H^{-1}} \le C\}$ is compactly embedded into $L^2 L^2$.
		
		Hence, again taking a subsequence if necessary, we may assume that both $f^{n_i, \nu_i}$, $f^{\nu_i}$ converge, to some limits, strongly in $L^2L^2$ (and weakly in $L^2H^1$) as $i\to\infty$.
		
		We next show that both $f^{n_i, \nu_i}$, and $f^{\nu_i}$ converge to $f^\nu$ strongly in $L^2H^1$ and $L^\infty L^2$ as $i\to\infty$, contradicting the assumption \eqref{eqcontr3}.
		
		We will only show the required convergence for $f^{n_i, \nu_i} \xrightarrow {i \to \infty} f^{\nu}$, since the same result for $f^{\nu_i}$ follows the same proof by additionally assuming that $u_{n_i}$ are a constant sequence $u_{n_i} = u$ for all $i \in \mathbb{N}$.
		
		Since $u_{n_i} \xrightharpoonup{i \to \infty} u$ are uniformly bounded and converge in weak-$*$ $L^\infty L^\infty$, and as we have already shown that $f^{n_i,\nu_i}$ converges to some $\bar{f}$ strongly in $L^2L^2$ as $i\to\infty$, then the product $f^{n_i, \nu_i}u_{n_i}$ converges to $\bar{f}u$ weakly in $L^2 L^2$ as $i\to\infty$.
		
		Then by the weak formulation of \eqref{eqADE}, for any $\phi \in C_c^\infty(\mathbb{T}^d \times [0,T))$, one has
		\begin{multline*}
			\int_{\mathbb{T}^d\times[0,T]} \bar{f} \left( \frac{\partial \phi}{\partial t} + u\cdot\nabla \phi + \nu \Delta \phi\right) \; dxdt \\
			\begin{aligned}
				& = \lim_{i\to\infty} \int_{\mathbb{T}^d\times[0,T]} f^{n_i,\nu_i}\left(\frac{\partial \phi}{\partial t} + u_{n_i}\cdot\nabla \phi + \nu_i \Delta \phi\right) \; dxdt \\
				& = - \int_{\mathbb{T}^d} f_0 \phi_0 \; dx.
			\end{aligned}
		\end{multline*}
		
		Consequently, we see that $\bar{f}$ is a weak solution to \eqref{eqADE} along $u$ with initial data $f_0$. By Theorem \ref{ADEwellposed}, this solution is unique and so must be $f^\nu$.
		
		We are left to upgrade the convergence $f^{n_i,\nu_i} \xrightarrow{i\to\infty} \bar{f}$ from strong in $L^2 L^2$ and weak in $L^2 H^1$, to strong in $L^2 H^1$ and strong in $L^\infty L^2$.
		
		In light of the Trace Formula \eqref{eqtrace}, and the uniform bound on $f^{n_i,\nu_i}$ for all $i\in\mathbb{N}$ in $C^0L^2$ (by Theorem \ref{ADEwellposed}), strong convergence $f^{n_i,\nu_i} \xrightarrow{i\to\infty} f^\nu$ in $L^2 L^2$ implies $f^{n_i,\nu_i}(\cdot, T) \xrightharpoonup{i\to\infty} f^\nu(\cdot, T)$ converges weakly in $L^2(\mathbb{T}^d)$ (note that $f^{n_i,\nu_i}$, $f^{\nu}$ are uniformly bounded in $C^0L^2$ by Theorem \ref{ADEwellposed}). In particular
		\begin{equation*}
			\liminf_{i\to\infty} \int_{\mathbb{T}^d} \left|f^{n_i,\nu_i}(\cdot, T)\right|^2 \; dx \ge \int_{\mathbb{T}^d} \left|f^\nu(\cdot, T)\right|^2 \; dx.
		\end{equation*}
		
		In light of the Energy Identity \eqref{eqenergy}, this implies
		\begin{equation*}
			\limsup_{i\to\infty} \int_{\mathbb{T}^d \times [0,T]} \left|\nabla f^{n_i,\nu_i}\right|^2 \; dx dt \le \int_{\mathbb{T}^d \times [0,T]} \left|\nabla f^\nu\right|^2 \; dx dt.
		\end{equation*}
		
		Since also $f^{n_i,\nu_i} \xrightharpoonup{i\to\infty} f^\nu$ weakly in $L^2H^1$ we have
		\begin{equation*}
			\liminf_{i\to\infty} \int_{\mathbb{T}^d \times [0,T]} \left( \left|f^{n_i,\nu_i}\right|^2 + \left|\nabla f^{n_i,\nu_i}\right|^2 \right) \; dx dt \ge \int_{\mathbb{T}^d \times [0,T]} \left( \left|f^\nu\right|^2 + \left|\nabla f^\nu\right|^2 \right) \; dx dt.
		\end{equation*}
		
		However, since $f^{n_i,\nu_i} \xrightarrow{i\to\infty} f^\nu$ strongly in $L^2L^2$, from the above we must have convergence of the norms $\left\|f^{n_i,\nu_i}\right\|_{L^2H^1}^2 \xrightarrow{i\to\infty} \left\|f^\nu\right\|_{L^2H^1}^2$. Thus the weak convergence $f^{n_i,\nu_i} \xrightharpoonup{i\to\infty} f^\nu$ in $L^2H^1$ is in fact strong, as required.
		
		To extend to strong convergence in $L^\infty L^2$, we use the fact that $f^\nu \in C^0L^2$ to take a smooth approximation $\phi_\epsilon \in C^\infty(\mathbb{T}^d\times [0,T])$, such that $\left\|\phi_\epsilon - f^\nu\right\|_{L^\infty L^2} \le \epsilon$. Then by the Trace Formula \eqref{eqtrace}, and the Energy Identity \eqref{eqenergy}, for all $t \in [0,T]$,
		\begingroup
		\allowdisplaybreaks
		\begin{align*}
			& \left\|f^{n_i,\nu_i}(\cdot, t) - f^\nu(\cdot, t)\right\|_{L^2(\mathbb{T}^d)}^2 \\
			& \qquad =
			\begin{aligned}[t]
				& \left\|f^{n_i,\nu_i}(\cdot, t)\right\|_{L^2(\mathbb{T}^d)}^2 + \left\|f^\nu(\cdot, t)\right\|_{L^2(\mathbb{T}^d)}^2 - 2 \int_{\mathbb{T}^d} f^{n_i,\nu_i}(\cdot, t) f^\nu(\cdot, t)\;dx
			\end{aligned} \\
			& \qquad \le
			\begin{aligned}[t]
				& \left\|f^{n_i,\nu_i}(\cdot, t)\right\|_{L^2(\mathbb{T}^d)}^2 + \left\|f^\nu(\cdot, t)\right\|_{L^2(\mathbb{T}^d)}^2 + 2\epsilon\left\|f_0\right\|_{L^2} - 2 \int_{\mathbb{T}^d} f^{n_i,\nu_i}(\cdot, t) \phi_\epsilon(\cdot, t)\;dx
			\end{aligned} \\
			& \qquad =
			\begin{aligned}[t]
				& \left\|f^{n_i,\nu_i}(\cdot, t)\right\|_{L^2(\mathbb{T}^d)}^2 + \left\|f^\nu(\cdot, t)\right\|_{L^2(\mathbb{T}^d)}^2 + 2\epsilon\left\|f_0\right\|_{L^2} - 2 \int_{\mathbb{T}^d} f_0 \phi_\epsilon(\cdot, 0)\;dx \\
				& - 2\int_{\mathbb{T}^d \times [0,t]}f^{n_i,\nu_i} \left(\frac{\partial \phi_\epsilon}{\partial t} + u_{n_i} \cdot \nabla \phi_\epsilon + \nu_i \Delta \phi_\epsilon \right) \;dxdt
			\end{aligned} \\
			& \qquad \le
			\begin{aligned}[t]
				& \left\|f^{n_i,\nu_i}(\cdot, t)\right\|_{L^2(\mathbb{T}^d)}^2 + \left\|f^\nu(\cdot, t)\right\|_{L^2(\mathbb{T}^d)}^2 + 2\epsilon\left\|f_0\right\|_{L^2} - 2 \int_{\mathbb{T}^d} f_0 \phi_\epsilon(\cdot, 0)\;dx \\
				& - 2\int_{\mathbb{T}^d \times [0,t]}f^\nu \left(\frac{\partial \phi_\epsilon}{\partial t} + u \cdot \nabla \phi_\epsilon + \nu \Delta \phi_\epsilon \right) \;dxdt \\
				& + 2\left\|f^{n_i,\nu_i} - f^\nu\right\|_{L^2L^2}\left\|\frac{\partial \phi_\epsilon}{\partial t}\right\|_{L^2 L^2} + 2|\nu_i-\nu|\left\|f^{n_i,\nu_i}\right\|_{L^2L^2}\left\|\Delta \phi_\epsilon\right\|_{L^2 L^2} \\
				& + 2\left|\int_{\mathbb{T}^d \times [0,t]} (f^{n_i,\nu_i}u_{n_i} - f^\nu u)\cdot \nabla \phi_\epsilon \; dxdt \right|
			\end{aligned} \\
			& \qquad =
			\begin{aligned}[t]
				& \left\|f^{n_i,\nu_i}(\cdot, t)\right\|_{L^2(\mathbb{T}^d)}^2 + \left\|f^\nu(\cdot, t)\right\|_{L^2(\mathbb{T}^d)}^2 + 2\epsilon\left\|f_0\right\|_{L^2} - 2 \int_{\mathbb{T}^d} f^\nu(\cdot, t) \phi_\epsilon(\cdot, t)\;dx \\
				& + 2\left\|f^{n_i,\nu_i} - f^\nu\right\|_{L^2L^2}\left\|\frac{\partial \phi_\epsilon}{\partial t}\right\|_{L^2 L^2} + 2|\nu_i-\nu|\left\|f^{n_i,\nu_i}\right\|_{L^2L^2}\left\|\Delta \phi_\epsilon\right\|_{L^2 L^2} \\
				& + 2\left|\int_{\mathbb{T}^d \times [0,t]} (f^{n_i,\nu_i}u_{n_i} - f^\nu u)\cdot \nabla \phi_\epsilon \; dxdt \right|
			\end{aligned} \\
			& \qquad \le
			\begin{aligned}[t]
				& \left\|f^{n_i,\nu_i}(\cdot, t)\right\|_{L^2(\mathbb{T}^d)}^2 - \left\|f^\nu(\cdot, t)\right\|_{L^2(\mathbb{T}^d)}^2 + 4\epsilon\left\|f_0\right\|_{L^2} \\
				& + 2\left\|f^{n_i,\nu_i} - f^\nu\right\|_{L^2L^2}\left\|\frac{\partial \phi_\epsilon}{\partial t}\right\|_{L^2 L^2} + 2|\nu_i-\nu|\left\|f^{n_i,\nu_i}\right\|_{L^2L^2}\left\|\Delta \phi_\epsilon\right\|_{L^2 L^2} \\
				& + 2\left|\int_{\mathbb{T}^d \times [0,t]} (f^{n_i,\nu_i}u_{n_i} - f^\nu u)\cdot \nabla \phi_\epsilon \; dxdt \right|
			\end{aligned} \\
			& \qquad =
			\begin{aligned}[t]
				& 4\epsilon\left\|f_0\right\|_{L^2} - \int_{\mathbb{T}^d \times [0,t]} \left(2\nu_i \left|\nabla f^{n_i,\nu_i}\right|^2 - 2\nu \left|\nabla f^\nu\right|^2\right) \;dxdt  \\
				& + 2\left\|f^{n_i,\nu_i} - f^\nu\right\|_{L^2L^2}\left\|\frac{\partial \phi_\epsilon}{\partial t}\right\|_{L^2 L^2} + 2|\nu_i-\nu|\left\|f^{n_i,\nu_i}\right\|_{L^2L^2}\left\|\Delta \phi_\epsilon\right\|_{L^2 L^2} \\
				& + 2\left|\int_{\mathbb{T}^d \times [0,t]} (f^{n_i,\nu_i}u_{n_i} - f^\nu u)\cdot \nabla \phi_\epsilon \; dxdt \right|
			\end{aligned} \\
			& \qquad \le
			\begin{aligned}[t]
				& 4\epsilon\left\|f_0\right\|_{L^2} + 2\left\|{\nu_i}^\frac{1}{2}f^{n_i,\nu_i} - {\nu}^\frac{1}{2}f^\nu\right\|_{L^2 H^1}\left({\nu_i}^\frac{1}{2}\left\|f^{n_i,\nu_i}\right\|_{L^2 H^1} + {\nu}^\frac{1}{2}\left\|f^\nu\right\|_{L^2 H^1}\right) \\
				& + C_\epsilon\left(\left\|f^{n_i,\nu_i} - f^\nu\right\|_{L^2L^2} + |\nu_i-\nu|\left\|f^{n_i,\nu_i}\right\|_{L^2L^2}\right) \\
				& + 2\left|\int_{\mathbb{T}^d \times [0,t]} (f^{n_i,\nu_i}u_{n_i} - f^\nu u)\cdot \nabla \phi_\epsilon \; dxdt \right|,
			\end{aligned}
		\end{align*}
		\endgroup
		with $\epsilon>0$ arbitrary, and $C_\epsilon$ a constant that depends only on $\phi_\epsilon$, and not on $t \in [0,T]$.
		
		Therefore by the convergence $\nu_i \xrightarrow{i\to\infty} \nu$, and $f^{n_i,\nu_i} \xrightarrow{i\to\infty} f^\nu$ strongly in $L^2H^1$, to show strong convergence of $f^{n_i,\nu_i}\xrightarrow{i\to\infty}f^\nu$ in $L^\infty L^2$ we need only show, for fixed $\phi_\epsilon$
		\begin{equation*}
			\sup_{t \in[0,T]} \left|\int_{\mathbb{T}^d \times [0,t]} (f^{n_i,\nu_i}u_{n_i} - f^\nu u)\cdot \nabla \phi_\epsilon \; dxdt \right| \xrightarrow{i \to \infty} 0.
		\end{equation*}
		
		Suppose not, then there exists a sequence $t_i \in [0,T]$ for $i\in\mathbb{N}$ such that
		\begin{equation}
			\left|\int_{\mathbb{T}^d\times[0,t_i]}\left(f^{n_i,\nu_i}u_{n_i} - f^\nu u\right)\cdot\nabla\phi_\epsilon \; dxdt \right| \ge c, \label{eqcontr}
		\end{equation}
		for some $c>0$.
		
		By compactness of $[0,T]$, and taking a subsequence if necessary, we may further assume that $t_i \xrightarrow{i\to\infty} \bar{t}$ for some $\bar{t}\in[0,T]$. However, there is strong convergence $1_{t\in[0,t_i]}f^{n_i,\nu_i}, \; 1_{t\in[0,t_i]}f^\nu \xrightarrow{i\to\infty} 1_{t\in[0,\bar{t}]}f^\nu$ in $L^2 L^2$, and boundedness of the sequence $u_{n_i}$ in $L^\infty L^\infty$, and weak-$*$ convergence $u_{n_i} \xrightharpoonup{i\to\infty} u$ in $L^\infty L^\infty$. Therefore, the products $1_{t\in[0,t_i]}f^{n_i,\nu_i} u_{n_i} - 1_{t\in[0,t_i]}f^\nu u \xrightharpoonup{i\to\infty} 0$ converge weakly in $L^2 L^2$, contradicting \eqref{eqcontr}, and hence proving the claim.
	\end{proof}
	
	We now give the following corollary of the previous two propositions. By first fixing a viscosity $\nu_i>0$ and applying Proposition \ref{ADEtoADE}, we find that advection-diffusion is not sensitive to suitably small scale changes to the vector field $u$. However, inviscid transport is, and so we adjust the inviscid solution to suit our requirements without much change to advection-diffusion with viscosity $\nu$. Having done so, we then apply Proposition \ref{ADEtoTE} to find a smaller $\nu_{i+1}>0$ for which advection-diffusion is close to the adjusted inviscid transport. Repeating, we may make finer and finer adjustments to our vector field, for which a well-prepared vanishing sequence of viscosity $\{\nu_i\}_{i\in\mathbb{N}}$ permits a good approximation of advection-diffusion to suitable inviscid transport.
	
	\begin{theorem}\label{vcontrol}
		Fix some $M>0$, and a metric $d_*$ inducing the weak-$*$ topology on
		\begin{equation*}
			X = \left\{u\in L^\infty \left( [0,T];L^\infty(\mathbb{T}^d;\mathbb{R}^d)\right) : \left\|u\right\|_{L^\infty L^\infty}\le M \right\}.
		\end{equation*}
		
		Let $Y\subset X$ be the set of all divergence-free vector fields admitting a unique renormalised weak solution (unique in the class of all $L^\infty L^\infty$ weak solutions) to \eqref{eqTE} for any initial data $f_0 \in L^\infty(\mathbb{T}^d)$.
		
		Let $\left\{u_i\right\}_{i\in\mathbb{N}}\subset Y$, $f_0 \in L^\infty(\mathbb{T}^d)$. For each $n\in\mathbb{N}$, and $\nu > 0$, denote by $f^{n,\nu}$, respectively $f^n$, the unique weak solution to \eqref{eqADE}, respectively \eqref{eqTE}, along $u_n$ with initial data $f_0$. Then,
		\begin{enumerate}[label=\textbf{S.\arabic*},ref=S.\arabic*]
			\item \label{list1} For all $n\in\mathbb{N}$, there exists $\nu_n>0$, $\epsilon_n>0$ depending only on $\left\{u_i\right\}_{i=1}^n$ (and in particular not on $f_0$), with $\nu_n\xrightarrow{n\to\infty}0$ monotonically, such that the following hold true:
			\item \label{list2} For all $p\in[1,\infty)$,
			\begin{equation*}
				\sup_{0<\nu\le\nu_n}\left\|f^{n,\nu}-f^n\right\|_{L^\infty L^p}\xrightarrow{n\to\infty}0.
			\end{equation*}
			\item \label{list3} If $d_*(u_{n+1},u_n) \le \epsilon_n$ for all $n\in\mathbb{N}$, then $u_n \xrightharpoonup{n\to\infty} u_\infty$ converges in weak-$*$ $L^\infty L^\infty$ to some divergence-free vector field $u_\infty \in L^\infty L^\infty$,
			\item \label{list4} and if we denote by $f^{\infty,\nu}$ the unique weak solution to \eqref{eqADE} along $u_\infty$ with initial data $f_0$, then for all $p \in [1,\infty)$,
			\begin{equation*}
				\sup_{\nu_n \le \nu\le\nu_1}\left\|f^{n,\nu}-f^{\infty,\nu}\right\|_{L^\infty L^p}\xrightarrow{n\to\infty}0.
			\end{equation*}
			\item \label{list5} In particular, for all $p\in[1,\infty)$,
			\begin{equation*}
				\left\|f^{\infty,\nu_n}-f^n\right\|_{L^\infty L^p}\xrightarrow{n\to\infty}0.
			\end{equation*}
		\end{enumerate}
	\end{theorem}
	\begin{proof}
		By separability of $L^1(\mathbb{T}^d)$, take a countable dense sequence $\{f_0^m\}_{m \in\mathbb{N}}$ in $L^1(\mathbb{T}^d)$, satisfying $f_0^m \in L^\infty(\mathbb{T}^d)$.
		
		For all $m\in\mathbb{N}$, $n \in \mathbb{N}$, and $\nu > 0$, denote by $f^{m;n,\nu}$, respectively $f^{m;n}$, the unique weak solution to \eqref{eqADE}, respectively \eqref{eqTE}, along $u_n$ with initial data $f^m_0$. If in addition $u_n \xrightharpoonup{n\to\infty} u_\infty$ in weak-$*$ $L^\infty L^\infty$, denote by $f^{m;\infty,\nu}$ the unique weak solution to \eqref{eqADE} along $u_\infty$ with initial data $f_0^m$.
		
		For the given $f_0 \in L^\infty(\mathbb{T}^d)$, for each $n \in \mathbb{N}$, and $\nu > 0$, denote by $f^{n,\nu}$, respectively $f^n$, the unique weak solution to \eqref{eqADE}, respectively \eqref{eqTE}, along $u_n$ with initial data $f_0$. If in addition $u_n \xrightharpoonup{n\to\infty} u_\infty$ in weak-$*$ $L^\infty L^\infty$, denote by $f^{\infty,\nu}$ the unique weak solution to \eqref{eqADE} along $u_\infty$ with initial data $f_0$.
		
		Notice, by linearity of \eqref{eqADE}, \eqref{eqTE}, and the Initial $L^p$-Inequality \eqref{eqinitineq}, that for each $m \in \mathbb{N}$, $n\in\mathbb{N}$, $\nu >0$, we have
		\begin{gather}
			\left\|f^{n,\nu}-f^{m;n,\nu}\right\|_{L^\infty L^1} \le\left\|f_0 - f^m_0\right\|_{L^1}, \label{eqdensity1} \\
			\left\|f^{n} - f^{m;n}\right\|_{L^\infty L^1} \le\left\|f_0 - f^m_0\right\|_{L^1}, \label{eqdensity2} \\
			\left\|f^{\infty,\nu}-f^{m;\infty,\nu}\right\|_{L^\infty L^1}\le\left\|f_0 - f^m_0\right\|_{L^1}. \label{eqdensity3}
		\end{gather}
		
		We will use this and $L^1(\mathbb{T}^d)$-density to only consider the countable set of initial data $\{f_0^m\}_{m\in\mathbb{N}}$.
		
		Given $n\in\mathbb{N}$ and $\{u_i\}_{i=1}^n\subset Y$ we need to find $\nu_n>0$, $\epsilon_n>0$ such that \eqref{list1}-\eqref{list5} hold. For the purpose of the proof we additionally find $\delta_n>0$ such that, for all $m\in\{1,...,n\}$, $\nu\in[\nu_n,\nu_1]$, and $w\in Y$ with $d_*(w,u_n)\le\delta_n$,
		\begin{equation}\label{eqdelta}
			\left\|f^{m;n,\nu} - g^{m;\nu}\right\|_{L^\infty L^1} \le \frac{1}{n},
		\end{equation}
		where $g^{m,\nu}$ is the unique weak solution to \eqref{eqADE} along $w$ with initial data $f_0^m$.
		
		We do so by induction. Given $\{u_i\}_{i=1}^n$ as in the statement of the theorem assume $\{(\nu_i, \epsilon_i, \delta_i)\}_{i=1}^{n-1}$ are already chosen (an empty set if $n=1$).
		
		By applying Proposition \ref{ADEtoTE} for each initial data $\{f_0^m\}_{m=1}^n$ there exists some $\nu_n > 0$ such that for all $m\in\{1,...,n\}$, $\nu \in (0,\nu_n]$,
		\begin{equation}\label{eqnu}
			\left\|f^{m;n,\nu} - f^{m;n} \right\|_{L^\infty L^1} \le \frac{1}{n}.
		\end{equation}
		
		One may choose $\nu_n$ to be smaller, so that $\nu_n \le \frac{1}{n}$, and if $n\ne1$, $\nu_n < \nu_{n-1}$.
		
		We next find $\delta_n > 0$ satisfying \eqref{eqdelta}. Assume to the contrary that no such $\delta_n$ exists, then there is a sequence of divergence-free vector fields $\{w_k\}_{k \in \mathbb{N}} \subset Y$ with $w_k \xrightharpoonup{k\to\infty} u_n$ converging in weak-$*$ $L^\infty L^\infty$, which violate Proposition \ref{ADEtoADE} for at least one of the initial data $\{f_0^m\}_{m=1}^n$, where we have substituted $a=\nu_n$, $b = \nu_1$ into Proposition \ref{ADEtoADE}.
		
		One may choose $\delta_n$ smaller so that $\delta_n \le \frac{1}{n}$. Finally, take
		\begin{equation}\label{eqepsilon}
			\epsilon_n = \min_{k\in\{0,...,n-1\}}\{\delta_{n-k}2^{-k-1}\}.
		\end{equation}
		
		\eqref{list1} follows immediately from our choice of $\nu_n$. Meanwhile \eqref{list2} with $p=1$ follows from \eqref{eqdensity1}, \eqref{eqdensity2}, \eqref{eqnu}. Interpolation with the existing uniform bound on $f^{n,\nu}$, $f^n$ in $L^\infty L^\infty$ then gives \eqref{list2} for $p\in(1,\infty)$.
		
		We are now left to show \eqref{list3}-\eqref{list5} are satisfied when in addition $d_*(u_{n+1},u_n) \le \epsilon_n$ for all $n\in\mathbb{N}$.
		
		By \eqref{eqepsilon}, $\sum_{i=n}^\infty \epsilon_i \le \delta_n$, and also $\delta_n \to 0$, so $\{u_i\}_{i\in\mathbb{N}}$ is a $d_*$-Cauchy sequence, uniformly bounded in $L^\infty L^\infty$ by $M$, and so converges in weak-$*$ $L^\infty L^\infty$ to some limit $u_\infty$, proving \eqref{list3}.
		
		Moreover, for all $n\in\mathbb{N}$, $d_*(u_\infty, u_n) \le \delta_n$, so taking in \eqref{eqdelta} $w=u_\infty$, $g^{m;\nu}=f^{m;\infty,\nu}$ then \eqref{list4} follows from \eqref{eqdensity1}, \eqref{eqdensity3}, and interpolation with the existing uniform bound on $f^{n,\nu}$, $f^{\infty,\nu}$ in $L^\infty L^\infty$.
		
		Finally \eqref{list5} is a simple corollary of \eqref{list2} and \eqref{list4}.
	\end{proof}
	
	\section{Non-uniqueness}\label{Nonuniqueness}
	\subsection{Notation}\label{torusnotation}
	For $x \in \mathbb{R}^2$ we set $x=(x_1, x_2)$, and define the corresponding unit vectors $e_1 = (1,0)$, $e_2 = (0,1)$. For $i \in \{1,2\}$ an index, we define by $\hat{i}$ the other index, that is $\hat{i} \in \{1,2\}\setminus\{i\}$.
	
	When working on the 2-torus, $[x] \in \mathbb{T}^2 = \mathbb{R}^2/\mathbb{Z}^2$ shall instead always be written in terms of a representative $x \in \mathbb{R}^2$, usually $x \in [0,1)^2$. For notational convenience, functions on the torus $\mathbb{T}^2$ will often be considered periodic functions on $\mathbb{R}^2$, and vice versa. Integrals over $\mathbb{T}^2$ are then strictly speaking over $[0,1)^2$, such as when taking norms.
	\vspace{\baselineskip}
	
	First, we define the shear flows that form the building block of our construction. We shall later apply these shears in a carefully designed order reminiscent of a cantor subset of $[0,T]$. The arising self-cancellation of these shear flows will produce non-uniqueness of renormalised weak solutions (Definition \ref{renormalised}) for any initial data to \eqref{eqTE}. Moreover, the small spatial scale of the shears, see \eqref{eqshearweakconv} below, will allow us to approximate the vanishing viscosity limit of \eqref{eqADE}.
	\begin{definition}[Lagrangian shear flow]\label{shearflow}
		For $L\in\mathbb{N}$, we divide $\mathbb{R}$ into disjoint intervals $\bigcup_{m \in \mathbb{Z}} \left[\frac{m}{2L},\frac{m+1}{2L}\right)$. We define for each $i\in\{1,2\}$, $L\in \mathbb{N}$, the periodic vector field $u^{(i;L)}: \mathbb{T}^2 \to \mathbb{R}^2$,
		\begin{equation*}
			u^{(i;L)}(x) = \begin{cases} e_i & \text{if } x_{\hat{i}} \in \left[\frac{m}{2L},\frac{m+1}{2L}\right) \text{ for even $m$}, \\ -e_i & \text{if } x_{\hat{i}} \in \left[\frac{m}{2L},\frac{m+1}{2L}\right) \text{ for odd $m$}, \end{cases}
		\end{equation*}
		where, since $2L$ is even, the definition is a periodic function of $x\in\mathbb{R}^2$, and so is well-defined on $\mathbb{T}^2$. $u^{(i;L)}$ is bounded and divergence-free in the distributional sense since it is of the form $u^{(1;L)}(x)=(g(x_2),0)$, or $u^{(2;L)}(x)=(0,g(x_1))$ for some $g \in L^\infty(\mathbb{T}^2)$. We refer to this vector field as the $(i;L)$-Lagrangian shear.
		
		We shall denote by $\big\{y_t^{(i;L)}\big\}_{t \in (-\infty, \infty)}$ the following Lagrangian flow (Definition \ref{lagrangian}) along $u^{(i;L)}$,
		\begin{equation*}
			y_t^{(i;L)} : \begin{array}[t]{rcl} \mathbb{T}^2 & \to & \mathbb{T}^2 \\ x & \mapsto & \begin{cases} x+t e_i \mod \mathbb{Z}^2 & \text{if } x_{\hat{i}} \in \left[\frac{m}{2L},\frac{m+1}{2L}\right) \text{ for even $m$}, \\ x-t e_i \mod \mathbb{Z}^2 & \text{if } x_{\hat{i}} \in \left[\frac{m}{2L},\frac{m+1}{2L}\right) \text{ for odd $m$}. \end{cases} \end{array}
		\end{equation*}
		
		As per Definition \ref{lagrangian} of Lagrangian flows, this preserves the Lebesgue-measure on $\mathbb{T}^2$, e.g. for $i=1$ we can decompose a Lebesgue-measurable subset $A \subset \mathbb{T}^2$ into $A_m = A \cap \left(\mathbb{T} \times \left[\frac{m}{2L},\frac{m+1}{2L}\right)\right)$, for which each $\big(y_t^{(1;L)}\big)^{-1}(A_m)$ is a translation. Moreover, it is invertible with inverse given by $\big(y_t^{(i;L)}\big)^{-1}(x) = y_{-t}^{(i;L)}(x)$, and is absolutely continuous since with respect to $t$ since it is differentiable with derivative $\frac{\partial}{\partial t} y_t^{(i;L)} (x) = u^{(i;L)}(x) = u^{(i;L)}(y_t^{(i;L)}(x))$. Therefore, $\big\{y_t^{(i;L)}\big\}_{t \in (-\infty, \infty)}$ is a Lagrangian flow along $u^{(i;L)}$.
		
		Finally, for each $i \in \{1,2\}$, we show
		\begin{equation}\label{eqshearweakconv}
			u^{(i;L)} \xrightharpoonup{L\to\infty} 0,
		\end{equation}
		with convergence in weak-$*$ $L^\infty(\mathbb{T}^2)$.
		
		The proof of \eqref{eqshearweakconv} is similar to that of the Riemann-Lebesgue lemma. Notice that $u^{(i;L)}\left(x+\frac{1}{2L}e_{\hat{i}}\right) = - u^{(i;L)}(x)$ and so for any test function $f \in L^1(\mathbb{T}^2)$, by changing variables we have that
		\begin{equation*}
			\int_{\mathbb{T}^2} f(x)u^{(i;L)}(x) \; dx = \int_{\mathbb{T}^2} \frac{f(x) - f\left(x+\frac{1}{2L}e_{\hat{i}}\right)}{2} u^{(i;L)}(x) \; dx.
		\end{equation*}
		
		Now $f\left(x+\frac{1}{2L}e_{\hat{i}}\right) \xrightarrow{L\to\infty} f$ converges strongly in $L^1(\mathbb{T}^2)$. Together with the uniform bound $\left\| u^{(i;L)} \right\|_{L^\infty} \le 1$ this implies the above integral converges to zero as $L\to\infty$, as required.
	\end{definition}
	\vspace{\baselineskip}
	
	Next, we show the well-posedness of \eqref{eqTE} along these vector fields. A weaker version of the following result, valid when $f\in L^\infty L^\infty$, follows from the breakthrough well-posedness theory of Ambrosio for vector fields of bounded variation in space, \cite{ambrosio2004transport}. This is, at its heart, rather more involved than required for the simple case of shear flows. Therefore, we prefer instead to give a direct elementary proof of the uniqueness for all $f\in L^1 L^1$.
	\begin{proposition}[Shear flow uniqueness]\label{shearuniq}
		We follow the notation introduced in Definition \ref{shearflow}.
		
		Suppose $f\in L^1L^1$ is a weak solution to \eqref{eqTE} along $u^{(i;L)}$ on an open interval $I\subset(0,T)$. Then there exists some $g \in L^1(\mathbb{T}^2)$ such that $f$ is (a.e. in $\mathbb{T}^2\times I$) equal to the following Lagrangian solution (Definition \ref{lagrangian}) associated with $g$,
		\begin{equation*}
			f(\cdot,t)=g \circ \big(y^{(i;L)}_t\big)^{-1}.
		\end{equation*}
		
		In particular, $f$ is a renormalised weak solution to \eqref{eqTE} (Definition \ref{renormalised}). Moreover, $g \circ \big(y^{(i;L)}_t\big)^{-1}\in C^0((-\infty,\infty);L^1(\mathbb{T}^2))$.
	\end{proposition}
	\begin{proof}
		Without loss of generality suppose $i=1$, and use the shorthand $y_t=y_t^{(1;L)}$ for each $t \in (-\infty, \infty)$, that is
		\begin{equation*}
			y_t(x) = \begin{cases} x+t e_1 & \text{if } x_{2} \in \left[\frac{m}{2L},\frac{m+1}{2L}\right) \text{ for even $m$}, \\ x-t e_1 & \text{if } x_{2} \in \left[\frac{m}{2L},\frac{m+1}{2L}\right) \text{ for odd $m$}. \end{cases}
		\end{equation*}
		
		Since the $e_2$-component of $y_t(x)$ is unchanged, that is $[y_t(x)]_2 = x_2$, we see
		\begin{equation*}
			u^{(1;L)}(y_t(x),t) = u^{(1;L)}(x,t).
		\end{equation*}
		
		Fix $m \in \{0,1,...2L-1\}$. Take a test function $\phi \in C_c^\infty\left(\left(\mathbb{T} \times \left(\frac{m}{2L},\frac{m+1}{2L}\right)\right) \times I\right)$, and let $\psi(x,t)=\phi(y_t^{-1}(x), t)$. This is well-defined since $y_t^{-1}$ preserves the set $\mathbb{T} \times \left(\frac{m}{2L},\frac{m+1}{2L}\right)$.
		
		Then $\psi \in C_c^\infty\left(\left(\mathbb{T} \times \left(\frac{m}{2L},\frac{m+1}{2L}\right)\right) \times I\right)$ is bounded, and supported and smooth on $\mathbb{T} \times \left(\frac{m}{2L},\frac{m+1}{2L}\right) \times I$ since $y_t^{-1}$ is smooth there. By the chain rule, the following point-wise equality holds
		\begin{equation*}
			\left(\frac{\partial \psi}{\partial t} + u^{(1;L)}\cdot\nabla \psi\right)(x,t) = \frac{\partial \phi}{\partial t}(y_t^{-1}(x),t).
		\end{equation*}
		
		Since $f \in L^1L^1$ is assumed to be a weak solution to \eqref{eqTE} along $u^{(1;L)}$ on the open interval $I\subset (0,T)$,
		\begin{equation*}
			\int_{\mathbb{T} \times \left(\frac{m}{2L},\frac{m+1}{2L}\right) \times I} f(x,t) \left(\frac{\partial \psi}{\partial t} + u^{(1;L)}\cdot\nabla \psi\right)(x,t) \; dxdt = 0,
		\end{equation*}
		which implies
		\begin{equation*}
			\int_{\mathbb{T} \times \left(\frac{m}{2L},\frac{m+1}{2L}\right) \times I} f(x,t) \frac{\partial \phi}{\partial t}(y_t^{-1}(x),t) \; dxdt = 0,
		\end{equation*}
		and after changing variables,
		\begin{equation*}
			\int_{\mathbb{T} \times \left(\frac{m}{2L},\frac{m+1}{2L}\right) \times I} f(y_t(x), t) \; \frac{\partial \phi}{\partial t}(x,t) \; dxdt = 0.
		\end{equation*}
		
		This holds for all $\phi \in C_c^\infty\left(\left(\mathbb{T} \times \left(\frac{m}{2L},\frac{m+1}{2L}\right)\right) \times I\right)$, and so as a distribution $f(y_t(x), t) \in L^1L^1$ is independent of $t\in I$, see for example Theorem 3.1.4' in \cite{hormander2003}. Therefore, there must exists some $g_m \in L^1(\mathbb{T} \times \left(\frac{m}{2L},\frac{m+1}{2L}\right))$ such that (for a.e. $t \in I$)
		\begin{equation*}
			f(y_t(x), t) = g_m(x) \text{ a.e. on } \mathbb{T} \times \left(\frac{m}{2L},\frac{m+1}{2L}\right).
		\end{equation*}
		
		Repeating for each $m \in \{0,1,...2L-1\}$ gives some $g \in L^1(\mathbb{T}^2)$ such that,
		\begin{equation*}
			f(y_t(x), t) = g(x) \text{ a.e. on } \mathbb{T}^2.
		\end{equation*}
		
		And so $f(\cdot, t) = g \circ y_t^{-1}$ is a Lagrangian solution to \eqref{eqTE} (Definition \ref{lagrangian}), and hence also a renormalised weak solution (Definition \ref{renormalised}, Remark \ref{lagrenorm}).
		
		To prove that $g\circ y_t^{-1} \in C^0((-\infty,\infty);L^1(\mathbb{T}^2))$ we shall use the fact that $y_t$ is bijective, measure preserving, and $1$-Lipschitz in time, i.e. for all $x \in \mathbb{T}^2$, $t,s \in \mathbb{R}$
		\begin{equation*}
			\left|y_t(x) - y_s(x)\right| \le |t-s|,
		\end{equation*}
		and so, by replacing $x$ with $y_s^{-1}(x)$,
		\begin{equation}\label{eqshearlip}
			\left|y_t \circ y_s^{-1}(x) - x\right| \le |t-s|.
		\end{equation}
		
		The following follows a standard argument. We fix $t \in I$ and write
		\begin{equation*}
			g\circ y_s^{-1} = \left( g\circ y_t^{-1}\right) \circ y_t \circ y_s^{-1}.
		\end{equation*}
		
		Now mollify $g\circ y_t^{-1} $ in $L^1(\mathbb{T}^2)$, that is for each $\epsilon > 0$ take some $\phi_\epsilon \in C^\infty(\mathbb{T}^2)$ such that
		\begin{equation*}
			\left\|\phi_\epsilon - g\circ y_t^{-1}\right\|_{L^1(\mathbb{T}^2)} \le \epsilon.
		\end{equation*}
		
		But also, since $y_t$, $y_s^{-1}$ are measure preserving,
		\begin{equation*}
			\left\|\phi_\epsilon \circ y_t \circ y_s^{-1} - g\circ y_s^{-1}\right\|_{L^1(\mathbb{T}^2)} \le \epsilon.
		\end{equation*}
		
		Then since $\phi_\epsilon$ is uniformly continuous on $\mathbb{T}^2$, by \eqref{eqshearlip} there exists $\delta \in (0,1)$ such that if $|t-s| < \delta$ then $\|\phi_\epsilon \circ y_t \circ y_s^{-1} - \phi_\epsilon\|_{L^1} \le \epsilon$, and so 
		\begin{equation*}
			\|g\circ y_t^{-1} - g\circ y_s^{-1}\|_{L^1(\mathbb{T}^2)} \le 3 \epsilon,
		\end{equation*}
		which completes the proof.
	\end{proof}
	\vspace{\baselineskip}
	
	Next, we demonstrate how these shear flows self-cancel. It is this property that gives rise to non-uniqueness in our later construction. This behaviour is quite different to recent works on anomalous dissipation, \cite{drivas2019anomalous,elgindi2023norm}, and even \cite{scoop,armstrong2023anomalous}, which also achieve some non-uniqueness of the vanishing viscosity limit. Contrary to our approach, these constructions aim to exploit some mixing effect of shear flows, and non-uniqueness occurs by exploiting different mixing along different viscosity subsequences. In contrast, the commutator relation below ensures that shear flows in our construction precisely reverse mixing that occurs previously. Non-uniqueness instead occurs by quantitatively different transport along different viscosity subsequences rather than by quantitatively different mixing along different viscosity subsequences.
	\begin{proposition}[Cancellation]\label{shearcanc}
		Let $L_1, L_2 \in \mathbb{N}$, and $i_1, i_2 \in \{1,2\}$ such that $i_1\ne i_2$. Suppose that $2\tau_2 = \frac{1}{2L_1}$, and that $2L_2\tau_1$ is an odd integer.
		
		Then the composition
		\begin{gather*}
			\left(y_{\tau_2}^{(i_2;L_2)}\right)^2 \circ y_{\tau_1}^{(i_1;L_1)} \circ \left(y_{\tau_2}^{(i_2;L_2)}\right)^2 \circ y_{\tau_1}^{(i_1;L_1)} = \mathrm{Id},
		\end{gather*}
		where we have denoted by $f^2 := f \circ f$.
	\end{proposition}
	\begin{proof}
		Without loss of generality, suppose $i_1=1,i_2=2$.
		
		Since $\left(y_{\tau_2}^{(2;L_2)}\right)^2 = y_{2\tau_2}^{(2;L_2)}$ we instead prove the more general result
		\begin{gather*}
			y_{\tau_2}^{(2;L_2)}\circ y_{\tau_1}^{(1;L_1)}\circ y_{\tau_2}^{(2;L_2)}\circ y_{\tau_1}^{(1;L_1)} = \mathrm{Id},
		\end{gather*}
		whenever both $2L_1\tau_2$ and $2L_2\tau_1$ are odd integers.
		
		To this end, we divide $\mathbb{T}^2$ into tiles as follows. Given $x \in [0,1)^2$ we find the unique integers $m_1, m_2 \in \mathbb{Z}$ such that
		\begin{equation}\label{eqequivclass}
			x \in \left[\frac{m_1}{2L_2},\frac{m_1+1}{2L_2}\right) \times \left[\frac{m_2}{2L_1},\frac{m_2+1}{2L_1}\right).
		\end{equation}
		
		We are concerned with the parity of $m_1$ and $m_2$ (even or odd), that is we have four `colours' on our tiling, two for each coordinate.
		
		This defines two equivalence relations $\left[\cdot\right]_{m_1}, \left[\cdot\right]_{m_2}$ on $\mathbb{T}^2$, each with two equivalence classes corresponding to the parity of either $m_1$ or $m_2$ in \eqref{eqequivclass}.
		
		Since $2L_1\tau_2$, $2L_2\tau_1$ are odd integers, the action of $y_{\tau_1}^{(1;L_1)}$ changes the parity of $m_1$ and not $m_2$, while the action of $y_{\tau_2}^{(2;L_2)}$ changes the parity of $m_2$ and not $m_1$.
		
		Introduce the shorthand
		\begingroup
		\allowdisplaybreaks
		\begin{gather*}
			x^{(0)}=x, \\
			x^{(1)}=y_{\tau_1}^{(1;L_1)}(x^{(0)}), \\
			x^{(2)}=y_{\tau_2}^{(2;L_2)}(x^{(1)}), \\
			x^{(3)}=y_{\tau_1}^{(1;L_1)}(x^{(2)}), \\
			x^{(4)}=y_{\tau_2}^{(2;L_2)}(x^{(3)}).
		\end{gather*}
		\endgroup
		
		Then, by the above discussion,
		\begin{gather*}
			\left[x^{(2)}\right]_{m_2} \ne \left[x^{(0)}\right]_{m_2}, \\
			\left[x^{(3)}\right]_{m_1} \ne \left[x^{(1)}\right]_{m_1}.
		\end{gather*}
		
		From the definitions of $y_{\tau_1}^{(1;L_1)}, y_{\tau_2}^{(2;L_2)}$ this implies
		\begin{gather*}
			x^{(3)} - x^{(2)} = -\left(x^{(1)} - x^{(0)}\right), \\
			x^{(4)} - x^{(3)} = -\left(x^{(2)} - x^{(1)}\right),
		\end{gather*}
		and hence $x^{(4)} = x^{(0)}$ as required.
	\end{proof}
	\vspace{\baselineskip}
	
	Now for some sequences $\{L_k\}_{k\in\mathbb{N}}$, $\{i_k\}_{k\in\mathbb{N}}$, $\{\tau_k\}_{k\in\mathbb{N}}$ to be specified later, we first choose the order in which to apply these shear flows $y_{\tau_k}^{(i_k;L_k)}$. We use a double index $(k,m)\in\mathcal{D}\subset\mathbb{Z}^2$, where $k$ shall index the parameters of the Lagrangian shear $y_{\tau_k}^{(i_k;L_k)}$, while $m$ denotes the $m$\textsuperscript{th} occurrence of that particular shear.
	
	To exploit Proposition \ref{shearcanc}, we shall ensure that there are exactly two occurrences of a suitable shear $y_{\tau_{k+1}}^{(i_{k+1};L_{k+1})}$ between each occurrence of $y_{\tau_k}^{(i_k;L_k)}$. If we denote by $t_{k,m}\in[0,1]$ the times at which the shears are first applied, we may illustrate this construction in the following diagram:
	
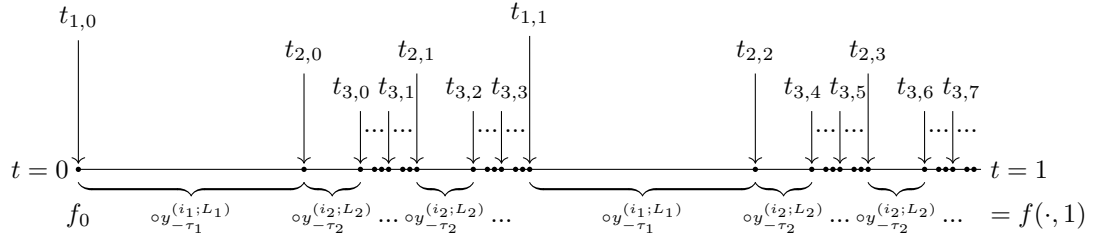
\begin{equation*}
		\begin{tikzpicture}
			\pgfdeclarelayer{background}
			\pgfsetlayers{background,main}
			\begin{scope}[scale=12]
				\draw[] (0,0) -- (1,0);
				\node[anchor=east] (t0) at (0,0) {$t=0$};
				\node[anchor=west] (t1) at (1,0) {$t=1$};
				\node[yshift=20mm] (t10) at (0,0) {${t_{1,0}}$};
				\draw[<-] ([yshift=6/12mm]0,0) -- (t10);
				\node[anchor=base,yshift=20mm] (t11) at (1/2,0) {${t_{1,1}}$};
				\draw[<-] ([yshift=6/12mm]1/2,0) -- (t11);
				\node[anchor=base,yshift=15mm] (t20) at (2/8,0) {${t_{2,0}}$};
				\draw[<-] ([yshift=6/12mm]2/8,0) -- (t20);
				\node[anchor=base,yshift=15mm] (t21) at (3/8,0) {${t_{2,1}}$};
				\draw[<-] ([yshift=6/12mm]3/8,0) -- (t21);
				\node[anchor=base,yshift=15mm] (t22) at (6/8,0) {${t_{2,2}}$};
				\draw[<-] ([yshift=6/12mm]6/8,0) -- (t22);
				\node[anchor=base,yshift=15mm] (t23) at (7/8,0) {${t_{2,3}}$};
				\draw[<-] ([yshift=6/12mm]7/8,0) -- (t23);
				\node[anchor=base,yshift=10mm] (t305) at (10.5/32,0) {${t_{3,0}}\;t_{3,1}$};
				\node[anchor=base,yshift=5mm] (n1) at (10.5/32,0) {$...$};
				\draw[<-] ([yshift=6/12mm]10/32,0) -- ([shift={(-0.5/32,0)}]t305.south);
				\node[anchor=base,yshift=5mm] (n1) at (11.5/32,0) {$...$};
				\draw[<-] ([yshift=6/12mm]11/32,0) -- ([shift={(0.5/32,0)}]t305.south);
				\node[anchor=base,yshift=10mm] (t325) at (14.5/32,0) {${t_{3,2}}\;t_{3,3}$};
				\node[anchor=base,yshift=5mm] (n1) at (14.5/32,0) {$...$};
				\draw[<-] ([yshift=6/12mm]14/32,0) -- ([shift={(-0.5/32,0)}]t325.south);
				\node[anchor=base,yshift=5mm] (n1) at (15.5/32,0) {$...$};
				\draw[<-] ([yshift=6/12mm]15/32,0) -- ([shift={(0.5/32,0)}]t325.south);
				\node[anchor=base,yshift=10mm] (t345) at (26.5/32,0) {${t_{3,4}}\;t_{3,5}$};
				\node[anchor=base,yshift=5mm] (n1) at (26.5/32,0) {$...$};
				\draw[<-] ([yshift=6/12mm]26/32,0) -- ([shift={(-0.5/32,0)}]t345.south);
				\node[anchor=base,yshift=5mm] (n1) at (27.5/32,0) {$...$};
				\draw[<-] ([yshift=6/12mm]27/32,0) -- ([shift={(0.5/32,0)}]t345.south);
				\node[anchor=base,yshift=10mm] (t365) at (30.5/32,0) {${t_{3,6}}\;t_{3,7}$};
				\node[anchor=base,yshift=5mm] (n1) at (30.5/32,0) {$...$};
				\draw[<-] ([yshift=6/12mm]30/32,0) -- ([shift={(-0.5/32,0)}]t365.south);
				\node[anchor=base,yshift=5mm] (n1) at (31.5/32,0) {$...$};
				\draw[<-] ([yshift=6/12mm]31/32,0) -- ([shift={(0.5/32,0)}]t365.south);
				\foreach \a in {0,1}{
					\fill[] (\a/2,0) circle (1/12pt);
					\foreach \b in {0,1}{
						\fill[] (\a/2+2/8+\b/8,0) circle (1/12pt);
						\foreach \c in {0,1}{
							\fill[] (\a/2+2/8+\b/8+2/32+\c/32,0) circle (1/12pt);
							\foreach \d in {0,1}{
								\fill[] (\a/2+2/8+\b/8+2/32+\c/32+2/128+\d/128,0) circle (1/12pt);
								\foreach \e in {0,1}{
								}
							}
						}
					}
				}
				\node[anchor=base,yshift=-20pt] (n1) at (0,0) {$f_0$};
				\draw[thick,decorate,decoration = {
					calligraphic brace,
					raise=5pt,
					amplitude=5pt,
					aspect=0.5
				}] (1/4,0) -- (0,0) node[anchor=base,pos=0.5,yshift=-20pt,black,font=\tiny]{$\circ y_{-\tau_1}^{(i_1;L_1)}$};
				\draw[thick,decorate,decoration = {
					calligraphic brace,
					raise=5pt,
					amplitude=5pt,
					aspect=0.5
				}] (3/4,0) -- (2/4,0) node[anchor=base,pos=0.5,yshift=-20pt,black,font=\tiny]{$\circ y_{-\tau_1}^{(i_1;L_1)}$};
				\draw[thick,decorate,decoration = {
					calligraphic brace,
					raise=5pt,
					amplitude=5pt,
					aspect=0.5
				}] (5/16,0) -- (4/16,0) node[anchor=base,pos=0.5,yshift=-20pt,black,font=\tiny]{$\circ y_{-\tau_2}^{(i_2;L_2)}$};
				\node[anchor=base,yshift=-20pt] (n1) at (5.5/16,0) {$...$};
				\draw[thick,decorate,decoration = {
					calligraphic brace,
					raise=5pt,
					amplitude=5pt,
					aspect=0.5
				}] (7/16,0) -- (6/16,0) node[anchor=base,pos=0.5,yshift=-20pt,black,font=\tiny]{$\circ y_{-\tau_2}^{(i_2;L_2)}$};
				\node[anchor=base,yshift=-20pt] (n1) at (7.5/16,0) {$...$};
				\draw[thick,decorate,decoration = {
					calligraphic brace,
					raise=5pt,
					amplitude=5pt,
					aspect=0.5
				}] (13/16,0) -- (12/16,0) node[anchor=base,pos=0.5,yshift=-20pt,black,font=\tiny]{$\circ y_{-\tau_2}^{(i_2;L_2)}$};
				\node[anchor=base,yshift=-20pt] (n1) at (13.5/16,0) {$...$};
				\draw[thick,decorate,decoration = {
					calligraphic brace,
					raise=5pt,
					amplitude=5pt,
					aspect=0.5
				}] (15/16,0) -- (14/16,0) node[anchor=base,pos=0.5,yshift=-20pt,black,font=\tiny]{$\circ y_{-\tau_2}^{(i_2;L_2)}$};
				\node[anchor=base,yshift=-20pt] (n1) at (15.5/16,0) {$...$};
				\node[anchor=base west,yshift=-20pt] (n1) at (1,0) {$=f(\cdot, 1)$};
			\end{scope}
		\end{tikzpicture}
	\end{equation*}
	
	If we include only the shears $y_{\tau_k}^{(i_k;L_k)}$ for $k\le K$, it can now be seen how Proposition \ref{shearcanc} may create two different behaviours of the trace $f(\cdot, 1)$ as $2K \to \infty$, and $2K+1 \to \infty$.
	
	The frequencies of the shears, $L_k\in\mathbb{N}$, will later be chosen to grow sufficiently quickly so that advection-diffusion \eqref{eqADE} along a viscosity subsequence $\nu_k>0$ will only include the effect of the first $k$ shears, as per Theorem \ref{vcontrol}.
	
	To this end, we first define a total order $<_\mathrm{time}$ on the indexing set $\mathcal{D}$, and then define $t_{k,m}$ so that they respect this ordering.
	
	\begin{definition}[Lexicographic dyadic ordering]\label{dyord}
		We define the following set of `dyadic' pairs
		\begin{equation*}
			\mathcal{D}=\{(k,m):k\in \mathbb{N}, m \in \mathbb{Z}, 0\le m<2^k\},
		\end{equation*}
		and define a total order $<_\mathrm{time}$ on $\mathcal{D}$ via
		\begin{equation}\label{eqdyord}
			(k_1,m_1) <_\mathrm{time} (k_2,m_2) \text{ if and only if }\begin{cases}
				m_1 2^{-k_1} < m_2 2^{-k_2}, \text{ or} \\
				m_1 2^{-k_1} = m_2 2^{-k_2} \text{ and } k_1<k_2.
			\end{cases}
		\end{equation}
		
		Define also for each $K \in \mathbb{N}$ the finite subset
		\begin{equation*}
			\mathcal{D}_K = \{(k,m) \in \mathcal{D} : k \le K\},
		\end{equation*}
		which inherits the total order $(\mathcal{D},<_\mathrm{time})$, which is now a finite order.
		
		We define for each $k \in \mathbb{N}$, and $m \in \mathbb{Z}$ with $0\le m < 2^k$,
		\begin{equation}\label{eqtkm}
			t_{k,m}  = \sum_{(k',m')<_\mathrm{time}(k,m)} 2^{-2k'} < 1,
		\end{equation}
		where the bound on $t_{k,m}$ follows from direct calculation of $\sum_{k\in\mathbb{N}}2^k2^{-2k}=1$, and an empty sum is zero.
		
		We illustrate this arrangement in the following diagram:
		
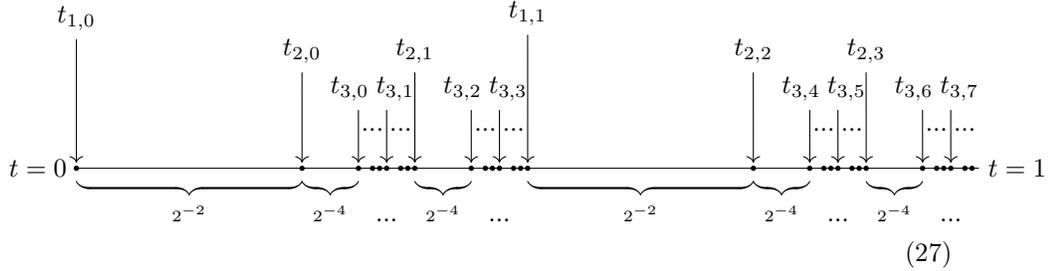
\begin{equation}\label{eqdydiag}
			\begin{tikzpicture}
				\pgfdeclarelayer{background}
				\pgfsetlayers{background,main}
				\begin{scope}[scale=12]
					\draw[] (0,0) -- (1,0);
					\node[anchor=east] (t0) at (0,0) {$t=0$};
					\node[anchor=west] (t1) at (1,0) {$t=1$};
					\node[yshift=20mm] (t10) at (0,0) {${t_{1,0}}$};
					\draw[<-] ([yshift=6/12mm]0,0) -- (t10);
					\node[anchor=base,yshift=20mm] (t11) at (1/2,0) {${t_{1,1}}$};
					\draw[<-] ([yshift=6/12mm]1/2,0) -- (t11);
					\node[anchor=base,yshift=15mm] (t20) at (2/8,0) {${t_{2,0}}$};
					\draw[<-] ([yshift=6/12mm]2/8,0) -- (t20);
					\node[anchor=base,yshift=15mm] (t21) at (3/8,0) {${t_{2,1}}$};
					\draw[<-] ([yshift=6/12mm]3/8,0) -- (t21);
					\node[anchor=base,yshift=15mm] (t22) at (6/8,0) {${t_{2,2}}$};
					\draw[<-] ([yshift=6/12mm]6/8,0) -- (t22);
					\node[anchor=base,yshift=15mm] (t23) at (7/8,0) {${t_{2,3}}$};
					\draw[<-] ([yshift=6/12mm]7/8,0) -- (t23);
					\node[anchor=base,yshift=10mm] (t305) at (10.5/32,0) {${t_{3,0}}\;t_{3,1}$};
					\node[anchor=base,yshift=5mm] (n1) at (10.5/32,0) {$...$};
					\draw[<-] ([yshift=6/12mm]10/32,0) -- ([shift={(-0.5/32,0)}]t305.south);
					\node[anchor=base,yshift=5mm] (n1) at (11.5/32,0) {$...$};
					\draw[<-] ([yshift=6/12mm]11/32,0) -- ([shift={(0.5/32,0)}]t305.south);
					\node[anchor=base,yshift=10mm] (t325) at (14.5/32,0) {${t_{3,2}}\;t_{3,3}$};
					\node[anchor=base,yshift=5mm] (n1) at (14.5/32,0) {$...$};
					\draw[<-] ([yshift=6/12mm]14/32,0) -- ([shift={(-0.5/32,0)}]t325.south);
					\node[anchor=base,yshift=5mm] (n1) at (15.5/32,0) {$...$};
					\draw[<-] ([yshift=6/12mm]15/32,0) -- ([shift={(0.5/32,0)}]t325.south);
					\node[anchor=base,yshift=10mm] (t345) at (26.5/32,0) {${t_{3,4}}\;t_{3,5}$};
					\node[anchor=base,yshift=5mm] (n1) at (26.5/32,0) {$...$};
					\draw[<-] ([yshift=6/12mm]26/32,0) -- ([shift={(-0.5/32,0)}]t345.south);
					\node[anchor=base,yshift=5mm] (n1) at (27.5/32,0) {$...$};
					\draw[<-] ([yshift=6/12mm]27/32,0) -- ([shift={(0.5/32,0)}]t345.south);
					\node[anchor=base,yshift=10mm] (t365) at (30.5/32,0) {${t_{3,6}}\;t_{3,7}$};
					\node[anchor=base,yshift=5mm] (n1) at (30.5/32,0) {$...$};
					\draw[<-] ([yshift=6/12mm]30/32,0) -- ([shift={(-0.5/32,0)}]t365.south);
					\node[anchor=base,yshift=5mm] (n1) at (31.5/32,0) {$...$};
					\draw[<-] ([yshift=6/12mm]31/32,0) -- ([shift={(0.5/32,0)}]t365.south);
					\foreach \a in {0,1}{
						\fill[] (\a/2,0) circle (1/12pt);
						\foreach \b in {0,1}{
							\fill[] (\a/2+2/8+\b/8,0) circle (1/12pt);
							\foreach \c in {0,1}{
								\fill[] (\a/2+2/8+\b/8+2/32+\c/32,0) circle (1/12pt);
								\foreach \d in {0,1}{
									\fill[] (\a/2+2/8+\b/8+2/32+\c/32+2/128+\d/128,0) circle (1/12pt);
									\foreach \e in {0,1}{
									}
								}
							}
						}
					}
					\draw[thick,decorate,decoration = {
						calligraphic brace,
						raise=5pt,
						amplitude=5pt,
						aspect=0.5
					}] (1/4,0) -- (0,0) node[anchor=base,pos=0.5,yshift=-20pt,black,font=\tiny]{$2^{-2}$};
					\draw[thick,decorate,decoration = {
						calligraphic brace,
						raise=5pt,
						amplitude=5pt,
						aspect=0.5
					}] (3/4,0) -- (2/4,0) node[anchor=base,pos=0.5,yshift=-20pt,black,font=\tiny]{$2^{-2}$};
					\draw[thick,decorate,decoration = {
						calligraphic brace,
						raise=5pt,
						amplitude=5pt,
						aspect=0.5
					}] (5/16,0) -- (4/16,0) node[anchor=base,pos=0.5,yshift=-20pt,black,font=\tiny]{$2^{-4}$};
					\node[anchor=base,yshift=-20pt] (n1) at (5.5/16,0) {$...$};
					\draw[thick,decorate,decoration = {
						calligraphic brace,
						raise=5pt,
						amplitude=5pt,
						aspect=0.5
					}] (7/16,0) -- (6/16,0) node[anchor=base,pos=0.5,yshift=-20pt,black,font=\tiny]{$2^{-4}$};
					\node[anchor=base,yshift=-20pt] (n1) at (7.5/16,0) {$...$};
					\draw[thick,decorate,decoration = {
						calligraphic brace,
						raise=5pt,
						amplitude=5pt,
						aspect=0.5
					}] (13/16,0) -- (12/16,0) node[anchor=base,pos=0.5,yshift=-20pt,black,font=\tiny]{$2^{-4}$};
					\node[anchor=base,yshift=-20pt] (n1) at (13.5/16,0) {$...$};
					\draw[thick,decorate,decoration = {
						calligraphic brace,
						raise=5pt,
						amplitude=5pt,
						aspect=0.5
					}] (15/16,0) -- (14/16,0) node[anchor=base,pos=0.5,yshift=-20pt,black,font=\tiny]{$2^{-4}$};
					\node[anchor=base,yshift=-20pt] (n1) at (15.5/16,0) {$...$};
				\end{scope}
			\end{tikzpicture}
		\end{equation}
	\end{definition}
	\vspace{\baselineskip}
	
	Next, we construct the `fractal' vector fields exploiting the above.
	\begin{definition}[Fractal shear flow]\label{fracflow}
		Consider a sequence of tuples $\{(i_k, L_k, \tau_k)\}_{k \in \mathbb{N}}$ with $i_k \in \{1,2\}$, $L_k \in \mathbb{N}$, and $\tau_k > 0$.
		\begin{enumerate}[label=(\roman*)]
			\item We say $\{(i_k, L_k, \tau_k)\}_{k \in \mathbb{N}}$ satisfies the Finiteness Condition if for all $k\in\mathbb{N}$ we have
			\begin{equation}\label{eqfincond}
				\text{(Finiteness Condition)\quad} \tau_k < 2^{-2k}.
			\end{equation}
			\item We say $\{(i_k, L_k, \tau_k)\}_{k \in \mathbb{N}}$ satisfies the Cancellation Conditions if for all $k\in\mathbb{N}$ we have
			\begin{equation}\label{eqcancond}
				\text{(Cancellation Conditions)\quad}\begin{gathered}
					i_{k+1} \ne i_k, \\
					2\tau_{k+1} = \frac{1}{2L_k}, \\
					2L_{k+1}\tau_k \text{ an odd integer}.
				\end{gathered}
			\end{equation}
		\end{enumerate}
		
		Following the notation of Definition \ref{dyord}, in particular \eqref{eqtkm}, when $\{(i_k, L_k, \tau_k)\}_{k \in \mathbb{N}}$ satisfy the Finiteness Condition \eqref{eqfincond}, we may define for each $K \in \mathbb{N}$ the following fractal shear flows on $\mathbb{T}^2 \times [0, 1] \to \mathbb{R}^2$,
		\begin{gather*}
			u_K^{\{(i_k, L_k, \tau_k)\}_{k=1}^K}(x,t) = \begin{cases} u^{(i_k;L_k)}(x) & \text{if $t \in [t_{k,m}, t_{k,m} + \tau_k]$ for some $(k,m) \in \mathcal{D}_K$}, \\ 0 & \text{otherwise}, \end{cases} \\
			u_\infty^{\{(i_k, L_k, \tau_k)\}_{k \in \mathbb{N}}}(x,t) = \begin{cases} u^{(i_k;L_k)}(x) & \text{if $t \in [t_{k,m}, t_{k,m} + \tau_k]$ for some $(k,m) \in \mathcal{D}$}, \\ 0 & \text{otherwise}. \end{cases}
		\end{gather*}
		
		By \eqref{eqtkm}, this is well-defined since the Finiteness Condition \eqref{eqfincond} implies that the intervals $[t_{k,m}, t_{k,m} + \tau_k]$ are disjoint subsets of $[0,1]$.
	\end{definition}
	\vspace{\baselineskip}
	
	Next, we show how this construction exploits Proposition \ref{shearcanc} to create non-uniqueness of renormalised solutions to \eqref{eqTE}.
	\begin{proposition}[Fractal behaviour]\label{shearsolconv}
		We follow the notation introduced in Definitions \ref{dyord}, \ref{fracflow}.
		
		Let $f_0 \in L^\infty(\mathbb{T}^2)$, and suppose we have an infinite sequence of tuples $\{(i_k,L_k,\tau_k)\}_{k\in\mathbb{N}}$ with $i_k \in \{1,2\}$, $L_k \in \mathbb{N}$, and $\tau_k > 0$, satisfying the Finiteness Condition \eqref{eqfincond}. Then for each $K \in \mathbb{N}$ there exists a unique weak solution $f^K$ to \eqref{eqTE} along $u_K^{\{(i_k,L_k,\tau_k)\}_{k=1}^K}$ with initial data $f_0$. Moreover, $f^K\in (C^0L^1) \cap (L^\infty L^\infty)$, is a Lagrangian solution, and hence also a renormalised weak solution (Definitions \ref{renormalised}, \ref{lagrangian}).
		
		If in addition $\{(i_k,L_k,\tau_k)\}_{k\in\mathbb{N}}$ satisfy the Cancellation Conditions \eqref{eqcancond} then
		\begin{gather*}
			f^{2K}\xrightarrow{K\to\infty}f^{\mathrm{even}}, \\
			f^{2K+1}\xrightarrow{K\to\infty}f^{\mathrm{odd}},
		\end{gather*}
		with the above convergence in weak-$*$ $L^\infty L^\infty$, and strong in $L^p L^\infty$ for any $p\in[1,\infty)$. Moreover, the sequences $f^{2K}\in C^0L^1$ and $f^{2K+1}\in C^0L^1$ are eventually constant as $K\to\infty$ when restricted to the spatio-temporal domain $\mathbb{T}^2\times[t_{k,m},t_{k,m}+2^{-2k}]$ for any $(k,m)\in\mathcal{D}$, or to the spatio-temporal domain $\mathbb{T}^2\times\{1\}$. Furthermore, the limit functions $f^\mathrm{even}, f^\mathrm{odd}$ are renormalised weak solutions to \eqref{eqTE} along $u_\infty^{\{(i_k,L_k,\tau_k)\}_{k\in\mathbb{N}}}$ with initial data $f_0$.
		
		If $f_0$ is not constant, then $f^\mathrm{even}\ne f^\mathrm{odd}$, and in particular,
		\begin{equation}\label{eqnonuniq}
			\begin{gathered}
				f^{\mathrm{even}}(\cdot, 1) = f_0, \\
				f^{\mathrm{odd}}(\cdot, 1) = f_0 \circ y_{-2\tau_1}^{(i_1;L_1)}.
			\end{gathered}
		\end{equation}
	\end{proposition}
	\begin{proof}
		Recall the language and notation introduced in Definitions \ref{dyord}, \ref{fracflow}, as well as Definitions \ref{renormalised}, \ref{lagrangian} of renormalised and Lagrangian solutions to \eqref{eqTE}.
		
		By Theorems \ref{weakcont}, \ref{TEexistence}, there exists some weak solution $f^K \in C_{\mathrm{weak-}*}^0L^\infty$ to \eqref{eqTE} along $u_K^{\{(i_k,L_k,\tau_k)\}_{k=1}^K}$ with initial data $f_0$.
		
		We aim to show $f \in C^0 L^1$ with
		\begin{equation}\label{eqlagpiece}
			f^K(\cdot, t) = f_0 \circ \big(\tilde{y}_t^K\big)^{-1},
		\end{equation}
		for $\big\{\tilde{y}_t^K\big\}_{t\in[0,1]}$ a Lagrangian flow along $u_K^{\{(i_k,L_k,\tau_k)\}_{k=1}^K}$ which does not depend on $f_0$, thus proving uniqueness.
		
		From the definition of $u_K^{\{(i_k,L_k,\tau_k)\}_{k=1}^K}$, finiteness of the set $\mathcal{D}_K$, and disjointness of the intervals $[t_{k,m}, t_{k,m}+\tau_k]$ for all $(k,m) \in \mathcal{D}_K$, we may piecewise apply Proposition \ref{shearuniq} to $f^K$. We use that $f^K \in C_{\mathrm{weak-}*}^0L^\infty$ to glue together the pieces, and that any weak solution to \eqref{eqTE} along $u\equiv 0$ on an open time interval $I$ is a constant function of $t\in I$, see for example Theorem 3.1.4' in \cite{hormander2003}. This constructs for each $K \in \mathbb{N}$ a Lagrangian flow $\tilde{y}_t^K$ satisfying \eqref{eqlagpiece}, and moreover gives an expression for $\tilde{y}_t^K$ in terms of the Lagrangian shear flows $y_t^{(i_k;L_k)}$, see \eqref{eqlagint} below.
		
		For any $(k,m) \in \mathcal{D}_K$, define
		\begin{equation*}
			t_{\mathrm{suc}_K(k,m)} = \begin{cases}
				1 & \text{if $(k,m)$ is maximal in $(\mathcal{D}_K,<_\mathrm{time})$, else} \\
				t_{\mathcal{S}} & \text{for $\mathcal{S}\in\mathcal{D}_K$ the successor of $(k,m)$ in $(\mathcal{D}_K,<_\mathrm{time})$}.
			\end{cases}
		\end{equation*}
		
		Note from \eqref{eqtkm}, and the Finiteness Condition \eqref{eqfincond}, that
		\begin{equation}\label{eqsucbound2}
			\tau_k < 2^{-2k} \le t_{\mathrm{suc}_K(k,m)} - t_{k,m}.
		\end{equation}
		
		Therefore, for each $t \in [t_{k,m},t_{\mathrm{suc}_K(k,m)}]$,
		\begin{equation}\label{eqlagint}
			\tilde{y}_t^K \circ\big(\tilde{y}_{t_{k,m}}^K\big)^{-1} = \begin{cases}
				y_{t-t_{k,m}}^{(i_k;L_k)} & \text{if } t \in [t_{k,m},t_{k,m}+\tau_k], \\
				y_{\tau_k}^{(i_k;L_k)} & \text{if } t \in [t_{k,m}+\tau_k,t_{\mathrm{suc}_K(k,m)}].
			\end{cases}
		\end{equation}
		
		This expression defines $\tilde{y}_t^K$ for all $t \in [0,1]$, since $t_{1,0} = 0$ and $\tilde{y}_0^K = \mathrm{Id}$. This completes the proof of uniqueness.
		
		Next, by considering \eqref{eqdyord}, \eqref{eqtkm}, we may show that (see the illustration \eqref{eqdydiag})
		\begin{equation}\label{eqsuc}
			t_{\mathrm{suc}_K(k,m)} = \begin{cases}
				t_{k,m}+2^{-2k} & \text{if } k < K, \\
				t_{k,m} + 2^{1-2k} & \text{if } k=K. \\
			\end{cases}
		\end{equation}
		
		Moreover, for all $(k,m) \in \mathcal{D}$,
		\begin{equation}\label{eqfracstruc}
			\begin{aligned}
				t_{k,m}+2^{-2k} & = t_{k+1,2m}, \\
				t_{k,m}+2^{1-2k} & = t_{k,m+1}\quad \text{when $m$ even}.
			\end{aligned}
		\end{equation}
		
		Assume that additionally $\{(i_k,L_k,\tau_k)\}_{k\in\mathbb{N}}$ satisfy the Cancellation Conditions \eqref{eqcancond}.
		
		Fix some $K \in \mathbb{N}$. We claim that, for each $(k,m) \in \mathcal{D}_K$, and for all $\left.t \in [t_{k,m},t_{k,m}+2^{-2k}]\right.$, one has as maps $\mathbb{T}^2\to\mathbb{T}^2$,
		\begin{equation}\label{eqparity}
			\tilde{y}_t^{K+2} = \tilde{y}_t^K.
		\end{equation}
		
		By \eqref{eqsucbound2}, \eqref{eqlagint}, it is sufficient to show, for all $(k,m)\in\mathcal{D}_K$, that as maps $\mathbb{T}^2\to\mathbb{T}^2$,
		\begin{equation*}
			\tilde{y}_{t_{\mathrm{suc}_K(k,m)}}^{K+2} \circ \big(\tilde{y}_{t_{\mathrm{suc}_{K+2}(k,m)}}^{K+2}\big)^{-1} = \tilde{y}_{t_{\mathrm{suc}_K(k,m)}}^K \circ \big(\tilde{y}_{t_{\mathrm{suc}_{K+2}(k,m)}}^K\big)^{-1}.
		\end{equation*}
		
		By \eqref{eqsuc}, for all $k<K$ this is immediate. Meanwhile, by \eqref{eqsuc}, \eqref{eqfracstruc}, for $k=K$, we may rewrite this, and now need to prove that for all $m \in \mathbb{Z}$ with $0\le m<2^K$,
		\begin{equation*}
			\tilde{y}_{t_{K+1,2m}+2^{-2K}}^{K+2} \circ \big(\tilde{y}_{t_{K+1,2m}}^{K+2}\big)^{-1} = \tilde{y}_{t_{K+1,2m}+2^{-2K}}^{K} \circ \big(\tilde{y}_{t_{K+1,2m}}^{K}\big)^{-1}.
		\end{equation*}
		
		In particular, it is certainly sufficient to show, for any $(k,m) \in \mathcal{D}_{K+1}$ with $m$ even, that as maps $\mathbb{T}^2\to\mathbb{T}^2$,
		\begin{equation}\label{eqshearcanc}
			\tilde{y}_{t_{k,m}+2^{2-2k}}^{K} \circ \big(\tilde{y}_{t_{k,m}}^{K}\big)^{-1} \begin{aligned}[t]
				& = \begin{cases}
					y_{2\tau_k}^{(i_k;L_k)} & \text{if } k = K\; (\mathrm{mod}\;2), \\
					\mathrm{Id} & \text{otherwise}.
				\end{cases} \\
				& = : Y_{k,m}^K,
			\end{aligned}
		\end{equation}
		where we have defined the shorthand $Y_{k,m}^K:\mathbb{T}^2\to\mathbb{T}^2$.
		
		Fixing an even $m\in\mathbb{Z}$ with $0\le m < 2^{K+1}$, we prove this by induction. For $k=K+1$, by \eqref{eqsucbound2}, \eqref{eqlagint} we see that as maps $\mathbb{T}^2\to\mathbb{T}^2$,
		\begin{equation*}
			\tilde{y}_{t_{\mathrm{suc}_K(k,m)}}^K \circ \big(\tilde{y}_{t_{\mathrm{suc}_{K+2}(k,m)}}^K\big)^{-1} = \mathrm{Id}.
		\end{equation*}
		After rewriting this using \eqref{eqsuc}, \eqref{eqfracstruc}, this proves \eqref{eqshearcanc} for $k=K+1$.
		
		Now let $k\le K$. By \eqref{eqfracstruc}, and since by assumption $m$ is even, we may rewrite (it may be helpful to consider the illustration \eqref{eqdydiag})
		\begin{equation*}
			Y_{k,m}^K = \begin{aligned}[t]
				& \left(\tilde{y}_{t_{k+1,2m+2}+2^{-2k}}^{K} \circ \big(\tilde{y}_{t_{k+1,2m+2}}^{K}\big)^{-1}\right) \\
				& \circ \left(\tilde{y}_{t_{k,m+1}+2^{-2k}}^{K} \circ \big(\tilde{y}_{t_{k,m+1}}^{K}\big)^{-1}\right) \\
				& \circ \left(\tilde{y}_{t_{k+1,2m}+2^{-2k}}^{K} \circ \big(\tilde{y}_{t_{k+1,2m}}^{K}\big)^{-1}\right) \\
				& \circ \left(\tilde{y}_{t_{k,m}+2^{-2k}}^{K} \circ \big(\tilde{y}_{t_{k,m}}^{K}\big)^{-1}\right).
			\end{aligned}
		\end{equation*}
		
		Then by the definition of $Y_{k+1,m'}^K$, \eqref{eqsucbound2}, and \eqref{eqlagint}, the right hand side may be rewritten again as
		\begin{equation*}
			Y_{k,m}^K = Y_{k+1,2m+2}^K \circ y_{\tau_k}^{(i_k;L_k)} \circ Y_{k+1,2m}^K \circ y_{\tau_k}^{(i_k;L_k)}.
		\end{equation*}
		
		By Proposition \ref{shearcanc}, the result \eqref{eqshearcanc} for $k$ now follows from the same result for $k+1$. This completes the induction, and thus the proof of \eqref{eqshearcanc}, and so also the proof of \eqref{eqparity}.
		
		Next, note by \eqref{eqtkm}, that
		\begin{align*}
			E & = \left(\bigcup_{(k,m)\in\mathcal{D}} [t_{k,m},t_{k,m}+2^{-2k}]\right) \cup \left\{t_{k,m}+2^{2-2k}:(k,m)\in\mathcal{D} \text{ with $m$ even}\right\} \\
			& \subset [0,1],
		\end{align*}
		has Lebesgue-measure 1. Moreover, by \eqref{eqparity}, \eqref{eqshearcanc}, the sequences $f^{2K}$, $f^{2K+1}$ are eventually constant on $\mathbb{T}^2 \times [t_{k,m},t_{k,m}+2^{-2k}]$, and each $\mathbb{T}^2 \times \{t_{k,m}+2^{2-2k}\}$ with $m$ even.
		
		Therefore, for all $t\in E$, and so for a.e. $t \in [0,1]$, we see that
		\begin{equation}\label{eqpointconv}
			\begin{gathered}
				f^{2K}(\cdot, t) \xrightarrow{K\to\infty} f_0 \circ \big(\tilde{y}^\mathrm{even}_t\big)^{-1}, \\
				f^{2K+1}(\cdot, t) \xrightarrow{K\to\infty} f_0 \circ \big(\tilde{y}^\mathrm{odd}_t\big)^{-1},
			\end{gathered}
		\end{equation}
		for some Lebesgue-measure preserving $\{\tilde{y}^\mathrm{even}_t\}_{t\in E}$, $\{\tilde{y}^\mathrm{odd}_t\}_{t\in E}$ independent of $f_0$, with the convergence strong in $L^\infty(\mathbb{T}^2)$.
		
		Moreover, by the dominated convergence Theorem, we see that \eqref{eqpointconv} converges strongly in $L^p(E;L^\infty(\mathbb{T}^2))$ for all $p\in[1,\infty)$, and hence also in weak-$*$ $L^\infty(E;L^\infty(\mathbb{T}^2))$.
		
		By Definition \ref{fracflow}, for all $t\in[0,1]$ (in particular $t\in E$), $u_K^{\{(i_k, L_k, \tau_k)\}_{k=1}^K}(\cdot, t)$ is bounded by 1 in $L^\infty(\mathbb{T}^2;\mathbb{R}^2)$, and as $K\to\infty$ the sequence is eventually equal to $u_\infty^{\{(i_k, L_k, \tau_k)\}_{k \in \mathbb{N}}}(\cdot,t)$. Therefore it too converges strongly in $L^p(E;L^\infty(\mathbb{T}^2;\mathbb{R}^2))$ for all $p\in[1,\infty)$.
		
		Fix $t \in E$, and apply the Trace Formula \eqref{eqtrace} to $f^{2K}$, $f^{2K+1}$. Taking now the limit $K\to\infty$ then gives that for all $t\in E$, for any $\phi\in C^\infty(\mathbb{T}^d\times[0,T])$,
		\begin{gather*}
			\begin{aligned}
				& \int_{\mathbb{T}^2} f_0 \circ \big(\tilde{y}^\mathrm{even}_t\big)^{-1} \phi(\cdot, t) \;dx \\
				& = \int_{\mathbb{T}^2} f_0 \phi_0 \; dx + \int_{\mathbb{T}^2\times\left(E\cap[0,t]\right)} f_0 \circ \big(\tilde{y}^\mathrm{even}_t\big)^{-1}\left(\frac{\partial\phi}{\partial t}+u_\infty^{\{(i_k, L_k, \tau_k)\}_{k \in \mathbb{N}}}\cdot\nabla\phi +\nu\Delta \phi\right) dxdt,
			\end{aligned} \\
			\begin{aligned}
				& \int_{\mathbb{T}^2} f_0 \circ \big(\tilde{y}^\mathrm{odd}_t\big)^{-1} \phi(\cdot, t) \;dx \\
				& = \int_{\mathbb{T}^2} f_0 \phi_0 \; dx + \int_{\mathbb{T}^2\times\left(E\cap[0,t]\right)} f_0 \circ \big(\tilde{y}^\mathrm{odd}_t\big)^{-1}\left(\frac{\partial\phi}{\partial t}+u_\infty^{\{(i_k, L_k, \tau_k)\}_{k \in \mathbb{N}}}\cdot\nabla\phi +\nu\Delta \phi\right) dxdt,
			\end{aligned}
		\end{gather*}
		where $\phi_0=\phi(\cdot,0)$.
		
		This implies that $f_0 \circ \big(\tilde{y}^\mathrm{even}_t\big)^{-1}, f_0 \circ \big(\tilde{y}^\mathrm{odd}_t\big)^{-1} \in C^0_{\mathrm{weak-}*}(E;L^\infty(\mathbb{T}^2))$ and so may be extended to $f^\mathrm{even}, f^\mathrm{odd}\in C^0_{\mathrm{weak-}*}([0,1];L^\infty)$, as argued in Theorem \ref{weakcont}.
		
		For any $\phi \in C_c^\infty(\mathbb{T}^2\times[0,1))$ let $t=1$ in the above (noting that $1=t_{1,0}+2^{2-2}\in E$). This proves that $f^\mathrm{even}$, $f^\mathrm{odd}$ are both weak solutions to \eqref{eqTE} along $u_\infty^{\{(i_k, L_k, \tau_k)\}_{k \in \mathbb{N}}}$ with initial data $f_0$.
		
		Moreover, for any $\beta \in C_b^0(\mathbb{R})$, we may rewrite $\beta\big(f_0\circ\big(\tilde{y}^\mathrm{even}_t\big)^{-1}\big) = \beta(f_0)\circ\big(\tilde{y}^\mathrm{even}_t\big)^{-1}$. By repeating the above arguments we have also that $\beta(f_0)\circ\big(\tilde{y}^\mathrm{even}_t\big)^{-1}$ (can be extended to) a weak solution to \eqref{eqTE} along $u_\infty^{\{(i_k, L_k, \tau_k)\}_{k \in \mathbb{N}}}$ with initial data $\beta(f_0)$. Therefore, we see that $f^\mathrm{even}$ is a renormalised weak solution to \eqref{eqTE} along $u_\infty^{\{(i_k, L_k, \tau_k)\}_{k \in \mathbb{N}}}$ with initial data $f_0$. Similarly for $f^\mathrm{odd}$.
		
		Next, observe that by \eqref{eqshearcanc}, \eqref{eqpointconv}, for all $(k,m)\in\mathcal{D}$ with $m$ even, that
		\begin{equation}\label{eqparitystruc}
			\begin{gathered}
				f^\mathrm{even}(\cdot,t_{k,m}+2^{2-2k}) = f^\mathrm{even}(\cdot,t_{k,m}) \circ \begin{cases}
					\big(y_{2\tau_k}^{(i_k;L_k)}\big)^{-1} & \text{for even } k, \\
					\mathrm{Id} & \text{for odd } k.
				\end{cases} \\
				f^\mathrm{odd}(\cdot,t_{k,m}+2^{2-2k}) = f^\mathrm{odd}(\cdot,t_{k,m}) \circ \begin{cases}
					\mathrm{Id} & \text{for even } k, \\
					\big(y_{2\tau_k}^{(i_k;L_k)}\big)^{-1} & \text{for odd } k.
				\end{cases}
			\end{gathered}
		\end{equation}
		
		In particular, since $t_{1,0}=0$, we have proved \eqref{eqnonuniq}.
		
		It remains to show, when $f_0$ is not constant, that from \eqref{eqparitystruc} we deduce $f^\mathrm{even} \ne f^\mathrm{odd}$. Assume to the contrary that they are equal and call this solution $f$. Then by \eqref{eqparitystruc} (noting that by \eqref{eqdyord}, \eqref{eqtkm}, $1-2^{2-2k} = t_{k,2^k-2}$, see for example the illustration \eqref{eqdydiag}), we have for all $k\in\mathbb{N}$, that
		\begin{gather*}
			f(\cdot, 1) = f(\cdot,1-2^{2-2k}), \\
			f(\cdot, 1) = f(\cdot,1-2^{2-2k}) \circ \big(y_{2\tau_k}^{(i_k;L_k)}\big)^{-1}.
		\end{gather*}

		In particular, taking $k=1$, we see $f(\cdot,1)=f_0$. Substituting this back in gives $f_0 = f(\cdot,1-2^{2-2k})$ for all $k\in\mathbb{N}$. Again, substituting this back in gives that for all $k\in\mathbb{N}$,
		\begin{equation*}
			f_0 = f_0 \circ y_{2\tau_k}^{(i_k;L_k)}.
		\end{equation*}
		
		In terms of the unit vectors $e_1$, $e_2$, we have that for a.e. $x\in\mathbb{T}^2$, $f_0(x) = f_0 (x+2\tau_k e_{i_k})$. Convolving with a smooth function $\rho \in C_c^\infty(\mathbb{R}^2)$ then gives that for all $x \in \mathbb{T}^2$,
		\begin{equation}\label{eqconst}
			(\rho * f_0)(x) = (\rho * f_0)(x+2\tau_k e_{i_k}).
		\end{equation}
		
		By the Finiteness and Cancellation Conditions \eqref{eqfincond}, \eqref{eqcancond}, we see that $\tau_k \xrightarrow{k\to\infty} 0$, and $i_{k+1}\ne i_k$ (so equal to both $1$ and $2$ infinitely often). Now $(\rho * f_0) \in C^\infty(\mathbb{T}^2)$, and so by \eqref{eqconst} we must have that $\rho*f_0$ is a constant. This holds for all $\rho \in C_c^\infty(\mathbb{R}^2)$, and therefore implies also that $f_0$ is a constant, reaching the required contradiction.
	\end{proof}
	\vspace{\baselineskip}
	
	Finally, for a suitably fast growing sequence $\{L_k\}_{k\in\mathbb{N}}\subset\mathbb{N}$, we apply Theorem \ref{vcontrol} to the sequence of vector fields $u_n^{\{(i_k,L_k,\tau_k)\}_{k=1}^n}(x,t)$. This allows us to control the vanishing viscosity limit along the vector field $u_\infty^{\{(i_k,L_k,\tau_k)\}_{k\in\mathbb{N}}}$.
	\begin{theorem}[Non-unique renormalised vanishing viscosity limit solutions]\label{vvtheorem}
		There exists a divergence-free vector field $u \in L^\infty([0,1];L^\infty(\mathbb{T}^2;\mathbb{R}^2))$, and a sequence $\{\nu_n\}_{n\in\mathbb{N}}$ with $\nu_n > 0$ and $\nu_n\xrightarrow{n\to\infty}0$, such that for any initial data $f_0\in  L^\infty(\mathbb{T}^2)$, and for $f^\nu$ the unique solution to \eqref{eqADE} along $u$ with initial data $f_0$, one has
		\begin{gather*}
			f^{\nu_{2n}}\xrightarrow{n\to\infty}f^\mathrm{even}, \\
			f^{\nu_{2n+1}}\xrightarrow{n\to\infty}f^\mathrm{odd},
		\end{gather*}
		with the above convergence in weak-$*$ $L^\infty([0,1];L^\infty(\mathbb{T}^2))$, and strong in $L^p([0,1];L^p(\mathbb{T}^2))$ for all $p \in [1,\infty)$. Furthermore, the limit functions $f^\mathrm{even}, f^\mathrm{odd}$ are renormalised weak solutions to \eqref{eqTE} along $u$ with initial data $f_0$.
		
		If $f_0$ is not constant, then $f^\mathrm{even}\ne f^\mathrm{odd}$, and moreover the set of weak-$*$ limit points of $f^\nu\in L^\infty([0,1];L^\infty(\mathbb{T}^2))$ as $\nu\to0$ is uncountable.
	\end{theorem}
	\begin{proof}
		Recall the language and notation introduced in Definitions \ref{dyord}, \ref{fracflow}, as well as Definitions \ref{renormalised}, \ref{lagrangian} of renormalised and Lagrangian solutions to \eqref{eqTE}.
		
		Consider any infinite sequence of tuples $\{(i_k, L_k, \tau_k)\}_{k \in \mathbb{N}}$ with $i_k \in \{1,2\}$, $L_k \in \mathbb{N}$, $\tau_k >0$, and satisfying the Finiteness Condition \eqref{eqfincond}.
		
		By Proposition \ref{shearsolconv} there exists a Lagrangian solution $f^n\in C^0L^1$, unique in the class of all weak solutions, to \eqref{eqTE} along $u_n^{\{(i_k,L_k,\tau_k)\}_{k=1}^n}$ for any initial data $f_0 \in L^\infty(\mathbb{T}^2)$. Moreover, $u_n^{\{(i_k,L_k,\tau_k)\}_{k=1}^n}$ is bounded by $1$ in $L^\infty([0,1];L^\infty(\mathbb{T}^2;\mathbb{R}^2))$.
		
		Let $d_*$ be a metric inducing the weak-$*$ topology on
		\begin{equation*}
			X = \left\{u\in L^\infty \left( [0,1];L^\infty(\mathbb{T}^2;\mathbb{R}^2)\right) : \left\|u\right\|_{L^\infty L^\infty}\le 1 \right\}.
		\end{equation*}
		
		Let $f_0 \in L^\infty(\mathbb{T}^2)$, and denote for each $n\in\mathbb{N}$, $\nu > 0$, by $f^{n,\nu}$, respectively $f^n$, the unique weak solution to \eqref{eqADE}, respectively \eqref{eqTE}, along $u_n$ with initial data $f_0$. Moreover denote by $f^{\infty,\nu}$ the unique weak solution to \eqref{eqADE} along $u_\infty^{\{(i_k,L_k,\tau_k)\}_{k\in\mathbb{N}}}$ with initial data $f_0$.
		
		Then by Theorem \ref{vcontrol},
		\begin{enumerate}[label=\textbf{S.\arabic*},ref=S.\arabic*]
			\item \label{listnu2} For all $n\in\mathbb{N}$, there exists $\nu_n>0$, $\epsilon_n>0$ depending only on $\{(i_k,L_k,\tau_k)\}_{k=1}^n$ (and in particular not on $f_0$), with $\nu_n\xrightarrow{n\to\infty}0$ monotonically, such that the following holds true:
			\setcounter{enumi}{4}
			\item \label{listconv2} If $d_*(u_{n+1},u_n) \le \epsilon_n$ for all $n\in\mathbb{N}$, then for all $p \in [1,\infty)$
			\begin{equation*}
				\left\|f^{\infty,\nu_n}-f^n\right\|_{L^\infty L^p}\xrightarrow{n\to\infty}0.
			\end{equation*}
		\end{enumerate}
		
		We now construct such a sequence $\{(i_k,L_k,\tau_k)\}_{k\in \mathbb{N}}$. Let $(i_1, L_1, \tau_1)=(1,2^2,2^{-2})$. We proceed by induction on $n\in\mathbb{N}$. Assume $\{(i_k,L_k,\tau_k)\}_{k=1}^n$ are given, and satisfy for $k\in\{1,...,n\}$.
		\begin{equation*}
			\begin{gathered}
				i_k \in \{1,2\}, \\
				L_k \in \mathbb{N}, \\
				L_k \ge 2^{2k}, \\
				0 < \tau_k \le 2^{-2k},
			\end{gathered}
		\end{equation*}
		
		It is straightforward to check that this is satisfied for the base case $(i_1, L_1, \tau_1)=(1,2^2,2^{-2})$.
		
		If $n \ge 2$, we assume in the inductive hypothesis that also, for $k\in\{1,...,n-1\}$,
		\begin{equation*}
			\begin{gathered}
				i_{k+1} \ne i_k, \\
				2\tau_{k+1} = \frac{1}{2L_k}, \\
				2L_{k+1}\tau_{k}\; \text{is an odd integer}, \\
				d_*\left(u_{k+1}^{\{(i_{l},L_{l},\tau_{l})\}_{l=1}^{k+1}},u_{k}^{\{(i_{l},L_{l},\tau_{l})\}_{l=1}^{k}}\right) \le \epsilon_{k}.
			\end{gathered}
		\end{equation*}
		
		We then choose $(i_{n+1}, L_{n+1}, \tau_{n+1})$ as follows. Let $i_{n+1}\in\{1,2\}\setminus\{i_n\}$. Let $\tau_{n+1}=\frac{1}{4L_n}$ which by the inductive hypothesis satisfies $\tau_{n+1}\le 2^{-2k-2}$. Let
		\begin{equation}\label{eqLn}
			L_{n+1}=2L_{n-1}(2M+1),
		\end{equation}
		for some large $M\in\mathbb{N}$ to be chosen, where when $n=1$ we take $L_0 = 1$. This will ensure $2L_{n+1}\tau_n$ is an odd integer.
		
		By \eqref{eqshearweakconv} and Definition \ref{fracflow}, we see for $n\in\mathbb{N}$ fixed, that
		\begin{equation*}
			d_*\left(u_{n+1}^{\{(i_k,L_k,\tau_k)\}_{k=1}^{n+1}},u_n^{\{(i_k,L_k,\tau_k)\}_{k=1}^n}\right)\xrightarrow{(L_{n+1})\to\infty}0.
		\end{equation*}
		
		Therefore, by taking $M\in\mathbb{N}$ large enough in \eqref{eqLn}, we have that both $L_{n+1}\ge 2^{2n+2}$, and
		\begin{equation*}
			d_*\left(u_{n+1}^{\{(i_k,L_k,\tau_k)\}_{k=1}^{n+1}},u_n^{\{(i_k,L_k,\tau_k)\}_{k=1}^n}\right) \le \epsilon_n,
		\end{equation*}
		completing the induction.
		
		We have constructed $\{(i_k, L_k, \tau_k)\}_{k \in \mathbb{N}}$ satisfying the Finiteness and Cancellation Conditions \eqref{eqfincond}, \eqref{eqcancond}, and also \eqref{listnu2}, \eqref{listconv2}. The main statement of the theorem is now a straightforward corollary of Proposition \ref{shearsolconv}.
		
		Finally we show that, if $f_0$ is not constant, then the set of weak-$*$ limit points of $f^\nu\in L^\infty([0,1];L^\infty(\mathbb{T}^2))$ as $\nu\to0$ is uncountable. Working in the weak-$*$ topology of $L^\infty L^\infty$, the set of vanishing viscosity limit points is bounded, and so is a metric space. Moreover, the map $\nu\to f^\nu$ is continuous, which implies that the set of limit points is connected. If it is a connected metric space, then it is either a singleton or uncountable. Therefore, we conclude by observing that $f^\mathrm{even}\ne f^\mathrm{odd}$ are at least two limit points.
	\end{proof}
	
	\section{Inadmissibility}\label{Inadmissibility}
	\subsection{Notation}
	We continue with the same notation for the 2-torus introduced in Section \ref{torusnotation}. However, we no longer make use of the notation in Definitions \ref{shearflow}, \ref{dyord}, \ref{fracflow}. In addition, we define the binary expansion of some $x=(x_1,x_2) \in \mathbb{T}^2$ as follows. For the representative $(x_1,x_2)\in[0,1)^2$ denote for $i\in\{1,2\}$, $k\in\mathbb{N}$, by $x_{i,k}$ the $k$\textsuperscript{th} binary digit of the $i$\textsuperscript{th} coordinate of $x$. That is, for each $i \in \{1,2\}$,
	\begin{equation}\label{eqbinexp}
		x_i = \sum_{k=1}^\infty x_{i,k} 2^{-k},
	\end{equation}
	where $x_{i,k} \in \{0,1\}$ and $x_{i,k} \centernot{\xrightarrow{k\to\infty}} 1$ (as is standard to ensure uniqueness of the binary expansion).
	\vspace{\baselineskip}
	
	First, we define vector fields and corresponding Lagrangian flows, which swap points in $\mathbb{T}^2$ according to their binary expansion. These `binary swaps' form the building block of our construction. Subsequently we shall give a divergence-free vector field in $L^\infty([0,100];L^\infty(\mathbb{T}^2;\mathbb{R}^2))$ with $L^\infty L^\infty$-norm equal to 1, which perfectly mixes the transported scalar to its spatial average, and subsequently unmixes, any initial data to \eqref{eqTE}. Our aim is then to show that this behaviour is the unique limit point of vanishing viscosity of the associated solution to \eqref{eqADE}. In order to ensure uniqueness of the limit points it is necessary to `gradually' perform these binary swaps, that is they must be restricted to gradually smaller regions of space, see \eqref{eqgradualmixing} below.
	
	\begin{definition}[Lagrangian binary swap]\label{lagswap}
		Suppose $i \in \{1,2\}$, $k \in \mathbb{N}$, $n\in\{1,...,2^{\left\lfloor k/2\right\rfloor}\}$, $L \in \mathbb{N}$, with $L \ge k+1$.
		
		The proof below will define a time-dependent, divergence-free vector field called the $(i,k,n;L)$-binary swap
		\begin{equation*}
			u^{(i,k,n;L)} : \mathbb{T}^2 \times [0,3\cdot2^{-k}] \to \mathbb{R}^2,
		\end{equation*}
		and a corresponding Lagrangian flow map (Definition \ref{lagrangian}) $\big\{y^{(i,k,n;L)}_t\big\}_{t \in [0,3\cdot2^{-k}]}$ with the properties \eqref{eqbinaryswap}-\eqref{eqrectangle} below.
		
		Define $J_{k,n}=\left[(n-1)2^{-\left\lfloor k/2\right\rfloor},n2^{-\left\lfloor k/2\right\rfloor}\right]\subset\mathbb{T}$. Then at time $t=3\cdot2^{-k}$, for a.e. $x=(x_1,x_2) \in \mathbb{T}^2$, $y^{(i,k,n;L)}_{3\cdot2^{-k}}$ will swap the $k$\textsuperscript{th} and $(k+1)$\textsuperscript{th} binary digits of $x_i$ if $x_i\in J_{k,n}$.
		
		That is $y_0^{(i,k,n;L)} = \mathrm{Id}$. For a.e. $x=(x_1,x_2)\in\mathbb{T}^2$ denote by $(x'_1,x'_2)=y^{(i,k,n;L)}_{3\cdot2^{-k}}(x)$. Following the notation for binary expansions in \eqref{eqbinexp}, for $j\in\{1,2\}$ the coordinate, and for $l\in\mathbb{N}$ the binary digit,
		\begin{equation}\label{eqbinaryswap}
			\begin{aligned}
				x'_j & = x_j \; \text{if} \; j\neq i, \\
				x'_{i,l} & = x_{i,l} \; \text{for} \; l\notin\{k,k+1\}, \\
				x'_{i,k+1} & = \begin{cases}
					x_{i,k} & \text{if } x_i \in J_{k,n}, \\
					x_{i,k+1} & \text{otherwise},
				\end{cases} \\
				x'_{i,k} & = \begin{cases}
					x_{i,k+1} & \text{if } x_i \in J_{k,n}, \\
					x_{i,k} & \text{otherwise}.
				\end{cases}
			\end{aligned}
		\end{equation}
		
		Additionally, the vector field $u^{(i,k,n;L)}$ will satisfy
		\begin{equation}\label{equnifbound}
			\left\|u^{(i,k,n;L)}\right\|_{L^\infty L^\infty} \le 1,
		\end{equation}
		and for all $t\in[0,3\cdot2^{-k}]$, and all $x=(x_1,x_2)\in\mathbb{T}^2$ with $x_i \notin J_{k,n}$,
		\begin{equation}\label{eqgradualmixing}
			u^{(i,k,n;L)}(x,t)=0.
		\end{equation}
		
		Moreover, for $i\in\{1,2\}$, $k\in\mathbb{N}$ fixed, as $L \to \infty$,
		\begin{equation}\label{eqweakconv}
			u^{(i,k,n;L)} \xrightharpoonup{L \to \infty} 0, \; \text{in weak-$*$ $L^\infty L^\infty$}.
		\end{equation}
		
		Finally, for any $r, r' \in \{1,...,2^{k-1}\}$, for the spatial intervals $J=[(r-1)2^{1-k},r2^{1-k}]$, and $J'=[(r'-1)2^{1-k},r'2^{1-k}]$, the Lagrangian-flow $y^{(i,k,n;L)}_t$ preserves the squares
		\begin{equation}\label{eqrectangle}
			y^{(i,k,n;L)}_t:J\times J'\leftrightarrow J\times J'.
		\end{equation}
		\begin{proof}
			We now construct the above vector field and Lagrangian flow. We shall give the construction for $i=1$ and then define its coordinate reflection for $i=2$. That is, for each $(x_1,x_2) \in \mathbb{T}^2$, $t \in [0,3\cdot2^{-k}]$
			\begin{gather}
				u^{(2,k,n;L)}((x_1,x_2),t) = u^{(1,k,n;L)}((x_2,x_1),t), \label{eqi2u} \\
				y_t^{(2,k,n;L)}((x_1,x_2)) = y_t^{(1,k,n;L)}((x_2,x_1)). \label{eqi2y}
			\end{gather}
			so that \eqref{equnifbound}, \eqref{eqweakconv}, \eqref{eqbinaryswap} with $i=2$ follow from the same result for $i=1$.
			
			We shall achieve the required binary-swaps \eqref{eqbinaryswap} by piecing together particular rotating vector fields which perform half rotations of rectangular regions of $\mathbb{T}^2$. Suppose $W,H>0$, and define the following 1-Lipschitz stream-function
			\begin{gather*}
				\psi:\left(-\frac{W}{2},\frac{W}{2}\right)\times\left(-\frac{H}{2},\frac{H}{2}\right)\subset\mathbb{R}^2 \to \mathbb{R}, \\
				\psi((x_1,x_2)) = \min\{W,H\}\cdot\max\left\{\left(\frac{x_1}{W}\right)^2,\left(\frac{x_2}{H}\right)^2\right\}.
			\end{gather*}
			
			This defines a 1-bounded, time-independent, divergence-free, vector field $\nabla^\perp\psi = \big(-\frac{\partial\psi}{\partial x_2}, \frac{\partial\psi}{\partial x_1}\big)$ on the open rectangle $\left(-\frac{W}{2},\frac{W}{2}\right)\times\left(-\frac{H}{2},\frac{H}{2}\right)$. $\nabla^\perp\psi$ has rectangular flow lines, as illustrated in the following diagram:
			
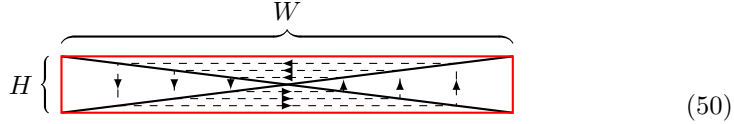
\begin{equation}\label{eqrectflow}
				\begin{tikzpicture}[scale=3/2]
					\pgfdeclarelayer{background}
					\pgfsetlayers{background,main}
					\begin{scope}[scale=1/2,decoration={
							markings,
							mark=at position 1/36 with {\arrowreversed[xshift=-2pt]{latex[length=4pt]}},
							mark=at position 10/36 with {\arrowreversed[xshift=-2pt]{latex[length=4pt]}},
							mark=at position 19/36 with {\arrowreversed[xshift=-2pt]{latex[length=4pt]}},
							mark=at position 28/36 with {\arrowreversed[xshift=-2pt]{latex[length=4pt]}},
						}]
						\draw[dashed,postaction={decorate}] (3,3/8) rectangle (5,5/8);
						\draw[dashed,postaction={decorate}] (2,1/4) rectangle (6,3/4);
						\draw[dashed,postaction={decorate}] (1,1/8) rectangle (7,7/8);
						\draw[thick] (0,0) -- (8,1);
						\draw[thick,] (0,1) -- (8,0);
						\draw[thick,red] (0,0) rectangle (8,1);
						\coordinate (n3) at (0,0);
						\coordinate (n4) at (0,1);
						\coordinate (n5) at (8,0);
						\coordinate (n6) at (8,1);
						\draw[thick,decorate,decoration = {
							calligraphic brace,
							raise=5pt,
							amplitude=5pt,
							aspect=0.5
						}] (0,1) -- (8,1) node[pos=0.5,above=10pt,black]{$W$};
						\draw[thick,decorate,decoration = {
							calligraphic brace,
							raise=5pt,
							amplitude=2.5pt,
							aspect=0.5
						}] (0,0) -- (0,1) node[pos=0.5,left=7pt,black]{$H$};
					\end{scope}
				\end{tikzpicture}
			\end{equation}
			
			Moreover, on each triangular segment, $\nabla^\perp\psi$ is a linear function of space, and each flow line within each segment has a time period equal to $\max\{W,H\}$.
			
			Therefore, $\nabla^\perp\psi$ admits a Lagrangian flow $\big\{y_t\big\}_{t\in(-\infty,\infty)}$ on $\left(-\frac{W}{2},\frac{W}{2}\right)\times\left(-\frac{H}{2},\frac{H}{2}\right)$. That is for all $t\in(-\infty,\infty)$, $y_t:\left(-\frac{W}{2},\frac{W}{2}\right)\times\left(-\frac{H}{2},\frac{H}{2}\right)\leftrightarrow \left(-\frac{W}{2},\frac{W}{2}\right)\times\left(-\frac{H}{2},\frac{H}{2}\right)$ is a Lebesgue-measure preserving bijection, $y_0= \mathrm{Id}$, and for all $x\in\left(-\frac{W}{2},\frac{W}{2}\right)\times\left(-\frac{H}{2},\frac{H}{2}\right)$, and all $t\in(-\infty,\infty)$, $\frac{d}{dt}y_t(x)=\nabla^\perp\psi(y_t(x))$.
			
			Moreover, we have $y_0= \mathrm{Id}$, while $y_{2\max\{W,H\}}$ is exactly a half rotation of $\left(-\frac{W}{2},\frac{W}{2}\right)\times\left(-\frac{H}{2},\frac{H}{2}\right)$ around its centre. Finally, both $y_t$ and $y_t^{-1}$ are Lipschitz in both $x$ and $t$. That is, for all $x,x' \in \left(-\frac{W}{2},\frac{W}{2}\right)\times\left(-\frac{H}{2},\frac{H}{2}\right)$, and all $t,s \in (-\infty,\infty)$,
			\begin{equation}\label{eqlipbound}
				\begin{gathered}
					\left|y_t(x) - y_s(x')\right| \le C(|x-x'| + |t-s|), \\
					\left|y_t^{-1}(x) - y_s^{-1}(x')\right| \le C(|x-x'| + |t-s|),
				\end{gathered}
			\end{equation}
			for some constant $C>0$.
			
			Consider now any open rectangle $Q\subset\mathbb{R}^2$ with width $W$ and height $H$. By translating \eqref{eqrectflow} to $Q$, we have the same vector field on $Q$, which, to simplify the following presentation, we symbolically notate by 
			\begin{equation*}
				\vcenter{\hbox{\begin{tikzpicture}[scale=1]
							\pgfdeclarelayer{background}
							\pgfsetlayers{background,main}
							\begin{scope}[scale=1/4]
								\draw[red, thick] (-8,0) rectangle (0,1);
								\draw[thick,postaction={decorate},decoration={
									markings,
									mark=at position 1/4 with {\arrow[scale=1]{latex}},
									mark=at position 3/4 with {\arrow[scale=1]{latex}},
								}] (-4,1/2) ellipse (2 and 1/4);
							\end{scope}
				\end{tikzpicture}}} : Q \to \mathbb{R}^2.
			\end{equation*}
			
			We iterate that this is a 1-bounded, time-independent, divergence-free vector field, admitting a Lipschitz Lagrangian flow $\big\{y_t\big\}_{t\in(-\infty,\infty)}$, as in \eqref{eqlipbound}. Moreover, $y_0=\mathrm{Id}$, while $y_{2\max\{W,H\}}$ is exactly a half rotation of $Q$ around its centre.
			
			We also denote the same vector field multiplied by $-1$ by the same diagram with the arrows reversed,
			\begin{equation*}
				\vcenter{\hbox{\begin{tikzpicture}[scale=1]
							\pgfdeclarelayer{background}
							\pgfsetlayers{background,main}
							\begin{scope}[scale=1/4]
								\draw[red, thick] (-8,0) rectangle (0,1);
								\draw[thick,postaction={decorate},decoration={
									markings,
									mark=at position 1/4 with {\arrowreversed[scale=1]{latex}},
									mark=at position 3/4 with {\arrowreversed[scale=1]{latex}},
								}] (-4,1/2) ellipse (2 and 1/4);
							\end{scope}
				\end{tikzpicture}}} = -
				\vcenter{\hbox{\begin{tikzpicture}[scale=1]
							\pgfdeclarelayer{background}
							\pgfsetlayers{background,main}
							\begin{scope}[scale=1/4]
								\draw[red, thick] (-8,0) rectangle (0,1);
								\draw[thick,postaction={decorate},decoration={
									markings,
									mark=at position 1/4 with {\arrow[scale=1]{latex}},
									mark=at position 3/4 with {\arrow[scale=1]{latex}},
								}] (-4,1/2) ellipse (2 and 1/4);
							\end{scope}
				\end{tikzpicture}}}.
			\end{equation*}
			
			This admits the Lagrangian flow $\big\{y_t^{-1}\big\}_{t\in(-\infty,\infty)}$, and so enjoys the same properties as above.
			
			Next, we piece these half-rotations together to perform orientation-preserving swaps. This is exactly the building block required to perform the binary swaps \eqref{eqbinaryswap}. Let $Q\subset\mathbb{R}^2$ be again an open rectangle with width $W>0$ and height $H>0$. Let $L \in \mathbb{N}$ and suppose in addition that $H$ is an integer multiple of $2^{1-L}$, and $W \ge 2^{1-L}$. 
			
			Subsequently, we define a time-dependent vector field, symbolically notated by $\vcenter{\hbox{\begin{tikzpicture}[scale=1/2]
						\pgfdeclarelayer{background}
						\pgfsetlayers{background,main}
						\begin{scope}
							\draw[dashed] (1,0) -- (1,1/2);
							\draw[thick,-latex,yshift=3pt] (0.5,1/4) -- (1.5,1/4);
							\draw[thick,latex-,yshift=-3pt] (0.5,1/4) -- (1.5,1/4);
							\draw[thick,red] (0,0) rectangle (2,1/2);
						\end{scope}
			\end{tikzpicture}}}: Q \times \left[0,3W\right] \to \mathbb{R}^2$, by
			\begin{equation}\label{eqflipdiag}
				\vcenter{\hbox{
						\begin{tikzpicture}[scale=1]
							\pgfdeclarelayer{background}
							\pgfsetlayers{background,main}
							\begin{scope}
								\draw[dashed] (1,0) -- (1,1);
								\draw[thick,-latex,yshift=3pt] (0.5,1/2) -- (1.5,1/2);
								\draw[thick,latex-,yshift=-3pt] (0.5,1/2) -- (1.5,1/2);
								\draw[thick,red] (0,0) rectangle (2,1);
							\end{scope}
				\end{tikzpicture}}}\;(\cdot,t) = \begin{cases}
					\begin{tikzpicture}[scale=2/3]
						\pgfdeclarelayer{background}
						\pgfsetlayers{background,main}
						\begin{scope}[scale=2,shift={(0.75,0.5)}]
							\foreach \y in {0,...,3}{
								\pgfmathparse{int(mod(\y,2))}
								\ifnum\pgfmathresult=0
								\draw[thick,postaction={decorate},decoration={
									markings,
									mark=at position 1/4 with {\arrowreversed[scale=1]{latex}},
									mark=at position 3/4 with {\arrowreversed[scale=1]{latex}},
								}] (1,\y/4+1/8) ellipse (1/2 and 1/16);
								\else
								\draw[thick,postaction={decorate},decoration={
									markings,
									mark=at position 1/4 with {\arrow[scale=1]{latex}},
									mark=at position 3/4 with {\arrow[scale=1]{latex}},
								}] (1,\y/4+1/8) ellipse (1/2 and 1/16);
								\fi
							}
							\foreach \y in {1,...,3}{
								\draw[dashed] (0,\y/4) -- (2,\y/4);
							}
							\foreach \y in {1,...,1}{
								\draw[thick,dashed] (0,\y/2) -- (2,\y/2);
							}
							\draw[thick,red] (0,0) rectangle (2,1);
							\coordinate (n3) at (0,0);
							\coordinate (n4) at (0,1);
							\draw[thick,decorate,decoration = {
								calligraphic brace,
								raise=2pt,
								amplitude=5pt,
								aspect=0.5
							}] (0,1) --  (2,1) node[pos=0.5,above=5pt,black]{$W$};
							\draw[thick,decorate,decoration = {
								calligraphic brace,
								raise=2pt,
								amplitude=5pt,
								aspect=0.5
							}] (0,0) --  (0,1) node[pos=0.5,left=5pt,black]{$H$};
							\draw[thick,decorate,decoration = {
								calligraphic brace,
								raise=2pt,
								amplitude=2pt,
								aspect=0.5
							}] (2,0.25) --  (2,0) node[pos=0.5,right=2pt,black]{$2^{-L}$};
							\draw[thick,decorate,decoration = {
								calligraphic brace,
								raise=2pt,
								amplitude=2pt,
								aspect=0.5
							}] (2,1) --  (2,0.5) node[pos=0.5,right=2pt,black]{$2^{1-L}$};
							\coordinate[right=10pt] (time0) at (2,1) {};
						\end{scope}
					\end{tikzpicture} & \text{if $t \in [0,2W)$,} \\
					\begin{tikzpicture}[scale=2/3]
						\pgfdeclarelayer{background}
						\pgfsetlayers{background,main}
						\begin{scope}[scale=2,shift={(3.25,0.5)}]
							\foreach \y in {0,...,3}{
								\pgfmathparse{int(mod(\y,2))}
								\foreach \x in {0,1}{
									\ifnum\pgfmathresult=0
									\draw[thick,postaction={decorate},decoration={
										markings,
										mark=at position 1/4 with {\arrowreversed[scale=1]{latex}},
										mark=at position 3/4 with {\arrowreversed[scale=1]{latex}},
									}] (\x+1/2,\y/4+1/8) ellipse (1/4 and 1/16);
									\else
									\draw[thick,postaction={decorate},decoration={
										markings,
										mark=at position 1/4 with {\arrow[scale=1]{latex}},
										mark=at position 3/4 with {\arrow[scale=1]{latex}},
									}] (\x+1/2,\y/4+1/8) ellipse (1/4 and 1/16);
									\fi
								}
							}
							\foreach \y in {1,...,3}{
								\draw[dashed] (0,\y/4) -- (2,\y/4);
							}
							\foreach \y in {1,...,1}{
								\draw[thick,dashed] (0,\y/2) -- (2,\y/2);
							}
							\draw[thick,dashed] (1,0) -- (1,1);
							\draw[thick,red] (0,0) rectangle (2,1);
							\draw[thick,decorate,decoration = {
								calligraphic brace,
								raise=2pt,
								amplitude=5pt,
								aspect=0.5
							}] (0,1) -- (1,1) node[pos=0.5,above=5pt,black]{$\frac{1}{2}W$};
							\draw[thick,decorate,decoration = {
								calligraphic brace,
								raise=2pt,
								amplitude=5pt,
								aspect=0.5
							}] (0,0) --  (0,1) node[pos=0.5,left=5pt,black]{$H$};
							\draw[thick,decorate,decoration = {
								calligraphic brace,
								raise=2pt,
								amplitude=2pt,
								aspect=0.5
							}] (2,1) --  (2,3/4) node[pos=0.5,right=2pt,black]{$2^{-L}$};
							\coordinate[right=10pt] (time1) at (2,1) {};
						\end{scope}
					\end{tikzpicture} & \text{if $t \in [2W,3W]$.}
				\end{cases}
			\end{equation}
			
			This has the property that at time $t=3W$ the left and right halves are swapped in an orientation preserving manner. This is illustrated below:
			
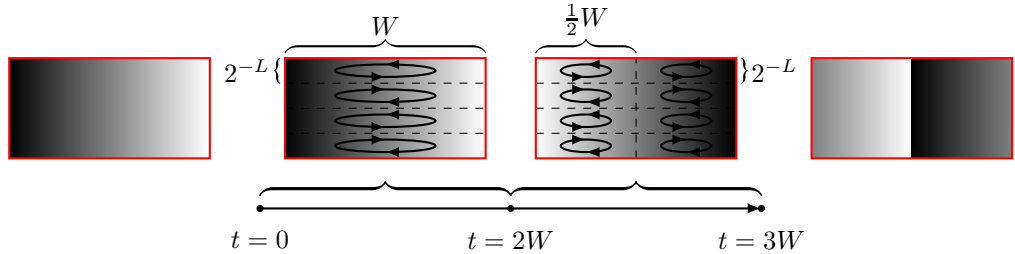
\begin{equation*}
				\begin{tikzpicture}[scale=2/3]
					\pgfdeclarelayer{background}
					\pgfsetlayers{background,main}
					\begin{scope}[scale=2,shift={(-2,0.5)}]
						\shade[left color=black!80,right color=white,shading angle=0] (0,0) rectangle (2,1);
						\shade[opacity=0.6,left color=black,right color=white,shading angle=90] (0,0) rectangle (2,1);
						\draw[thick,red] (0,0) rectangle (2,1);
					\end{scope}
					\begin{scope}[scale=2,shift={(0.75,0.5)}]
						\foreach \y in {0,...,3}{
							\pgfmathparse{int(mod(\y,2))}
							\ifnum\pgfmathresult=0
							\draw[thick,postaction={decorate},decoration={
								markings,
								mark=at position 1/4 with {\arrowreversed[scale=1]{latex}},
								mark=at position 3/4 with {\arrowreversed[scale=1]{latex}},
							}] (1,\y/4+1/8) ellipse (1/2 and 1/16);
							\else
							\draw[thick,postaction={decorate},decoration={
								markings,
								mark=at position 1/4 with {\arrow[scale=1]{latex}},
								mark=at position 3/4 with {\arrow[scale=1]{latex}},
							}] (1,\y/4+1/8) ellipse (1/2 and 1/16);
							\fi
						}
						\foreach \y in {1,...,3}{
							\draw[dashed] (0,\y/4) -- (2,\y/4);
						}
						\draw[thick,red] (0,0) rectangle (2,1);
						\coordinate (n3) at (0,0);
						\coordinate (n4) at (0,1);
						\draw[thick,decorate,decoration = {
							calligraphic brace,
							raise=2pt,
							amplitude=5pt,
							aspect=0.5
						}] (0,1) --  (2,1) node[pos=0.5,above=5pt,black]{$W$};
						\draw[thick,decorate,decoration = {
							calligraphic brace,
							raise=2pt,
							amplitude=2pt,
							aspect=0.5
						}] (0,3/4) --  (0,1) node[pos=0.5,left=2pt,black]{$2^{-L}$};
						\coordinate[right=10pt] (time0) at (2,1) {};
					\end{scope}
					\begin{scope}[scale=2,shift={(3.25,0.5)}]
						\foreach \y in {0,...,3}{
							\pgfmathparse{int(mod(\y,2))}
							\foreach \x in {0,1}{
								\ifnum\pgfmathresult=0
								\draw[thick,postaction={decorate},decoration={
									markings,
									mark=at position 1/4 with {\arrowreversed[scale=1]{latex}},
									mark=at position 3/4 with {\arrowreversed[scale=1]{latex}},
								}] (\x+1/2,\y/4+1/8) ellipse (1/4 and 1/16);
								\else
								\draw[thick,postaction={decorate},decoration={
									markings,
									mark=at position 1/4 with {\arrow[scale=1]{latex}},
									mark=at position 3/4 with {\arrow[scale=1]{latex}},
								}] (\x+1/2,\y/4+1/8) ellipse (1/4 and 1/16);
								\fi
							}
						}
						\foreach \y in {1,...,3}{
							\draw[dashed] (0,\y/4) -- (2,\y/4);
						}
						\draw[dashed] (1,0) -- (1,1);
						\draw[thick,red] (0,0) rectangle (2,1);
						\draw[thick,decorate,decoration = {
							calligraphic brace,
							raise=2pt,
							amplitude=5pt,
							aspect=0.5
						}] (0,1) -- (1,1) node[pos=0.5,above=5pt,black]{$\frac{1}{2}W$};
						\draw[thick,decorate,decoration = {
							calligraphic brace,
							raise=2pt,
							amplitude=2pt,
							aspect=0.5
						}] (2,1) --  (2,3/4) node[pos=0.5,right=2pt,black]{$2^{-L}$};
						\coordinate[right=10pt] (time1) at (2,1) {};
					\end{scope}
					\begin{scope}[scale=2,shift={(6,0.5)}]
						\shade[left color=black!80,right color=white,shading angle=0] (0,0) rectangle (2,1);
						\shade[opacity=0.6,left color=black!50,right color=white,shading angle=90] (0,0) rectangle (1,1);
						\shade[opacity=0.6,left color=black,right color=black!50,shading angle=90] (1,0) rectangle (2,1);
						\draw[thick, red] (0,0) rectangle (2,1);
						\coordinate[right=10pt] (time2) at (2,1) {};
					\end{scope}
					\draw[thick, -latex] (1,0) -- (11,0);
					\fill (1,0) circle (2pt);
					\fill (6,0) circle (2pt);
					\fill (11,0) circle (2pt);
					\node[below=5pt] (t0) at (1,0) {$t=0$};
					\node[below=5pt] (t1) at (6,0) {$t=2W$};
					\node[below=5pt] (t2) at (11,0) {$t=3W$};
					\draw[very thick,decorate,decoration = {
						calligraphic brace,
						raise=5pt,
						amplitude=5pt,
						aspect=0.5
					}] (1,0) --  (6,0) node[pos=0.5,above=5pt,black]{};
					\draw[very thick,decorate,decoration = {
						calligraphic brace,
						raise=5pt,
						amplitude=5pt,
						aspect=0.5
					}] (6,0) --  (11,0) node[pos=0.5,above=5pt,black]{};
				\end{tikzpicture}
			\end{equation*}
			
			Finally, suppose $k \in\mathbb{N}$, $n \in \{1,...,2^{\left\lfloor k/2\right\rfloor}\}$, and $L\in\mathbb{N}$, with $L\ge k+1$.
			
			Let $J_{k,n}=\left[(n-1)2^{-\left\lfloor k/2\right\rfloor},n2^{-\left\lfloor k/2\right\rfloor}\right]\subset\mathbb{T}$.
			
			For $W=2^{-k}$, $H=1$, and $L\ge k+1$ we may use
			$\vcenter{\hbox{\begin{tikzpicture}[scale=1/2]
						\pgfdeclarelayer{background}
						\pgfsetlayers{background,main}
						\begin{scope}
							\draw[dashed] (1,0) -- (1,1/2);
							\draw[thick,-latex,yshift=3pt] (0.5,1/4) -- (1.5,1/4);
							\draw[thick,latex-,yshift=-3pt] (0.5,1/4) -- (1.5,1/4);
							\draw[thick,red] (0,0) rectangle (2,1/2);
						\end{scope}
			\end{tikzpicture}}}$ to define the time-dependent vector field
			\begin{gather*}
				u^{(1,k,n;L)} : \mathbb{T}^2 \times [0,3\cdot2^{-k}] \to \mathbb{R}^2, \\
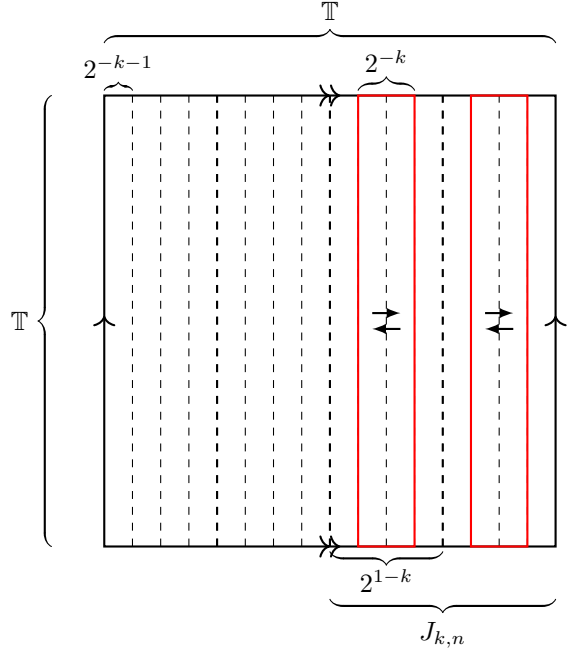
				\begin{tikzpicture}
					\pgfdeclarelayer{background}
					\pgfsetlayers{background,main}
					\begin{scope}[scale=3/2]
						\draw[thick,postaction={decorate},decoration={
							markings,
							mark=at position 1/8 with {\arrow[xshift=2pt]{>[length=4pt]}},
							mark=at position 3/8 with {\arrow[xshift=4pt]{>[length=4pt]>[length=4pt]}},
							mark=at position 5/8 with {\arrow[xshift=2pt]{<[length=4pt]}},
							mark=at position 7/8 with {\arrow[xshift=4pt]{<[length=4pt]<[length=4pt]}},
						}] (0,0) rectangle (4,4);
						\foreach \x in {1,...,15}{
							\draw[dashed] (\x/4,0) -- (\x/4,4);
						}
						\foreach \x in {1,...,3}{
							\draw[thick,dashed] (\x,0) -- (\x,4);
						}
						\foreach \x in {2,...,3}{
							\draw[thick, red] (\x+1/4,0) rectangle (\x+3/4,4);
							\draw[thick,-latex,yshift=2pt] (\x+3/8,2) -- (\x+5/8,2);
							\draw[thick,latex-,yshift=-2pt] (\x+3/8,2) -- (\x+5/8,2);
						}
						\coordinate (n1) at (3,0);
						\coordinate (n2) at (3,4);
						\draw[thick,decorate,decoration = {
							calligraphic brace,
							raise=2pt,
							amplitude=5pt,
							aspect=0.5
						}] (3,0) -- (2,0) node[pos=0.5,below=5pt,black]{$2^{1-k}$};
						\draw[thick,decorate,decoration = {
							calligraphic brace,
							raise=2pt,
							amplitude=3pt,
							aspect=0.5
						}] (0,4) -- (0.25,4) node[pos=0.5,above=3pt,black]{$2^{-k-1}$};
						\draw[thick,decorate,decoration = {
							calligraphic brace,
							raise=2pt,
							amplitude=5pt,
							aspect=0.5
						}] (2.25,4) -- (2.75,4) node[pos=0.5,above=5pt,black]{$2^{-k}$};
						\draw[thick,decorate,decoration = {
							calligraphic brace,
							raise=20pt,
							amplitude=5pt,
							aspect=0.5
						}] (4,0) -- (2,0) node[pos=0.5,below=25pt,black]{$J_{k,n}$};
						\draw[thick,decorate,decoration = {
							calligraphic brace,
							raise=20pt,
							amplitude=5pt,
							aspect=0.5
						}] (0,4) -- (4,4) node[pos=0.5,above=25pt,black]{$\mathbb{T}$};
						\draw[thick,decorate,decoration = {
							calligraphic brace,
							raise=20pt,
							amplitude=5pt,
							aspect=0.5
						}] (0,0) -- (0,4) node[pos=0.5,left=25pt,black]{$\mathbb{T}$};
					\end{scope}
				\end{tikzpicture}.
			\end{gather*}
			
			Note that $u^{(1,k,n;L)}$ is piecewise given by \eqref{eqrectflow}, and so is 1-bounded and divergence-free.
			
			By construction \eqref{eqbinaryswap}, \eqref{eqgradualmixing} are satisfied, while \eqref{eqrectangle} follows from the illustration \eqref{eqflipdiag} after noting the bound $2^{-L} \le 2^{1-k}$.
			
			Finally, to prove \eqref{eqweakconv}, notice in \eqref{eqflipdiag} we alternate between
			$\vcenter{\hbox{\begin{tikzpicture}[scale=1]
						\pgfdeclarelayer{background}
						\pgfsetlayers{background,main}
						\begin{scope}[scale=1/4]
							\draw[red, thick] (-8,0) rectangle (0,1);
							\draw[thick,postaction={decorate},decoration={
								markings,
								mark=at position 1/4 with {\arrow[scale=1]{latex}},
								mark=at position 3/4 with {\arrow[scale=1]{latex}},
							}] (-4,1/2) ellipse (2 and 1/4);
						\end{scope}
			\end{tikzpicture}}}$ and
			$\vcenter{\hbox{\begin{tikzpicture}[scale=1]
						\pgfdeclarelayer{background}
						\pgfsetlayers{background,main}
						\begin{scope}[scale=1/4]
							\draw[red, thick] (-8,0) rectangle (0,1);
							\draw[thick,postaction={decorate},decoration={
								markings,
								mark=at position 1/4 with {\arrowreversed[scale=1]{latex}},
								mark=at position 3/4 with {\arrowreversed[scale=1]{latex}},
							}] (-4,1/2) ellipse (2 and 1/4);
						\end{scope}
			\end{tikzpicture}}}$, and so for all $x \in\mathbb{T}^2$, $t \in [0,3\cdot2^{-k}]$
			\begin{equation*}
				u^{(1,k,n;L)}(x,t) = - u^{(1,k,n;L)}(x+2^{-L}e_2,t).
			\end{equation*}
			
			The proof of \eqref{eqweakconv} now follows the same argument as the proof of \eqref{eqshearweakconv}.
		\end{proof}
	\end{definition}
	\vspace{\baselineskip}
	
	Next, we show the well-posedness of \eqref{eqTE} along these vector fields. As for Proposition \ref{shearuniq}, a weaker version of the following result, when $f\in L^\infty L^\infty$, follows directly from the well-posedness theory of Ambrosio, \cite{ambrosio2004transport}. However, for the sake of completeness, we again prefer to give a direct elementary proof of uniqueness, valid for all $f\in L^1 L^1$.
	
	\begin{proposition}[Binary-swap uniqueness]\label{swapuniq}
		Suppose $i \in \{1,2\}$, $k\in\mathbb{N}$, $n\in\{1,...,2^{\left\lfloor k/2\right\rfloor}\}$, $L \in\mathbb{N}$, with $L \ge k+1$.
		
		Suppose $f\in L^1L^1$ is a weak solution to \eqref{eqTE} along $u^{(i,k,n;L)}$ on the open time interval $(0,3\cdot2^{-k})$.
		
		Then there exists some $f_0 \in L^1(\mathbb{T}^2)$ such that $f$ is (a.e. in $\mathbb{T}^2\times(0,3\cdot2^{-k})$) equal to the Lagrangian solution (Definition \ref{lagrangian})
		\begin{equation}\label{eqswaplag}
			f(\cdot, t)=f_0 \circ \big(y^{(i,k,n;L)}_t\big)^{-1}.
		\end{equation}
		
		In particular, $f$ is a renormalised weak solution to \eqref{eqTE} (Definition \ref{renormalised}) along $u^{(i,k,n;L)}$ with initial data $f_0$. Moreover $f_0 \circ \big(y^{(i,k,n;L)}_t\big)^{-1} \in C^0([0,3\cdot2^{-k}];L^1(\mathbb{T}^2))$.
	\end{proposition}
	\begin{proof}
		We shall give only a brief proof since this essentially follows the proof of Proposition \ref{shearuniq}. Alternatively, one may be able to prove the result via the $L^1([0,T]; BV)$-theory of transport developed by Ambrosio in \cite{ambrosio2004transport}.
		
		The case $i=2$ follows from the same result for $i=1$ and the definitions \eqref{eqi2u}, \eqref{eqi2y}. Hence, we may assume $i=1$.
		
		Notice, from the construction \eqref{eqflipdiag}, that $u^{(1,k,n;L)}$ is stationary/time-independent on the time intervals $I_1=[0,2\cdot2^{-k}]$, and $I_2=[2\cdot2^{-k},3\cdot2^{-k}]$.
		
		We claim it is sufficient to prove \eqref{eqswaplag} for some $f_0\in L^1(\mathbb{T}^2)$ for a.e. $t \in I_1$, and for some $f_0' \in L^1(\mathbb{T}^2)$ for a.e. $t \in I_2$. This is because, after we prove the continuity $f_0 \circ \big(y^{(1,k,n;L)}_t\big)^{-1} \in C^0([0,3\cdot2^{-k}];L^1(\mathbb{T}^2))$, we apply Theorem \ref{weakcont} to deduce $f_0=f_0'$.
		
		Restrict now to one of the time intervals $t\in I_1$ or $t\in I_2$. Then we may divide $\mathbb{T}^2$ into open rectangles $Q\subset \mathbb{T}^2$ (and a set of Lebesgue-measure zero between rectangles) on which $u^{(1,k,n;L)}$ is given either by zero,
		$\vcenter{\hbox{\begin{tikzpicture}[scale=1]
					\pgfdeclarelayer{background}
					\pgfsetlayers{background,main}
					\begin{scope}[scale=1/4]
						\draw[red, thick] (-8,0) rectangle (0,1);
						\draw[thick,postaction={decorate},decoration={
							markings,
							mark=at position 1/4 with {\arrow[scale=1]{latex}},
							mark=at position 3/4 with {\arrow[scale=1]{latex}},
						}] (-4,1/2) ellipse (2 and 1/4);
					\end{scope}
		\end{tikzpicture}}}$, or
		$\vcenter{\hbox{\begin{tikzpicture}[scale=1]
					\pgfdeclarelayer{background}
					\pgfsetlayers{background,main}
					\begin{scope}[scale=1/4]
						\draw[red, thick] (-8,0) rectangle (0,1);
						\draw[thick,postaction={decorate},decoration={
							markings,
							mark=at position 1/4 with {\arrowreversed[scale=1]{latex}},
							mark=at position 3/4 with {\arrowreversed[scale=1]{latex}},
						}] (-4,1/2) ellipse (2 and 1/4);
					\end{scope}
		\end{tikzpicture}}}$.
		
		As in the proof of Proposition \ref{shearuniq} we take a test function $\phi \in C_c^\infty(Q \times I_1)$, supported on $Q \times I_1$ (or $Q \times I_2$ respectively), and let $\psi(x, t) = \phi\big(\big(y^{(1,k,n;L)}_t\big)^{-1}(x),t\big)$. From the Lipschitz bound \eqref{eqlipbound}, $\psi\in L^\infty(Q\times I_1)$ is then Lipschitz (so may be taken as a test function in \eqref{eqTE}), supported on $Q$, and by chain rule the following point-wise equality holds
		\begin{equation*}
			\left(\frac{\partial \psi}{\partial t} + u^{(1,k,n;L)}\cdot\nabla \psi\right)(x,t) = \frac{\partial \phi}{\partial t}\left(\big(y^{(1,k,n;L)}_t\big)^{-1}(x),t\right).
		\end{equation*}
		
		Following the argument as in Proposition \ref{shearuniq} then proves \eqref{eqswaplag} holds a.e. on each $Q \times I_1$, for some $f_0$. By gluing together the $f_0$ for each $Q$, we show that \eqref{eqswaplag} holds a.e. on $\mathbb{T}^2 \times I_1$ (or $\mathbb{T}^2 \times I_2$ respectively).
		
		We are left to show $f_0 \circ \big(y^{(1,k,n;L)}_t\big)^{-1} \in C^0([0,3\cdot2^{-k}];L^1(\mathbb{T}^2))$. As in the proof of Proposition \ref{shearuniq}, this will follow from the global Lipschitz in time bound. That is we must show, for all $x \in \mathbb{T}^2$, and $t,s \in [0,3\cdot2^{-k}]$,
		\begin{equation*}
			\left|y^{(1,k,n;L)}_t\circ\big(y^{(1,k,n;L)}_s\big)^{-1}(x) - x \right| \le C |t-s|.
		\end{equation*}
		
		This follows by taking $x=x'=\big(y^{(1,k,n;L)}_s\big)^{-1}(z)$ and applying the local Lipschitz bound \eqref{eqlipbound}.
		
		As discussed, by Theorem \ref{weakcont} we may glue together $\mathbb{T}^2 \times I_1$ and $\mathbb{T}^2 \times I_2$, proving \eqref{eqswaplag}. $f$ is now a Lagrangian solution to \eqref{eqTE} along $u^{(1,k,n;L)}$ with initial data $f_0$, and so a renormalised weak solution by Remark \ref{lagrenorm}.
	\end{proof}
	\vspace{\baselineskip}
	
	Next, we define the times at which to apply a binary swap $y^{(i,k,n;L)}_{3\cdot2^{-k}}$. We use a quadruple index $(k,m,i,n)\in\mathcal{N}\subset\mathbb{N}^4$ where $k,i,n$ shall determine which binary swap $y^{(i,k,n;L)}_{3\cdot2^{-k}}$ to perform, with $L\in\mathbb{N}$, $L\ge k+1$ chosen later, while $m$ denotes the $m$\textsuperscript{th} occurrence of that binary swap. To define the times $t_{(k,m,i,n)}\in[0,1]$ at which the binary swaps will be applied, we first define an ordering $<_\mathrm{time}$ on the indexing set $\mathcal{N}$, and then define $t_{(k,m,i,n)}$ so that they respect this ordering. The resulting time ordering is illustrated in the diagram \eqref{eqtimediag} below, with the ordering designed to iteratively make $4^m$ copies of the initial data on a square lattice with widths $2^{-m}$, thus mixing the initial data to its spatial average as $m\to\infty$.
	
	We additionally define a well-order $<_\mathrm{lex}$ on the indexing set $\mathcal{N}$. This will inform which binary swaps $y^{(i,k,n;L)}_{3\cdot2^{-k}}$ will be included when approximating the vanishing viscosity limit of \eqref{eqADE}, as in Proposition \ref{vcontrol}.
	\begin{definition}[Total-orders $<_{\mathrm{lex}}$, $<_{\mathrm{time}}$]\label{totalorders}
		For any $p \in \mathbb{N}$, we define the lexicographic well-order $<_{\mathrm{lex}}$ on $\mathbb{N}^p$ via
		\begin{equation*}
			a <_{\mathrm{lex}} b \iff (\exists i \in \{1,...,p\}), \; \begin{cases}
				(\forall j < i), \; a_j = b_j, \\
				a_i < b_i.
			\end{cases}
		\end{equation*}
		
		Let
		\begin{equation*}
			\mathcal{N}=\left\{(k,m,i,n) \in \mathbb{N}^4:\;m\le k,i\in\{1,2\},\;n\le 2^{\left\lfloor k/2\right\rfloor}\right\},
		\end{equation*}
		then note that the well-order $(\mathcal{N},<_{\mathrm{lex}})$ is order isomorphic to $\mathbb{N}$ with the usual order.
		
		We also define another total-order $<_{\mathrm{time}}$ on $\mathcal{N}$, via
		\begin{equation}\label{eqtorder}
			(k_1,m_1,i_1,n_1) <_{\mathrm{time}} (k_2,m_2,i_2,n_2) \text{ if } (m_1,k_2,i_2,n_2) <_{\mathrm{lex}} (m_2,k_1,i_1,n_1),
		\end{equation}
		which is not a well-order.
		
		Define also for each $\mathcal{K} \in \mathcal{N}$ the finite subset
		\begin{equation*}
			\mathcal{N_K} = \{(k,m,i,n) \in \mathcal{N}: (k,m,i,n) \le_{\mathrm{lex}} \mathcal{K}\},
		\end{equation*}
		
		which inherits the finite (and therefore well-) orders $(\mathcal{N_K},<_{\mathrm{lex}})$, $(\mathcal{N_K},<_{\mathrm{time}})$.
		
		For each $(k,m,i,n) \in \mathcal{N}$ we define
		\begin{equation}\label{eqTkn}
			T_{(k,m,i,n)} = \sum_{(k',m',i',n')<_{\mathrm{time}}(k,m,i,n)}3\cdot2^{-k'} < 42,
		\end{equation}
		where the bound on $T_{(k,m,i,n)}$ follows from direct calculation of \[\sum_{k\in\mathbb{N}} 2k\cdot2^{\left\lfloor k/2\right\rfloor}\cdot3\cdot2^{-k} = 42,\] and an empty sum is zero.
		
		Moreover, for each $m \in \mathbb{N}$, let
		\begin{equation}\label{eqtm}
			T_m = \sum_{\substack{m' \le m \\ (k',m',i',n') \in \mathcal{N}}}3\cdot2^{-k'}<42.
		\end{equation}
		
		Writing also $T_0 = 0$, we illustrate the time ordering $<_\mathrm{time}$ in the following diagram,
		\begin{equation}\label{eqtimediag}
			\vcenter{\hbox{
					\begin{tikzpicture}[scale=0.86]
						\pgfdeclarelayer{background}
						\pgfsetlayers{background,main}
						\begin{scope}[scale=2/3]
							\draw[] (0,0) -- (18.5,0);
							\node[anchor=south,yshift=7mm] (t0) at (0,0) {$T_0$};
							\draw[] (0,-1) -- (0,1);
							\node[anchor=south,yshift=7mm] (t1) at (8,0) {$T_1$};
							\draw[] (8,-1) -- (8,1);
							\node[anchor=south,yshift=7mm] (t2) at (14,0) {$T_2$};
							\draw[] (14,-1) -- (14,1);
							\node[anchor=south,yshift=7mm] (t3) at (18,0) {$T_3$};
							\draw[] (18,-1) -- (18,1);
							\draw[thick,decorate,decoration = {
								calligraphic brace,
								raise=5pt,
								amplitude=5pt,
								aspect=0.5
							}] (6,0) -- (8,0) node[pos=0.5,above=7pt,black,font=\tiny]{$T_{(1,1,i,n)}$};
							\draw[] (6,-0.2) -- (6,0.2);
							\draw[thick,decorate,decoration = {
								calligraphic brace,
								raise=5pt,
								amplitude=5pt,
								aspect=0.5
							}] (6,0) -- (4,0) node[pos=0.5,below=7pt,black,font=\tiny]{$T_{(2,1,i,n)}$};
							\draw[] (4,-0.2) -- (4,0.2);
							\draw[thick,decorate,decoration = {
								calligraphic brace,
								raise=5pt,
								amplitude=5pt,
								aspect=0.5
							}] (3,0) -- (4,0) node[pos=0.5,above=7pt,black,font=\tiny]{$T_{(3,1,i,n)}$};
							\draw[] (3,-0.2) -- (3,0.2);
							\draw[thick,decorate,decoration = {
								calligraphic brace,
								raise=5pt,
								amplitude=5pt,
								aspect=0.5
							}] (3,0) -- (2,0) node[pos=0.5,below=7pt,black,font=\tiny]{$T_{(4,1,i,n)}$};
							\draw[] (2,-0.2) -- (2,0.2);
							\draw[] (1.5,-0.2) -- (1.5,0.2);
							\draw[] (1,-0.2) -- (1,0.2);
							\node[anchor=south,font=\small] (l1) at (0.5,0) {$...$};
							\draw[thick,decorate,decoration = {
								calligraphic brace,
								raise=5pt,
								amplitude=5pt,
								aspect=0.5
							}] (14,0) -- (12,0) node[pos=0.5,below=7pt,black,font=\tiny]{$T_{(2,2,i,n)}$};
							\draw[] (12,-0.2) -- (12,0.2);
							\draw[thick,decorate,decoration = {
								calligraphic brace,
								raise=5pt,
								amplitude=5pt,
								aspect=0.5
							}] (11,0) -- (12,0) node[pos=0.5,above=7pt,black,font=\tiny]{$T_{(3,2,i,n)}$};
							\draw[] (11,-0.2) -- (11,0.2);
							\draw[thick,decorate,decoration = {
								calligraphic brace,
								raise=5pt,
								amplitude=5pt,
								aspect=0.5
							}] (11,0) -- (10,0) node[pos=0.5,below=7pt,black,font=\tiny]{$T_{(4,2,i,n)}$};
							\draw[] (10,-0.2) -- (10,0.2);
							\draw[] (9.5,-0.2) -- (9.5,0.2);
							\draw[] (9,-0.2) -- (9,0.2);
							\node[anchor=south,font=\small] (l2) at (8.5,0) {$...$};
							\draw[thick,decorate,decoration = {
								calligraphic brace,
								raise=5pt,
								amplitude=5pt,
								aspect=0.5
							}] (17,0) -- (18,0) node[pos=0.5,above=7pt,xshift=-2.5mm,black,font=\tiny]{$T_{(3,3,i,n)}$};
							\draw[] (17,-0.2) -- (17,0.2);
							\draw[thick,decorate,decoration = {
								calligraphic brace,
								raise=5pt,
								amplitude=5pt,
								aspect=0.5
							}] (17,0) -- (16,0) node[pos=0.5,below=7pt,black,font=\tiny]{$T_{(4,3,i,n)}$};
							\draw[] (16,-0.2) -- (16,0.2);
							\draw[] (15.5,-0.2) -- (15.5,0.2);
							\draw[] (15,-0.2) -- (15,0.2);
							\node[anchor=south,font=\small] (l3) at (14.5,0) {$...$};
							\node[anchor=west] (l4) at (18.5,0) {$...$};
						\end{scope}
					\end{tikzpicture}
			}}
		\end{equation}
		where for fixed $k,m \in \mathbb{N}$ with $m \le k$, the bracket \begin{tikzpicture}
			\draw[thick,decorate,decoration = {
				calligraphic brace,
				raise=0pt,
				amplitude=5pt,
				aspect=0.5
			}] (0,0) -- (1,0) node[pos=0.5,above=2pt,black,font=\tiny]{$T_{(k,m,i,n)}$};
		\end{tikzpicture} contains for all $i\in\{1,2\}$, $n\in\mathbb{N}$ with $n \le 2^{\left\lfloor k/2\right\rfloor}$, precisely the time intervals
		\begin{equation*}
			\left[T_{(k,m,i,n)},T_{(k,m,i,n)}+3\cdot2^{-k}\right).
		\end{equation*}
	\end{definition}
	\vspace{\baselineskip}
	
	Next, we apply the divergence-free vector fields $u^{\left(i,k,n;L\right)}$ (for some $L\in\mathbb{N}$ to be chosen) in each time interval $\left[T_{(k,m,i,n)},T_{(k,m,i,n)}+3\cdot2^{-k}\right)$, with $m\in\{1,...,k\}$ an index denoting the $m$\textsuperscript{th} occurrence of this swap. The order is chosen so that the solution to \eqref{eqTE} at time $t=T_m$ creates $4^m$ identical copies of the initial data on a square lattice with widths $2^{-m}$.
	
	To see how this can be achieved, $i\in\{1,2\}$ denotes which coordinate the binary swap \eqref{eqbinaryswap} acts on, while $n\in\{1,...,2^{\left\lfloor k/2\right\rfloor}\}$ is an integer denoting which region of $\mathbb{T}^2$ the binary swap is completed on.
	
	For fixed $k\in\mathbb{N}$, $m\in\{1,...,k\}$, these binary swap commute, and together the bracket \begin{tikzpicture}
		\draw[thick,decorate,decoration = {
			calligraphic brace,
			raise=0pt,
			amplitude=5pt,
			aspect=0.5
		}] (0,0) -- (1,0) node[pos=0.5,above=2pt,black,font=\tiny]{$T_{(k,m,i,n)}$};
	\end{tikzpicture} swaps the $k$\textsuperscript{th} and $(k+1)$\textsuperscript{th} binary digit of both coordinates of every $x=(x_1,x_2)\in\mathbb{T}^2$.
	
	The order $\le_\mathrm{lex}$ is chosen such that between $T_{m-1}$ and $T_m$, an `undefined' binary digit passes all the way from infinitely far up the binary expansion, down to the $m$\textsuperscript{th} position. Since its values are undefined, this creates four copies (2 in each coordinate) of the $t=T_{m-1}$ data at $t=T_m$. See Proposition \ref{mixinglemma} below for the details.
	
	But first, we define this vector field and its finite approximations, which by Theorem \ref{vcontrol} will control the vanishing viscosity limit of \eqref{eqADE}.
	\begin{definition}[Gradual perfect mixing]\label{perfectmixing}
		We follow the notation introduced in Definition \ref{totalorders}.
		
		Notice that a particular $\mathcal{P}=(k,m,i,n)\in\mathcal{N}$ fixes $k\in\mathbb{N}$, $i\in\{1,2\}$, $n\in\{1,...,2^{\left\lfloor k/2\right\rfloor}\}$.
		
		For $\mathcal{K} \in \mathcal{N}$, and a finite sequence $\{L_{\mathcal{P}}\}_{\mathcal{P} \in \mathcal{N_K}}\subset\mathbb{N}$ with $L_{(k,m,i,n)} \ge k+1$ for all $(k,m,i,n)\in\mathcal{N_K}$, we define the following time-dependent vector field.
		\begin{gather*}
			u_{\mathcal{K}}^{\{L_\mathcal{P}\}_{\mathcal{P}\in\mathcal{N_K}}}:\mathbb{T}^2 \times [0,50] \to \mathbb{R}^2, \\
			u_{\mathcal{K}}^{\{L_\mathcal{P}\}_{\mathcal{P}\in\mathcal{N_K}}}(x,t) = \begin{cases}
				u^{\left(i,k,n;L_\mathcal{P}\right)}(x,t-T_\mathcal{P}) & \text{for } \begin{cases}
					\mathcal{P} = (k,m,i,n) \in \mathcal{N_K}, \text{ and} \\ t \in \left[T_\mathcal{P},T_\mathcal{P}+3\cdot2^{-k}\right),
				\end{cases} \\
				0 & \text{otherwise, in particular for $t \ge 42$}.
			\end{cases}
		\end{gather*}
		
		Meanwhile, for an infinite sequence $\{L_\mathcal{P}\}_{\mathcal{P} \in \mathcal{N}}\subset\mathbb{N}$ with $L_{(k,m,i,n)} \ge k+1$ for all $(k,m,i,n)\in\mathcal{N}$, we define the following time-dependent vector field.
		\begin{gather*}
			u_{\infty}^{\{L_\mathcal{P}\}_{\mathcal{P}\in\mathcal{N}}}:\mathbb{T}^2 \times [0,50] \to \mathbb{R}^2, \\
			u_{\infty}^{\{L_\mathcal{P}\}_{\mathcal{P}\in\mathcal{N}}}(x,t) = \begin{cases}
				u^{\left(i,k,n;L_\mathcal{P}\right)}(x,t-T_\mathcal{P}) & \text{for } \begin{cases}
					\mathcal{P}=(k,m,i,n) \in \mathcal{N}, \text{ and} \\ t \in \left[T_\mathcal{P},T_\mathcal{P}+3\cdot2^{-k}\right),
				\end{cases} \\
				0 & \text{otherwise, in particular for $t \ge 42$}.
			\end{cases}
		\end{gather*}
		
		These vector fields are well-defined since by \eqref{eqTkn}, the time intervals
		\begin{equation*}
			\left[T_{(k,m,i,n)},T_{(k,m,i,n)}+3\cdot2^{-k}\right),
		\end{equation*}
		are disjoint subsets of $[0,42]$.
	\end{definition}
	
	\begin{proposition}[Mixing]\label{mixinglemma}
		We follow the notation introduced in Definitions \ref{totalorders}, \ref{perfectmixing}.
		
		Let $f_0 \in L^\infty(\mathbb{T}^2)$, and suppose we have an infinite sequence $\{L_\mathcal{P}\}_{\mathcal{P}\in\mathcal{N}}\subset\mathbb{N}$ with $L_{(k,m,i,n)} \ge k+1$ for all $(k,m,i,n)\in\mathcal{N}$. Then for each $\mathcal{K} \in \mathcal{N}$ there exists a unique weak solution $f^\mathcal{K}$ to \eqref{eqTE} along $u_{\mathcal{K}}^{\{L_\mathcal{P}\}_{\mathcal{P}\in\mathcal{N_K}}}$ with initial data $f_0$. Moreover, $f^\mathcal{K}\in (C^0L^1) \cap (L^\infty L^\infty)$, is a Lagrangian solution, and hence also a renormalised weak solution (Definitions \ref{renormalised}, \ref{lagrangian}).
		
		By the well-order $(\mathcal{N},<_{\mathrm{lex}})$ (which is order isomorphic to $(\mathbb{N},<)$), as $\mathcal{K} \xrightarrow{\mathrm{lex}} \infty$,
		\begin{equation*}
			f^\mathcal{K}\xrightarrow{\mathcal{K}\to\infty}f^\infty,
		\end{equation*}
		with the above convergence in weak-$*$ $L^\infty L^\infty$, strong in $L^p([0,42]; L^p)$ and $\left.C^0([0,42-\epsilon];L^p)\right.$ for all $p \in [1,\infty)$, and all $\epsilon > 0$. The limit function $f^\infty \in C_{\mathrm{weak-}*}^0L^\infty$ is a weak solution to \eqref{eqTE} along $u_{\infty}^{\{L_\mathcal{P}\}_{\mathcal{P}\in\mathcal{N}}}$ with initial data $f_0$.
		
		Moreover, for all $t \in [42,50]$
		\begin{equation}\label{eqmixed}
			f^\infty(\cdot,t) \equiv \int_{\mathbb{T}^2} f_0(y) \; dy,
		\end{equation}
		i.e., $f^\infty(\cdot, t)$ is perfectly mixed after $t=42$.
	\end{proposition}
	\begin{proof}
		Recall the language and notation introduced in Definitions \ref{lagswap}, \ref{totalorders}, \ref{perfectmixing}, as well as Definitions \ref{renormalised}, \ref{lagrangian} of renormalised and Lagrangian solutions to \eqref{eqTE}.
		
		By Theorems \ref{weakcont}, \ref{TEexistence}, there exists some weak solution $f^\mathcal{K} \in C_{\mathrm{weak-}*}^0L^\infty$ to \eqref{eqTE} along $u_{\mathcal{K}}^{\{L_\mathcal{P}\}_{\mathcal{P}\in\mathcal{N_K}}}$ with initial data $f_0$.

		We aim to show $f \in C^0 L^1$ with
		\begin{equation}\label{eqlagexpr}
			f^\mathcal{K}(\cdot, t) = f_0 \circ \big(\tilde{y}_t^{\mathcal{K}}\big)^{-1},
		\end{equation}
		for $\big\{\tilde{y}_t^{\mathcal{K}}\big\}_{t \in [0,50]}$ a Lagrangian flow along $u_{\mathcal{K}}^{\{L_\mathcal{P}\}_{\mathcal{P}\in\mathcal{N_K}}}$ which does not depend on $f_0$, thus proving uniqueness.
		
		From the definition of $u_{\mathcal{K}}^{\{L_\mathcal{P}\}_{\mathcal{P}\in\mathcal{N_K}}}$, finiteness of the set $\mathcal{N_K}$, and disjointness of the intervals $\left[T_{(k,m,i,n)},T_{(k,m,i,n)}+3\cdot2^{-k}\right)$ for all $(k,m,i,n) \in \mathcal{N_K}$, we may piecewise apply Proposition \ref{swapuniq} to $f^\mathcal{K}$. We use that $f^\mathcal{K} \in C_{\mathrm{weak-}*}^0L^\infty$ to glue together the pieces, and that any weak solution to \eqref{eqTE} along $u\equiv 0$ on an open time interval $I$ is a constant function of $t\in I$, see for example Theorem 3.1.4' in \cite{hormander2003}. This constructs for each $\mathcal{K} \in \mathcal{N}$ a Lagrangian flow $\tilde{y}_t^\mathcal{K}$ satisfying \eqref{eqlagexpr}, and moreover gives an expression for $\tilde{y}_t^\mathcal{K}$ in terms of the binary swaps $y_t^{(i,k,n;L)}$, see \eqref{eqpiecelag} below.
		
		For each $\mathcal{P}=(k,m,i,n) \in \mathcal{N_K}$, define
		\begin{equation*}
			T_{\mathrm{suc}(\mathcal{P})} = \begin{cases}
				50 & \text{if $\mathcal{P}$ is maximal in $(\mathcal{N_K},<_{\mathrm{time}})$},\;\text{else} \\
				T_{\mathcal{P}'} & \text{for $\mathcal{P}'$ the successor of $\mathcal{P}$ in $(\mathcal{N_K},<_{\mathrm{time}})$}.
			\end{cases}
		\end{equation*}
		
		With each $\mathcal{P}=(k,m,i,n)\in\mathcal{N_K}$ fixed, and therefore fixed $k\in\mathbb{N}$, by \eqref{eqTkn} we have that
		\begin{equation}\label{eqsucbound}
			T_\mathcal{P} + 3\cdot2^{-k} \le T_{\mathrm{suc}(\mathcal{P})}.
		\end{equation}
		
		With each $\mathcal{P}=(k,m,i,n)\in\mathcal{N_K}$ fixed, and therefore fixed $k\in\mathbb{N}$, $i\in\{1,2\}$, $n\in\{1,...,2^{\left\lfloor k/2\right\rfloor}\}$, we have that for all $t \in [T_\mathcal{P}, T_{\mathrm{suc}(\mathcal{P})}]$
		\begin{equation}\label{eqpiecelag}
			\tilde{y}_t^{\mathcal{K}} \circ \big(\tilde{y}_{T_\mathcal{P}}^\mathcal{K}\big)^{-1} = \begin{cases}
				y_{t-T_\mathcal{P}}^{\left(i,k,n;L_\mathcal{P}\right)} & \text{if } t \in [T_\mathcal{P}, T_\mathcal{P} + 3\cdot2^{-k}], \\
				y_{3\cdot2^{-k}}^{\left(i,k,n;L_\mathcal{P}\right)} & \text{if } t \in [T_\mathcal{P} + 3\cdot2^{-k}, T_{\mathrm{suc}(\mathcal{P})}].
			\end{cases} 
		\end{equation}
		
		Meanwhile, if $\mathcal{P}_\mathrm{min}$ is minimal in $(\mathcal{N_K},<_{\mathrm{time}})$, we have that for all $t \in [0,T_{\mathcal{P}_\mathrm{min}}]$, $\tilde{y}_t^{\mathcal{K}} = \mathrm{Id}$. This completes the proof of uniqueness.
		
		Next, we wish to show, for any $m \in \mathbb{N}$, that the sequence $f^\mathcal{K} \in C^0L^1$ is Cauchy in $C^0([T_{m-1},T_m];L^1)$ as $\mathcal{K} \xrightarrow{\mathrm{lex}} \infty$, as follows.
		
		We proceed by induction on $m \in \mathbb{N}$. By the inductive hypothesis we have that $f^\mathcal{K}(\cdot, T_{m-1})\in L^1(\mathbb{T}^2)$ is Cauchy in $L^1(\mathbb{T}^2)$ as $\mathcal{K} \xrightarrow{\mathrm{lex}}\infty$ (for the base case $m=1$ this follows by $f^\mathcal{K}(\cdot, T_0)\equiv f_0$). From this we wish to deduce that $f^\mathcal{K} \in C^0L^1$ is Cauchy in $C^0([T_{m-1},T_m];L^1)$ as $\mathcal{K} \xrightarrow{\mathrm{lex}} \infty$. Before we proceed with the induction we will need some properties of the Lagrangian flow $\tilde{y}_t^{\mathcal{K}}$ on the interval $t\in[T_{m-1},T_m]$.
		
		Fix any $(k,m,i,n) \in \mathcal{N}$, which fixes $k\in\mathbb{N}$, $m\in\{1,...,k\}$, $i\in\{1,2\}$, $n\in\{1,...,2^{\left\lfloor k/2\right\rfloor}\}$.
		
		Recall the definition \eqref{eqtm} of $T_m$ and $T_{m-1}$, where we set $T_0=0$ if necessary. By \eqref{eqtorder}, we have (it may be helpful to consider the illustration \eqref{eqtimediag})
		\begin{gather*}
			T_{m-1} < T_{(k,m,i,n)}, \\
			T_{(k,m,i,n)} + 3\cdot2^{-k} \le T_m,
		\end{gather*}
		and also for each $\mathcal{P} = (k',m',i',n') >_{\mathrm{lex}} (k,m,i,n)$, which fixes $k'\in\mathbb{N}$, $m' \in \{1,...,k'\}$,
		\begin{gather*}
			T_\mathcal{P} + 3\cdot2^{-k'} \le T_{(k,m,i,n)}\quad \text{if } m' \le m, \\
			T_m < T_\mathcal{P}\quad \text{if } m<m'.
		\end{gather*}
		
		Therefore, by the piecewise structure \eqref{eqpiecelag}, for each $t \in [T_{(k,m,i,n)},T_m]$, and for each $\mathcal{K}' >_{\mathrm{lex}} \mathcal{K} \ge_{\mathrm{lex}} (k,m,i,n)$, we have that
		\begin{equation*}
			\tilde{y}_t^{\mathcal{K}'} \circ \big(\tilde{y}_{T_{(k,m,i,n)}}^{\mathcal{K}'}\big)^{-1} = \tilde{y}_t^{\mathcal{K}} \circ \big(\tilde{y}_{T_{(k,m,i,n)}}^{\mathcal{K}}\big)^{-1},
		\end{equation*}
		is unchanged. We shall need the inverse of these Lagrangian maps,
		\begin{equation}\label{equnchanged}
			\tilde{y}_{T_{(k,m,i,n)}}^{\mathcal{K}'} \circ \big(\tilde{y}_{t}^{\mathcal{K}'}\big)^{-1} = \tilde{y}_{T_{(k,m,i,n)}}^{\mathcal{K}} \circ \big(\tilde{y}_t^{\mathcal{K}}\big)^{-1}.
		\end{equation}
		
		Next, for each $\mathcal{P} = (k',m',i',n') \in \mathcal{N}$ with $k' < k$, which fixes $k'\in\mathbb{N}$, $m' \in \{1,...,k'\}$, we have that
		\begin{gather*}
			T_\mathcal{P} +3\cdot 2^{-k'} \le T_{m-1}\quad \text{if } m' < m, \\
			T_{(k,m,i,n)} < T_\mathcal{P}\quad \text{if } m \le m'.
		\end{gather*}
		
		Therefore, we may apply \eqref{eqrectangle} and the piecewise structure \eqref{eqpiecelag} to show the following.
		
		For any $r, r' \in \{1,...,2^{k-1}\}$, denote the intervals $J=[(r-1)2^{1-k},r 2^{1-k}]$, $J'=[(r'-1)2^{1-k},r' 2^{1-k}]$.
		
		For all $\mathcal{K} \in \mathcal{N}$, for all $t \in [T_{m-1},T_{(k,m,i,n)}]$, we see that the Lagrangian-flow $\tilde{y}_t^{\mathcal{K}} \circ \big(\tilde{y}_{T_{m-1}}^{\mathcal{K}}\big)^{-1}$ preserves the square lattice
		\begin{equation}\label{eqsquare}
			\tilde{y}_t^{\mathcal{K}} \circ \big(\tilde{y}_{T_{m-1}}^{\mathcal{K}}\big)^{-1} : J \times J' \leftrightarrow J \times J'.
		\end{equation}
		
		We are now ready to proceed with the induction. Recall we have assumed that for some $m\in\mathbb{N}$, $f^\mathcal{K}(\cdot, T_{m-1})\in L^1(\mathbb{T}^2)$ is Cauchy in $L^1(\mathbb{T}^2)$ as $\mathcal{K} \xrightarrow{\mathrm{lex}}\infty$. 
		
		We wish to deduce $f^\mathcal{K} \in C^0L^1$ is Cauchy in $C^0([T_{m-1},T_m];L^1(\mathbb{T}^2))$ as $\mathcal{K} \xrightarrow{\mathrm{lex}} \infty$.
		
		Denote by $f_{m-1} \in L^1(\mathbb{T}^2)$ the limit $f^\mathcal{K}(\cdot, T_{m-1}) \xrightarrow{\mathcal{K}\to\infty}f_{m-1}$ in $L^1(\mathbb{T}^2)$.
		
		Note, by \eqref{eqlagexpr}, for all $\mathcal{K} \in \mathcal{N}$, and for all $t \in [T_{m-1},T_m]$, we may write
		\begin{equation}\label{eqmlagflow}
			f^{\mathcal{K}}(\cdot, t) = f^{\mathcal{K}}(\cdot, T_{m-1}) \circ \tilde{y}_{T_{m-1}}^{\mathcal{K}} \circ \big(\tilde{y}_t^{\mathcal{K}}\big)^{-1}.
		\end{equation}
		
		Fix some $\epsilon > 0$. Let $\phi_\epsilon\in C^\infty(\mathbb{T}^2)$ be a smooth approximation of $f_{m-1}$ in $L^1(\mathbb{T}^2)$ so that $\left\|\phi_\epsilon - f_{m-1}\right\|_{L^1(\mathbb{T}^2)} \le \epsilon$.
		
		Fixing also some $k \in \mathbb{N}$, let $\mathcal{K} \in \mathcal{N}$ be large enough that we have $\mathcal{K} \ge_\mathrm{lex} (k,m,1,1)$, and that for all $\mathcal{K'} \ge_\mathrm{lex} \mathcal{K}$,
		\begin{equation*}
			\left\|f^\mathcal{K'}(\cdot,T_{m-1}) - f_{m-1}\right\|_{L^1(\mathbb{T}^2)} \le \epsilon.
		\end{equation*}
		
		Then, it also follows that
		\begin{equation*}
			\left\|f^\mathcal{K'}(\cdot, T_{m-1}) - \phi_\epsilon\right\|_{L^1(\mathbb{T}^2)} \le 2\epsilon.
		\end{equation*}
		
		Therefore, by \eqref{eqmlagflow}, and since Lagrangian flows are Lebesgue-measure preserving, for any $\mathcal{K}' \ge_\mathrm{lex} \mathcal{K}$, and for all $t \in [T_{m-1},T_m]$, we have that
		\begin{equation*}
			\left\|f^{\mathcal{K}'}(\cdot, t) - f^{\mathcal{K}}(\cdot, t) \right\|_{L^1(\mathbb{T}^2)} \le 4\epsilon + \left\|\phi_\epsilon \circ \tilde{y}_{T_{m-1}}^{\mathcal{K}'} \circ \big(\tilde{y}_t^{\mathcal{K}'}\big)^{-1} - \phi_\epsilon \circ \tilde{y}_{T_{m-1}}^{\mathcal{K}} \circ \big(\tilde{y}_t^{\mathcal{K}}\big)^{-1} \right\|_{L^1(\mathbb{T}^2)}.
		\end{equation*}
		
		We now split into two cases, $t \in [T_{m-1},T_{(k,m,1,1)}]$, and $t \in [T_{(k,m,1,1)},T_m]$. We first consider the former.
		
		For the square lattice of widths $2^{1-k}$, $J\times J'\subset\mathbb{T}^2$ defined in \eqref{eqsquare}, we have for all $x,y\in J \times J'$,
		\begin{align*}
			\left|\phi_\epsilon(x) - \phi_\epsilon(y) \right| & \le \left\|\nabla \phi_\epsilon\right\|_{L^\infty(\mathbb{T}^2)} |x-y| \\
			& \le \sqrt{2} \left\|\nabla \phi_\epsilon\right\|_{L^\infty(\mathbb{T}^2)} 2^{1-k}.
		\end{align*}
		
		So by \eqref{eqsquare}, for all $t \in [T_{m-1},T_{(k,m,1,1)}]$, and for all $\mathcal{K'} \ge_\mathrm{lex} \mathcal{K}$, we have that
		\begin{equation}\label{eqs6e1}
			\left\|\phi_\epsilon \circ \tilde{y}_{T_{m-1}}^{\mathcal{K}'} \circ \big(\tilde{y}_t^{\mathcal{K}'}\big)^{-1} - \phi_\epsilon \circ \tilde{y}_{T_{m-1}}^{\mathcal{K}} \circ \big(\tilde{y}_t^{\mathcal{K}}\big)^{-1} \right\|_{L^1(\mathbb{T}^2)} \le \sqrt{2}\left\|\nabla\phi_\epsilon\right\|_{L^\infty(\mathbb{T}^2)}2^{1-k}.
		\end{equation}
		
		Next, by \eqref{equnchanged}, for all $t \in [T_{(k,m,1,1)},T_m]$, we have that
		\begin{multline*}
			\left\|\phi_\epsilon \circ \tilde{y}_{T_{m-1}}^{\mathcal{K}'} \circ \big(\tilde{y}_t^{\mathcal{K}'}\big)^{-1} - \phi_\epsilon \circ \tilde{y}_{T_{m-1}}^{\mathcal{K}} \circ \big(\tilde{y}_t^{\mathcal{K}}\big)^{-1} \right\|_{L^1(\mathbb{T}^2)} \\
			\begin{aligned}[t]
				& = \left\|\phi_\epsilon \circ \tilde{y}_{T_{m-1}}^{\mathcal{K}'} \circ \big(\tilde{y}_{T_{(k,m,1,1)}}^{\mathcal{K}'}\big)^{-1} \circ \tilde{y}_{T_{(k,m,1,1)}}^{\mathcal{K}} \circ \big(\tilde{y}_t^{\mathcal{K}}\big)^{-1} - \phi_\epsilon \circ \tilde{y}_{T_{m-1}}^{\mathcal{K}} \circ \big(\tilde{y}_t^{\mathcal{K}}\big)^{-1} \right\|_{L^1(\mathbb{T}^2)} \\
				& = \left\|\phi_\epsilon \circ \tilde{y}_{T_{m-1}}^{\mathcal{K}'} \circ \big(\tilde{y}_{T_{(k,m,1,1)}}^{\mathcal{K}'}\big)^{-1} - \phi_\epsilon \circ \tilde{y}_{T_{m-1}}^{\mathcal{K}} \circ \big(\tilde{y}_{T_{(k,m,1,1)}}^{\mathcal{K}}\big)^{-1} \right\|_{L^1(\mathbb{T}^2)},
			\end{aligned}
		\end{multline*}
		where the last line is already bounded in \eqref{eqs6e1}, by $\sqrt{2}\left\|\nabla\phi_\epsilon\right\|_{L^\infty(\mathbb{T}^2)}2^{1-k}$. Putting everything together, we see that
		\begin{equation*}
			\left\|f^{\mathcal{K}'} - f^{\mathcal{K}} \right\|_{L^\infty([T_{m-1},T_m];L^1(\mathbb{T}^2))} \le 4\epsilon + \sqrt{2}\left\|\nabla\phi_\epsilon\right\|_{L^\infty(\mathbb{T}^2)}2^{1-k}.
		\end{equation*}
		
		$\left\|\nabla\phi_\epsilon\right\|_{L^\infty(\mathbb{T}^2)}$ depends only on $\epsilon$ and $f_{m-1}$, and in particular not on $k$. Thus for $k$ sufficiently large, i.e. for $\mathcal{K'},\mathcal{K} \in (\mathcal{N},<_\mathrm{lex})$ sufficiently large, we can make the right-hand side arbitrarily small. Therefore $f^\mathcal{K}$ is Cauchy in $L^\infty([T_{m-1},T_m];L^1(\mathbb{T}^2))$ as $\mathcal{K} \xrightarrow{\mathrm{lex}} \infty$, as required.
		
		This completes the induction. Observe that $T_m \xrightarrow{m\to\infty}42$, and so we have proven that $f^\mathcal{K}$ converges in $C^0([0,42-\epsilon];L^1)$ as $\mathcal{K} \xrightarrow{\mathrm{lex}} \infty$, for any $\epsilon > 0$.
		
		$f^\mathcal{K}$ are Lagrangian (and weak) solutions to \eqref{eqTE} along $u_{\mathcal{K}}^{\{L_\mathcal{P}\}_{\mathcal{P}\in\mathcal{N}}}$ with initial data $f_0$. Therefore, since $f_0 \in L^\infty(\mathbb{T}^2)$, $f^\mathcal{K}$ are uniformly bounded in $L^\infty L^\infty$. 
		
		Moreover, by Definition \ref{perfectmixing}, $u_{\mathcal{K}}^{\{L_\mathcal{P}\}_{\mathcal{P}\in\mathcal{N_K}}}(\cdot, t)$ is eventually constant, and equal to $u_{\infty}^{\{L_\mathcal{P}\}_{\mathcal{P}\in\mathcal{N}}}(\cdot, t)$ as $\mathcal{K}\xrightarrow{\mathrm{lex}} \infty$ for each $t \in [0,50]$, and are uniformly bounded in $L^\infty L^\infty$.
		
		Denote by $\bar{f}$ a weak-$*$ limit point of $\{f^\mathcal{K}\}_{\mathcal{K}\in\mathcal{N}}$ in $L^\infty([0,50];L^\infty)$, that is there exists a subsequence $\mathcal{K}_n \xrightarrow{n\to\infty} \infty$ with $f^{\mathcal{K}_n}\xrightarrow{n\to\infty}\bar{f}$ in weak-$*$ $L^\infty L^\infty$. Then for all $\phi \in C_c^\infty(\mathbb{T}^2 \times [0,50))$, we see that
		\begin{align*}
			& \int_{\mathbb{T}^2\times[0,50]} \bar{f} \left(\frac{\partial \phi}{\partial t} + u_{\infty}^{\{L_\mathcal{P}\}_{\mathcal{P}\in\mathcal{N}}}\cdot\nabla \phi\right) \; dxdt \\
			& = \lim_{n\to\infty} \int_{\mathbb{T}^2\times[0,50]} f^{\mathcal{K}_n} \left(\frac{\partial \phi}{\partial t} + u_{\infty}^{\{L_\mathcal{P}\}_{\mathcal{P}\in\mathcal{N}}}\cdot\nabla \phi\right) \; dxdt \\
			& = \lim_{n\to\infty} \int_{\mathbb{T}^2\times[0,50]} f^{\mathcal{K}_n} \left(\frac{\partial \phi}{\partial t} + u_{\mathcal{K}_n}^{\{L_\mathcal{P}\}_{\mathcal{P}\in\mathcal{N}_{\mathcal{K}_n}}}\cdot\nabla \phi\right) \; dxdt \\
			& = - \int_{\mathbb{T}^2} f_0 \phi_0 \; dx.
		\end{align*}
		
		Therefore, $\bar{f}$ is a weak solution to \eqref{eqTE} along $u_{\infty}^{\{L_\mathcal{P}\}_{\mathcal{P}\in\mathcal{N}}}$ with initial data $f_0$. Moreover, it is bounded in $L^\infty L^\infty$ and so by Theorem \ref{weakcont} the limit is in $C_{\mathrm{weak-}*}^0([0,50];L^\infty)$. However, $u_{\infty}^{\{L_\mathcal{P}\}_{\mathcal{P}\in\mathcal{N}}}(\cdot, t) \equiv 0$ for each $t \in [42,50]$ and so (say by \eqref{eqtrace}), for all $t\in[42,50]$, $\bar{f}(\cdot, t)$ is determined by $\bar{f}(\cdot, s)$ for $s\in[0,42)$.
		
		However, we have already shown $\underset{\mathcal{K}\to\infty}{\lim}f^\mathcal{K}$ converges in $C^0([0,42-\epsilon];L^1)$ for any $\epsilon > 0$. Therefore, the limit $\bar{f}$ is unique. Assume then that $\underset{\mathcal{K}\to\infty}{\lim} f^\mathcal{K}$ does not converge in weak-$*$ $L^\infty([0,50];L^\infty)$. Then, by the uniform bound in $L^\infty([0,50];L^\infty)$, there exist at least two limit points, contradicting uniqueness.
		
		Denote now the limit by $f^\infty$, that is $f^\mathcal{K}\xrightarrow{\mathcal{K}\to\infty}f^\infty$ in weak-$*$ $L^\infty([0,50];L^\infty)$ and strongly in $C^0([0,42-\epsilon];L^p)$ for all $\epsilon > 0$.
		
		We are left to show convergence in $L^p([0,42];L^p)$ and $C^0([0,42-\epsilon];L^p)$ for all $p \in [1,\infty)$, and all $\epsilon >0$. The latter of these follows by interpolation between convergence in $C^0([0,42-\epsilon];L^1)$ and the existing uniform bound in $L^\infty L^\infty$. When combined with the uniform bound in $L^2([0,42];L^2)$ this further implies convergence of the norm $\left\|f^\mathcal{K}\right\|_{L^2([0,42];L^2)}$. Convergence in $L^2([0,42];L^2)$ then follows from the already proved weak convergence in $L^2([0,42];L^2)$. The analogous result for $p\in[1,2)$ follows from compactness of the domain $\mathbb{T}^2 \times [0,42]$. Moreover, the convergence for $p \in (2,\infty)$ follows by interpolation with the existing uniform bound in $L^\infty([0,42]; L^\infty)$.
		
		Finally, we show the mixing formula \eqref{eqmixed}.
		
		For $K\in\mathbb{N}$ denote by $f^K = f^{(K,K,2,2^{\left\lfloor K/2\right\rfloor})}$, $\mathcal{N}_K = \mathcal{N}_{(K,K,2,2^{\left\lfloor K/2\right\rfloor})}$, and by $\tilde{y}_t^K = \tilde{y}_t^{(K,K,2,2^{\left\lfloor K/2\right\rfloor})}$.
		
		By \eqref{eqtorder}, observe that $(k,m,i,n) \in \mathcal{N}_K$ if and only if $k \le K$, and moreover $(K,K,2,2^{\left\lfloor K/2\right\rfloor}) \xrightarrow{\mathrm{lex}} \infty$ as $K\to\infty$.
		
		Fix $k,m\in\mathbb{N}$ with $k \le K$, $m\in\{1,...,k\}$.
		For all $i \in \{1,2\}$, for all $n\in\{1,...,2^{\left\lfloor k/2\right\rfloor}\}$, we see that $(k,m,i,n) \in \mathcal{N}_K$, and so consider the maps
		\begin{equation*}
			\tilde{y}_{T_{(k,m,i,n)}+3\cdot2^{-k}}^K\circ\big(\tilde{y}_{T_{(k,m,i,n)}}^K\big)^{-1}.
		\end{equation*}
		
		By \eqref{eqbinaryswap}, \eqref{eqsucbound}, \eqref{eqpiecelag}, these maps commute for all $i \in \{1,2\}$, for all $n\in\{1,...,2^{\left\lfloor k/2\right\rfloor}\}$.
		
		Moreover, by \eqref{eqtorder}, \eqref{eqTkn} (illustrated in \eqref{eqtimediag}), their composition is equal to
		\begin{equation*}
			\tilde{y}_{T_{(k,m,1,1)}+3\cdot2^{-k}}^K \circ \big(\tilde{y}_{T_{(k,m,2,2^{\lfloor k/2 \rfloor})}}^K\big)^{-1}.
		\end{equation*}
		
		Furthermore, we have the following expression for this map in terms of binary expansions.
		
		For $x=(x_1,x_2)\in\mathbb{T}^2$ denote by $(x'_1,x'_2)=\tilde{y}_{T_{(k,m,1,1)}+3\cdot2^{-k}}^K \circ \big(\tilde{y}_{T_{(k,m,2,2^{\lfloor k/2 \rfloor})}}^K\big)^{-1}(x)$.
		
		For a.e. $x\in\mathbb{T}^2$, for all $j\in\{1,2\}$, $x'_j$ has swapping the $k$\textsuperscript{th} and $(k+1)$\textsuperscript{th} binary digits of $x_j$, illustrated by the map $\textcolor{red}{Y_k}$ below. For the convenience of the reader, we use colour to denote swaps in the binary expansion. Following the notation for binary expansions introduced in \eqref{eqbinexp}, we have that for $j\in\{1,2\}$ the coordinate, and for $l\in\mathbb{N}$ the binary digit,
		\begin{align*}
			\left(\tilde{y}_{T_{(k,m,1,1)}+3\cdot2^{-k}}^K \circ \big(\tilde{y}_{T_{(k,m,2,2^{\lfloor k/2 \rfloor})}}^K\big)^{-1}(x)\right)_j & = \Bigg( \begin{tikzpicture}[baseline=(n1.base)]
				\pgfdeclarelayer{background}
				\pgfsetlayers{background,main}
				\begin{scope}[scale=1/4]
					\node[rectangle] (n1) at (0,0) {$0\;.\;x_{j,1}$};
					\node[rectangle] (n5) [right=0mm of n1.south east, anchor=south west, yshift=2.5pt] {$...$};
					\node[rectangle] (n7) [right=0mm of n5.south east, anchor=south west, yshift=-2.5pt] {$x_{j,k}$};
					\node[rectangle] (n8) [right=-1mm of n7.south east, anchor=south west] {$x_{j,k+1}$};
					\node[rectangle] (n9) [right=0mm of n8.south east, anchor=south west, yshift=2.5pt] {$...$};
					\draw[<->] ([xshift=10mm] n8.north west) arc (0:180:15mm);
					\node[rectangle,text=red] (l0) at ([xshift=-2.5mm, yshift=25mm] n8.north west) {$Y_{k}$};
				\end{scope}
			\end{tikzpicture} \Bigg), \\
			\left(\tilde{y}_{T_{(k,m,1,1)}+3\cdot2^{-k}}^K \circ \big(\tilde{y}_{T_{(k,m,2,2^{\lfloor k/2 \rfloor})}}^K\big)^{-1}(x)\right)_{j,l} & = \begin{cases}
				x_{j,k+1} & \text{if } l=k, \\
				x_{j,k} & \text{if } l=k+1, \\
				x_{j,l} & \text{otherwise}.
			\end{cases}
		\end{align*}
		
		Next, we fix $m \in \{1,...,K\}$, and use \eqref{eqtorder}, \eqref{eqTkn}, \eqref{eqtm}, (illustrated in \eqref{eqtimediag}), to piece together $\tilde{y}_{T_{(k,m,1,1)}+3\cdot2^{-k}}^K \circ \big(\tilde{y}_{T_{(k,m,2,2^{\lfloor k/2 \rfloor})}}^K\big)^{-1}$ for $k \in \{m,m+1,...\}$.
		
		From this, we deduce that for a.e. $x\in\mathbb{T}^2$, for all $m \in \{1,...,K\}$, and for $j\in\{1,2\}$ the coordinate, and $l\in\mathbb{N}$ the binary digit, we have that
		\begin{align*}
			\left(\tilde{y}_{T_m}^{K} \circ \big(\tilde{y}_{T_{m-1}}^{K}\big)^{-1} (x)\right)_{j} & = \Bigg( \begin{tikzpicture}[baseline=(n1.base),scale=1]
				\pgfdeclarelayer{background}
				\pgfsetlayers{background,main}
				\begin{scope}[scale=1/4]
					\node[rectangle] (n1) {$0\;.\;x_{j,1}$};
					\node[rectangle] (n2) [right=0mm of n1.south east, anchor=south west, yshift=2.5pt] {$...$};
					\node[rectangle] (n3) [right=0mm of n2.south east, anchor=south west, yshift=-2.5pt] {$x_{j,m}$};
					\node[rectangle] (n5) [right=3mm of n3.south east, anchor=south west, yshift=2.5pt] {$...$};
					\node[rectangle] (n7) [right=3mm of n5.south east, anchor=south west, yshift=-2.5pt] {$x_{j,K}$};
					\node[rectangle] (n8) [right=-1mm of n7.south east, anchor=south west] {$x_{j,K+1}$};
					\node[rectangle] (n9) [right=0mm of n8.south east, anchor=south west, yshift=2.5pt] {$...$};
					\draw[->] ([xshift=10mm] n8.north west) arc (0:180:10mm);
					\draw[->] ([xshift=-20mm] n8.north west) arc (0:180:10mm);
					\draw[->] ([xshift=-50mm] n8.north west) arc (0:180:10mm);
					\draw[->] ([xshift=-80mm] n8.north west) arc (0:180:10mm);
					\node[rectangle,text=red] (l0) at ([xshift=0mm, yshift=20mm] n8.north west) {$Y_{K}$};
					\node[rectangle,text=red] (l1) at ([xshift=-30mm, yshift=20mm] n8.north west) {$Y_{K-1}$};
					\node[rectangle,text=red] (l1) at ([xshift=-60mm, yshift=20mm] n8.north west) {$...$};
					\node[rectangle,text=red] (l2) at ([xshift=-90mm, yshift=20mm] n8.north west) {$Y_{m}$};
				\end{scope}
			\end{tikzpicture} \Bigg), \\
			\left(\tilde{y}_{T_m}^{K} \circ \big(\tilde{y}_{T_{m-1}}^{K}\big)^{-1} (x)\right)_{j,l} & = \begin{cases}
				x_{j,K+1} & \text{if } l=m, \\
				x_{j,l-1} & \text{if } m < l\le K+1, \\
				x_{j,l} & \text{otherwise}.
			\end{cases}
		\end{align*}
		
		Therefore, since $\tilde{y}_{T_0}^{K}=\mathrm{Id}$, we see that for a.e. $x\in\mathbb{T}^2$, for all $m \in \{1,...,K\}$, and for $j\in\{1,2\}$ the coordinate, and $l\in\mathbb{N}$ the binary digit,
		\begin{align}
			\left(\big(\tilde{y}_{T_m}^K\big)^{-1}(x)\right)_j & = \left( \begin{tikzpicture}[baseline=(n1.base),scale=1]
				\pgfdeclarelayer{background}
				\pgfsetlayers{background,main}
				\begin{scope}[scale=1/4]
					\node[rectangle] (n1) {$0\;.\;x_{j,\textcolor{red}{m+1}}$};
					\node[rectangle] (n2) [right=-1mm of n1.south east, anchor=south west, yshift=0pt] {$x_{j,\textcolor{red}{m+2}}$};
					\node[rectangle] (n3) [right=0mm of n2.south east, anchor=south west, yshift=2.5pt] {$...$};
					\node[rectangle] (n5) [right=0mm of n3.south east, anchor=south west, yshift=-2.5pt] {$x_{j,\textcolor{red}{K+1}}$};
					\node[rectangle] (n6) [right=-1mm of n5.south east, anchor=south west, yshift=0pt] {$x_{j,\textcolor{magenta}{m}}$};
					\node[rectangle] (n7) [right=-1mm of n6.south east, anchor=south west, yshift=0pt] {$x_{j,\textcolor{magenta}{m-1}}$};
					\node[rectangle] (n8) [right=0mm of n7.south east, anchor=south west, yshift=2.5pt] {$...$};
					\node[rectangle] (n10) [right=0mm of n8.south east, anchor=south west, yshift=-2.5pt] {$x_{j,\textcolor{magenta}{1}}$};
					\node[rectangle] (n11) [right=-1mm of n10.south east, anchor=south west, yshift=0pt] {$x_{j,K+2}$};
					\node[rectangle] (n12) [right=-1mm of n11.south east, anchor=south west, yshift=0pt] {$x_{j,K+3}$};
					\node[rectangle] (n12) [right=0mm of n12.south east, anchor=south west, yshift=2.5pt] {$...$};
				\end{scope}
			\end{tikzpicture} \right), \nonumber \\
			\left(\big(\tilde{y}_{T_m}^K\big)^{-1}(x)\right)_{j,l} & = \begin{cases}
				x_{j,m+l} & \text{if } l \le K+1-m, \\
				x_{j,K+2-l} & \text{if } K+2-m \le l \le K+1, \\
				x_{j,l} & \text{if } l \ge K+2.
			\end{cases} \label{eqbinaryexp}
		\end{align}
		
		For each $m\in\mathbb{N}$ define now the map $z_m:\mathbb{T}^2 \to \mathbb{T}^2$ by, for all $x\in\mathbb{T}^2$, for $j\in\{1,2\}$ to coordinate, and $l\in\mathbb{N}$ the binary digit,
		\begin{align*}
			\left(z_m(x)\right)_j
			& = \left( \begin{tikzpicture}[baseline=(n1.base),scale=1]
				\pgfdeclarelayer{background}
				\pgfsetlayers{background,main}
				\begin{scope}[scale=1/4]
					\node[rectangle] (n1) {$0\;.\;x_{j,\textcolor{red}{m+1}}$};
					\node[rectangle] (n2) [right=-1mm of n1.south east, anchor=south west, yshift=0pt] {$x_{j,\textcolor{red}{m+2}}$};
					\node[rectangle] (n3) [right=-1mm of n2.south east, anchor=south west, yshift=0pt] {$x_{j,\textcolor{red}{m+3}}$};
					\node[rectangle] (n4) [right=0mm of n3.south east, anchor=south west, yshift=2.5pt] {$...$};
				\end{scope}
			\end{tikzpicture} \right), \\
			\left(z_m(x)\right)_{j,l} & = \begin{cases}
				x_{j,m+l} & \text{for all $l\in\mathbb{N}$}.
			\end{cases}
		\end{align*}
		
		This map is an approximation of \eqref{eqbinaryexp}. Notice that for the initial data $f_0 \in L^\infty(\mathbb{T}^2)$, $f_0\circ z_m$ contains $4^m$ scaled copies of $f_0$, one on each tile in the square lattice with widths $2^{-m}$.
		
		For any $r, r' \in \{1,...,2^{k-1}\}$, define the intervals $J=[(r-1)2^{1-k},r 2^{1-k}]$, $J'=[(r'-1)2^{1-k},r' 2^{1-k}]$. Then $z_m$ is a bijection from this tile $J\times J'$ to $\mathbb{T}^2$, that is
		\begin{equation*}
			z_m:J\times J'\leftrightarrow\mathbb{T}^2.
		\end{equation*}
		
		Moreover, for $d\mu$ the Lebesgue-measure on $\mathbb{T}^2$, $\left.z_m\right|_{J\times J'}\circ d\mu = 4^{-m} d\mu$. In particular $z_m:\mathbb{T}^2 \to \mathbb{T}^2$ is measure preserving, and
		\begin{equation*}
			\int_{J\times J'} f_0\circ z_m\;dx = \frac{1}{4^m}\int_{\mathbb{T}^2}f_0\;dx.
		\end{equation*}
		
		Therefore, we deduce that
		\begin{equation}\label{eqcopy}
			f_0\circ z_m \xrightharpoonup{m\to\infty} \int_{\mathbb{T}^2}f_0(y)\;dy,
		\end{equation}
		with the above convergence in weak-$*$ $L^\infty(\mathbb{T}^2)$.
		
		By \eqref{eqlagexpr}, for all $K\in\mathbb{N}$, $m\in\{1,...,L\}$, we may write $f^K(\cdot, T_m) = f_0\circ\big(\tilde{y}_{T_m}^K\big)^{-1}$. We now wish to approximate $f^K(\cdot, T_m)$ by $f_0\circ z_m$ to prove the mixing formula \eqref{eqmixed}.
		
		Fix some $\epsilon > 0$. Take a smooth approximation $\phi_\epsilon\in C^\infty(\mathbb{T}^2)$ of $f_0$ in $L^1(\mathbb{T}^2)$, so that $\left\|\phi_\epsilon - f_0\right\|_{L^1(\mathbb{T}^2)}\le\epsilon$. Since $\big(\tilde{y}_{T_m}^K\big)^{-1}$, $z_m$ are measure preserving, we see that also
		\begin{equation}\label{eqsmoothaprox}
			\begin{gathered}
				\left\|\phi_\epsilon \circ \big(\tilde{y}_{T_m}^K\big)^{-1} - f_0 \circ \big(\tilde{y}_{T_m}^K\big)^{-1}\right\|_{L^1(\mathbb{T}^2)}\le\epsilon, \\
				\left\|\phi_\epsilon \circ z_m - f_0 \circ z_m\right\|_{L^1(\mathbb{T}^2)}\le\epsilon.
			\end{gathered}
		\end{equation}
		
		Next, consider the square lattice with widths $2^{m-K-1}$, that is for $r, r' \in \{1,...,2^{1+K-m}\}$ define $J=[(r-1)2^{m-K-1},r2^{m-K-1}]$, and $J'=[(r-1)2^{m-K-1},r2^{m-K-1}]$.
		
		Then, by \eqref{eqbinaryexp}, we see that for a.e. $x\in\mathbb{T}^2$, $\big(\tilde{y}_{T_m}^K\big)^{-1}(x) \in J\times J'$ if and only if $z_m (x) \in J \times J'$. But then, by the Lipschitz bound on $\phi_\epsilon$,
		\begin{equation*}
			\left|\phi_\epsilon \circ \big(\tilde{y}_{T_m}^K\big)^{-1}(x) - \phi_\epsilon \circ z_m (x)\right| \le \sqrt{2} \left\|\nabla \phi_\epsilon\right\|_{L^\infty} 2^{m-K-1}.
		\end{equation*}
		
		Therefore,
		\begin{equation*}
			\left\|\phi_\epsilon\circ\big(\tilde{y}_{T_m}^K\big)^{-1} - \phi_\epsilon\circ z_m \right\|_{L^\infty(\mathbb{T}^2)} \le \sqrt{2} \left\|\nabla \phi_\epsilon\right\|_{L^\infty} 2^{m-K-1}.
		\end{equation*}
		
		In light of \eqref{eqsmoothaprox}, we deduce that
		\begin{equation}\label{eqzmbound}
			\left\|f^K(\cdot,T_m) - f_0\circ z_m \right\|_{L^1(\mathbb{T}^2)} \le 2\epsilon + \sqrt{2} \left\|\nabla \phi_\epsilon\right\|_{L^\infty} 2^{m-K-1}.
		\end{equation}
		
		Recall that $f^K\xrightarrow{K\to\infty}f^\infty$ strongly in $C^0([0,42-\epsilon];L^1(\mathbb{T}^2))$ for all $\epsilon > 0$. Therefore, for $m \in \mathbb{N}$ fixed, $f^K(\cdot, T_m)\xrightarrow{K\to\infty}f^\infty(\cdot, T_m)$ strongly in $L^1(\mathbb{T}^2)$.
		
		So, by \eqref{eqzmbound}, for all $m\in\mathbb{N}$, we see that $f^\infty(\cdot, T_m) = f_0\circ z_m$. That is $f^\infty(\cdot, T_m)$ contains $4^m$ scaled copies of $f_0$, one on each tile in the square lattice with widths $2^{-m}$.
		
		Hence, by \eqref{eqcopy}, $f^\infty(\cdot, T_m) \xrightharpoonup{m\to\infty} \int_{\mathbb{T}^2}f_0(y)\;dy$ converges in weak-$*$ $L^\infty(\mathbb{T}^2)$.
		
		Since $f^\infty\in C_{\mathrm{weak-}*}^0L^\infty$, with $f^\infty(\cdot, t)$ independent of $t\in [42,50]$, and $T_m \xrightarrow{m\to\infty}42$, we have proved the mixing formula \eqref{eqmixed}.
	\end{proof}
	\vspace{\baselineskip}
	
	Finally, it remains to find a suitably fast-growing sequence $\{L_\mathcal{P}\}_{\mathcal{P}\in\mathcal{N}}\subset\mathbb{N}$ in Definition \ref{perfectmixing}. Then to apply Theorem \ref{vcontrol} to the sequence of vector fields $u_\mathcal{K}^{\{L_\mathcal{P}\}_{\mathcal{P}\in\mathcal{N_K}}}$, indexed by $\mathcal{K}\in\mathcal{N}$. This will allow us to approximate the vanishing viscosity limit to \eqref{eqADE} along $u_\infty^{\{L_\mathcal{P}\}_{\mathcal{P}\in\mathcal{N}}}$. Additionally, by subsequently time-reversing the vector field $u_\infty^{\{L_\mathcal{P}\}_{\mathcal{P}\in\mathcal{N}}}$ on the time interval $[50,100]$, we will obtain the inadmissible behaviour below.
	\begin{theorem}[Inadmissible vanishing viscosity limit]\label{perfmixvv}
		There exists a divergence-free vector field $u \in L^\infty([0,100];L^\infty(\mathbb{T}^2;\mathbb{R}^2))$, such that for any initial data $f_0\in L^\infty(\mathbb{T}^2)$, and for $f^\nu$ the unique solution to \eqref{eqADE} along $u$ with initial data $f_0$, one has
		\begin{equation*}
			f^\nu\xrightarrow{\nu\to0}f,
		\end{equation*}
		with the above convergence in weak-$*$ $L^\infty([0,100];L^\infty(\mathbb{T}^2))$, strong in $L^p([0,42];L^p(\mathbb{T}^2))$, $C^0([0,42-\epsilon];L^p(\mathbb{T}^2))$, $L^p([58,100];L^p(\mathbb{T}^2))$, and $C^0([58+\epsilon,100];L^p(\mathbb{T}^2))$ for all $p \in [1,\infty)$, and all $\epsilon >0$. The limit function $f \in C_{\mathrm{weak-}*}^0([0,100];L^\infty(\mathbb{T}^2))$ is a weak solution to \eqref{eqTE} along $u$ with initial data $f_0$.
		
		Moreover, for all $t \in [42,58]$
		\begin{equation*}
			f(\cdot, t) \equiv \int_{\mathbb{T}^2}f_0(y)\;dy,
		\end{equation*}
		is perfectly mixed to its spatial average.
		
		Furthermore, for all $t \in [0,100]$, $f(\cdot, t)=f(\cdot,100-t)$ and in particular,
		\begin{equation*}
			f(\cdot, 100) = f_0,
		\end{equation*}
		is perfectly unmixed. In particular, if $f_0$ is not constant, any $L^p(\mathbb{T}^2)$ norms of $f(\cdot, t)$ (for $p\in(1,\infty]$) increase after $t=58$, contrary to the entropy-admissibility criterion in \cite{dafermos2005hyperbolic}.
	\end{theorem}
	\begin{proof}
		Recall the language and notation introduced in Definitions \ref{totalorders}, \ref{perfectmixing}, as well as Definitions \ref{renormalised}, \ref{lagrangian} of renormalised and Lagrangian solutions to \eqref{eqTE}.
		
		For each $n \in \mathbb{N}$ we denote by $\mathcal{K}_n \in \mathcal{N}$ the isomorphism between the well orders $(\mathbb{N}, <)$ and $(\mathcal{N},<_\mathrm{lex})$. That is $n\mapsto\mathcal{K}_n$ is a bijection from $\mathbb{N}$ to $\mathcal{N}$, and for all $n_1, n_2 \in\mathbb{N}$, we have that $n_1 < n_2$ if and only if $\mathcal{K}_{n_1} <_\mathrm{lex} \mathcal{K}_{n_2}$.
		
		For an infinite sequence $\{L_\mathcal{P}\}_{\mathcal{P} \in \mathcal{N}}\subset\mathbb{N}$ with $L_{(k,m,i,q)} \ge k+1$ for all $(k,m,i,q) \in \mathcal{N}$, we define $u_n:\mathbb{T}^2 \times [0,100] \to \mathbb{R}^2$ by
		\begin{equation}\label{equnform}
			u_n(x,t) = \begin{cases}
				u_{\mathcal{K}_n}^{\{L_\mathcal{P}\}_{\mathcal{P}\in\mathcal{N}_{\mathcal{K}_n}}}(x,t) & \text{if } t \in [0,50], \\
				-u_{\mathcal{K}_n}^{\{L_\mathcal{P}\}_{\mathcal{P}\in\mathcal{N}_{\mathcal{K}_n}}}(x,100-t) & \text{if } t \in [50,100].
			\end{cases}
		\end{equation}
		
		$u_n \in L^\infty L^\infty$ is then bounded by $1$.
		
		Let $f_0 \in L^\infty(\mathbb{T}^2)$, then for $f^{\mathcal{K}_n}$ given by Proposition \ref{mixinglemma}, we have the following Lagrangian solution $f^n\in C^0L^1$ to \eqref{eqTE} along $u_n$ with initial data $f_0$,
		\begin{equation}\label{eqflip}
			f^n(x,t) = \begin{cases}
				f^{\mathcal{K}_n}(x,t) & \text{if } t \in [0,50], \\
				f^{\mathcal{K}_n}(x,100-t) & \text{if } t \in [50,100].
			\end{cases}
		\end{equation}
		
		We now apply Theorem \ref{vcontrol}. Let $d_*$ be a metric inducing the weak-$*$ topology on
		\begin{equation*}
			\left\{u\in L^\infty \left( [0,100];L^\infty(\mathbb{T}^2;\mathbb{R}^2)\right) : \left\|u\right\|_{L^\infty L^\infty}\le 1 \right\}.
		\end{equation*}
		
		Let $f_0 \in L^\infty(\mathbb{T}^2)$, and denote for each $n\in\mathbb{N}$, $\nu > 0$, by $f^{n,\nu}$, respectively $f^n$, the unique weak solution to \eqref{eqADE}, respectively \eqref{eqTE}, along $u_n$ with initial data $f_0$. Moreover denote by $f^{\infty,\nu}$ the unique weak solution to \eqref{eqADE} along $u_\infty^{\{L_\mathcal{P}\}_{\mathcal{P}\in\mathcal{N}}}$ with initial data $f_0$.
		
		Then by Theorem \ref{vcontrol},
		\begin{enumerate}[label=\textbf{S.\arabic*},ref=S.\arabic*]
			\item \label{listnu} For all $n \in \mathbb{N}$ there exists $\nu_n > 0$, $\epsilon_n > 0$ depending only on $\{L_\mathcal{P}\}_{\mathcal{P}\in\mathcal{N}_{\mathcal{K}_n}}$ (and in particular not on $f_0$), with $\nu_n\xrightarrow{n\to\infty}0$ monotonically, such that the following hold true:
			\item \label{listupper} For all $p\in[1,\infty)$
			\begin{equation*}
				\sup_{\nu\le\nu_n}\left\|f^{n,\nu}-f^n\right\|_{L^\infty L^p}\xrightarrow{n\to\infty}0,
			\end{equation*}
			\setcounter{enumi}{3}
			\item \label{listlower} If $d_*(u_{n+1},u_n) \le \epsilon_n$ for all $n\in\mathbb{N}$, then for all $p \in [1,\infty)$
			\begin{equation*}
				\sup_{\nu_n \le \nu\le \nu_1}\left\|f^{n,\nu}-f^{\infty,\nu}\right\|_{L^\infty L^p}\xrightarrow{n\to\infty}0.
			\end{equation*}
		\end{enumerate}
		
		We now choose $\{L_\mathcal{P}\}_{\mathcal{P}\in\mathcal{N}}$. Proceeding by induction on $N\in\mathbb{N}$, assume there exists a sequence $\{L_\mathcal{P}\}_{\mathcal{P}\in\mathcal{N}_{\mathcal{K}_N}}\subset\mathbb{N}$, so that with $\{u_{n}\}_{n=1}^N \subset L^\infty([0,100];L^\infty(\mathbb{T}^2;\mathbb{R}^2))$ given by \eqref{equnform}, and $\{\epsilon_{n}\}_{n=1}^N\subset (0,\infty)$ given by \eqref{listnu}, we have that for all $n \in \{1,...,N-1\}$,
		\begin{equation*}
			d_*(u_{n+1},u_{n}) \le \epsilon_{n}.
		\end{equation*}
		
		We next pick $L_{\mathcal{K}_{N+1}}\in\mathbb{N}$. Note that for any $L_{\mathcal{K}_{N+1}}\in\mathbb{N}$ we obtain by \eqref{equnform} a vector field $u_{N+1}$.
		
		By \eqref{eqweakconv} and Definition \ref{perfectmixing}, we see that as $L_{\mathcal{K}_{N+1}}\to\infty$,
		\begin{equation*}
			d_*(u_{N+1},u_N)\xrightarrow{(L_{\mathcal{K}_{N+1}})\to\infty}0,
		\end{equation*}
		and so we may pick $L_{\mathcal{K}_{N+1}}$ large enough that $d_*(u_{N+1},u_n) \le \epsilon_N$. This completes the inductive step.
		
		That is, there exists a sequence $\{L_\mathcal{P}\}_{\mathcal{P}\in\mathcal{N}}\subset\mathbb{N}$ such that for all $n\in\mathbb{N}$, with $u_n$ given by \eqref{equnform}, we have that $d_*(u_{n+1},u_n) \le \epsilon_n$. Therefore, \eqref{listupper}, and \eqref{listlower} are satisfied.
		
		Next, for all $n \in\mathbb{N}$, and for all $p\in[1,\infty)$, we see that
		\begin{align*}
			\sup_{\nu_{n+1} \le \nu\le \nu_n}\left\|f^{\infty,\nu}-f^n\right\|_{L^\infty L^p} \le & \sup_{\nu_{n+1} \le \nu\le \nu_n}\left\|f^{\infty,\nu}-f^{n+1,\nu}\right\|_{L^\infty L^p} \\
			& + \sup_{\nu_{n+1} \le \nu\le \nu_n}\left\|f^{n+1,\nu}-f^{n,\nu}\right\|_{L^\infty L^p} \\
			& + \sup_{\nu_{n+1} \le \nu\le \nu_n}\left\|f^{n,\nu}-f^n\right\|_{L^\infty L^p},
		\end{align*}
		
		Therefore, by \eqref{listupper}, \eqref{listlower}, if for all $p\in[1,\infty)$,
		\begin{equation}\label{eqfinaleq}
			\sup_{\nu_{n+1} \le \nu\le \nu_n}\left\|f^{n+1,\nu}-f^{n,\nu}\right\|_{L^\infty L^p} \xrightarrow{n\to\infty}0,
		\end{equation}
		then also for all $p\in[1,\infty)$,
		\begin{equation*}
			\sup_{\nu_{n+1} \le \nu\le \nu_n}\left\|f^{\infty,\nu}-f^n\right\|_{L^\infty L^p}\xrightarrow{n\to\infty}0.
		\end{equation*}
		
		Therefore, the statement of Theorem \ref{perfmixvv} follows from Proposition \ref{mixinglemma} applied to \eqref{eqflip}.
		
		Notice that for all $p\in(1,\infty)$, \eqref{eqfinaleq} follows from the case $p=1$ by interpolation with the existing uniform bound in $L^\infty L^\infty$. We, therefore, only prove \eqref{eqfinaleq} for $p=1$.
		
		Let $n\in\mathbb{N}$. Express $\mathcal{K}_{n+1} \in \mathcal{N}$ as $\mathcal{K}_{n+1} = (k,m,i,q)$ with $k\in\mathbb{N}$, $m\in\{1,...,k\}$, $i\in\{1,2\}$, and $q\in\{1,...,2^{\left\lfloor k/2\right\rfloor}\}$.
		
		Then, by the expression \eqref{equnform}, and Definitions \ref{totalorders}, \ref{perfectmixing}, we have that for all $t \notin [T_{\mathcal{K}_{n+1}},T_{\mathcal{K}_{n+1}}+3\cdot2^{-k}] \cup [100-T_{\mathcal{K}_{n+1}}-3\cdot2^{-k},100-T_{\mathcal{K}_{n+1}}]$, that $u_{n+1}(\cdot, t) = u_n(\cdot, t)$.
		
		Therefore, for all $\nu>0$, $f^{n+1,\nu}-f^{n,\nu} \in C^0L^1$ is a solution to \eqref{eqADE} on the time interval $[0,T_{\mathcal{K}_{n+1}}]$ along the same $u_n$ with initial data $0\in L^\infty(\mathbb{T}^2)$.
		
		Similarly, on the time interval $[T_{\mathcal{K}_{n+1}}+3\cdot2^{-k},100-T_{\mathcal{K}_{n+1}}-3\cdot2^{-k}]$ with initial data $(f^{n+1,\nu}-f^{n,\nu})(\cdot,T_{\mathcal{K}_{n+1}}+3\cdot2^{-k})\in L^\infty(\mathbb{T}^2)$.
		
		Similarly, on the time interval $[100-T_{\mathcal{K}_{n+1}},100]$ with initial data $(f^{n+1,\nu}-f^{n,\nu})(\cdot,100-T_{\mathcal{K}_{n+1}})\in L^\infty(\mathbb{T}^2)$.
		
		Applying the $L^p$-Inequality \eqref{eqlpineq} to these three cases shows that
		\begin{gather}
			\left\|f^{n+1,\nu}-f^{n,\nu}\right\|_{L^\infty([0,T_{\mathcal{K}_{n+1}}]; L^1(\mathbb{T}^2))} = 0, \label{eqinitheatbound} \\
			\left\|f^{n+1,\nu}-f^{n,\nu}\right\|_{L^\infty([T_{\mathcal{K}_{n+1}}+3\cdot2^{-k},100-T_{\mathcal{K}_{n+1}}-3\cdot2^{-k}]; L^1(\mathbb{T}^2))} \le \left\|\left(f^{n+1,\nu}-f^{n,\nu}\right)(\cdot, T_{\mathcal{K}_{n+1}}+3\cdot2^{-k})\right\|_{L^1(\mathbb{T}^2)}, \label{eqinitheatbound2} \\
			\left\|f^{n+1,\nu}-f^{n,\nu}\right\|_{L^\infty([100-T_{\mathcal{K}_{n+1}},100]; L^1(\mathbb{T}^2))} \le \left\|\left(f^{n+1,\nu}-f^{n,\nu}\right)(\cdot, 100-T_{\mathcal{K}_{n+1}})\right\|_{L^1(\mathbb{T}^2)}. \nonumber
		\end{gather}
		
		Therefore, using the continuity $f^{n+1,\nu}-f^{n,\nu}\in C^0L^1$, we have the bound
		\begin{multline}\label{eqheatbound}
			\left\|f^{n+1,\nu}-f^{n,\nu}\right\|_{L^\infty L^1} \\ \le \left\|f^{n+1,\nu}-f^{n,\nu}\right\|_{L^\infty([T_{\mathcal{K}_{n+1}},T_{\mathcal{K}_{n+1}}+3\cdot2^{-k}]\cap[100-T_{\mathcal{K}_{n+1}}-3\cdot2^{-k},100-T_{\mathcal{K}_{n+1}}]; L^1(\mathbb{T}^2))}.
		\end{multline}
		
		Next, for $k\in\mathbb{N}$ given in terms of $n\in\mathbb{N}$ by the expression $\mathcal{K}_{n+1}=(k,m,i,q)$, we aim to prove the bound, for all $n\in\mathbb{N}$, and for all $\nu>0$, 
		\begin{equation}\label{eqlocalheatbound}
			\begin{gathered}
				\left\|f^{n+1,\nu}-f^{n,\nu}\right\|_{L^\infty([T_{\mathcal{K}_{n+1}},T_{\mathcal{K}_{n+1}}+3\cdot2^{-k}];L^1(\mathbb{T}^2))} \\ \le 2\left\|f_0 \right\|_{L^\infty} 2^{-\lfloor k/2 \rfloor} + C\left\|f_0\right\|_{L^\infty}\sqrt{\nu 2^{-k}},
			\end{gathered}
		\end{equation}
		with $C>0$ independent of $n\in\mathbb{N}$ and $\nu>0$.
		
		In the expression $\mathcal{K}_{n+1} = (k,m,i,q)$, we assume, without loss of generality, that $i=1$. The case $i=2$ then follows the same argument with the coordinates reversed.
		
		Let $J=[(q-1)2^{-\left\lfloor k/2 \right\rfloor},q2^{-\left\lfloor k/2 \right\rfloor}]\subset\mathbb{T}$.
		
		Then, by \eqref{eqgradualmixing}, for all $t\in[0,3\cdot2^{-k}]$, and for all $x \notin J\times\mathbb{T}$, we see that the binary swap vector field $u^{(i,k,q,L_{\mathcal{K}_{n+1}})}(x,t)=0$.
		
		Therefore, for all $t \in [T_{\mathcal{K}_{n+1}},T_{\mathcal{K}_{n+1}}+3\cdot2^{-k}]\cap[100-T_{\mathcal{K}_{n+1}}-3\cdot2^{-k},100-T_{\mathcal{K}_{n+1}}]$, and for all $x \notin J\times\mathbb{T}$, we have that $u_{n+1}(x, t), u_n(x, t) = 0$.
		
		We first tackle the case $t\in[T_{\mathcal{K}_{n+1}},T_{\mathcal{K}_{n+1}}+3\cdot2^{-k}]$.
		
		Let $\Omega=[0,1-2^{-\lfloor k/2 \rfloor}]\times\mathbb{T}\subset\mathbb{T}^2$, a periodic strip.
		
		Then, on the spatio-temporal domain $((x_1,x_2),t)\in\Omega\times[0,3\cdot2^{-k}]$, we see that $(f^{n+1,\nu}-f^{n,\nu})((x_1+q2^{-\lfloor k/2 \rfloor},x_2),t+T_{\mathcal{K}_{n+1}})$ is a solution to the heat equation, and by \eqref{eqinitheatbound} has the initial data $0$. However, its boundary data is unknown, so we will construct a second solution to the heat equation on the same domain, with initial data 0, which is an upper bound on the boundary, and then apply the maximum principle.
		
		We have the a priori bound $\left\|f^{n+1,\nu}-f^{n,\nu}\right\|_{L^\infty L^\infty} \le 2\left\|f_0\right\|_{L^\infty}$.
		
		Therefore, we may apply hypo-ellipticity for the heat equation (see for example Section 4.4 in \cite{hormander2003}) to deduce that $(f^{n+1,\nu}-f^{n,\nu})((x_1+q2^{-\lfloor k/2 \rfloor},x_2),t+T_{\mathcal{K}_{n+1}})$ is smooth on the interior of the domain $\Omega\times[0,3\cdot2^{-k}]$.
		
		We introduce
		\begin{gather*}
			\mathrm{erf}(x) = \int_{-\infty}^x e^{-y^2} \; dy, \\
			C_0=\mathrm{erf}\left(0\right), \\
			a = 1-2^{-\lfloor k/2 \rfloor}.
		\end{gather*}
		
		We define the following solution to the heat equation.
		\begin{gather*}
			\theta:\Omega\times[0,3\cdot2^{-k}]\to\mathbb{R}, \\
			\theta((x_1,x_2),t) = 2\left\|f_0\right\|_{L^\infty}C_0^{-1}\left(\mathrm{erf}\left(\frac{-x_1}{\sqrt{4\nu t}}\right) + \mathrm{erf}\left(\frac{x_1-a}{\sqrt{4\nu t}}\right)\right).
		\end{gather*}
		
		Observe that $\theta$ has initial data $0$, and its value on the boundary $\partial\Omega = \{0,a\}\times\mathbb{T}$ is greater than or equal to $2\left\|f_0\right\|_{L^\infty}$. Also for all $(x_1,x_2)\in\Omega$, we have that $\theta((x_1,x_2),t)$ is an increasing function of $t\in[0,3\cdot2^{-k}]$.
		
		Therefore, by maximum principle, we have the following point-wise bound on the interior of the domain, for all $(x_1,x_2) \in (0,1-2^{-\lfloor k/2 \rfloor})\times\mathbb{T}$, and for all $t\in(0,3\cdot2^{-k})$,
		\begin{equation*}
			\left|(f^{n+1,\nu}-f^{n,\nu})((x_1+q2^{-\lfloor k/2 \rfloor},x_2),t+T_{\mathcal{K}_{n+1}})\right| \le \theta((x_1,x_2),3\cdot2^{-k}).
		\end{equation*}
		
		Hence,
		\begin{gather*}
			\left\|f^{n+1,\nu}-f^{n,\nu}\right\|_{L^\infty([T_{\mathcal{K}_{n+1}},T_{\mathcal{K}_{n+1}}+3\cdot2^{-k}];L^1((\mathbb{T}\setminus J)\times\mathbb{T}))} \\ \le 4\left\|f_0\right\|_{L^\infty}C_0^{-1} \int_0^a \mathrm{erf}\left(\frac{-x}{\sqrt{4\nu\cdot 3\cdot 2^{-k}}}\right) \; dx.
		\end{gather*}
		
		After changing variables, for $C=8\sqrt{3}C_0^{-1}\int_0^\infty\mathrm{erf}(-x)\;dx$, we have that
		\begin{gather*}
			\left\|f^{n+1,\nu}-f^{n,\nu}\right\|_{L^\infty([T_{\mathcal{K}_{n+1}},T_{\mathcal{K}_{n+1}}+3\cdot2^{-k}];L^1((\mathbb{T}\setminus J)\times\mathbb{T}))} \\ \le C\left\|f_0\right\|_{L^\infty}\sqrt{\nu 2^{-k}}.
		\end{gather*}
		
		Since additionally $J\times\mathbb{T}$ has Lebesgue-measure $2^{-\lfloor k/2 \rfloor}$, and we have the a priori bound $\left\|f^{n+1,\nu}-f^{n,\nu}\right\|_{L^\infty L^\infty} \le 2\left\|f_0\right\|_{L^\infty}$, we deduce \eqref{eqlocalheatbound}.
		
		Next, since $k$ in determined by $n$ through the map $n\mapsto \mathcal{K}_{n+1}=(k,i,m,q)$, denote this now by $k_n \in \mathbb{N}$.
		
		Since $n\mapsto\mathcal{K}_{n+1}$ respects the well-orders $(\mathbb{N},<)$, $(\mathcal{N},<_\mathrm{lex})$, we see that $k_n\xrightarrow{n\to\infty}\infty$. Moreover, by \eqref{listnu}, we have that $\nu_n\xrightarrow{n\to\infty}0$.
		
		Therefore, by \eqref{eqlocalheatbound}, we see that
		\begin{equation}\label{eqfirstbound}
			\sup_{\nu_{n+1} \le \nu\le \nu_n}\left\|f^{n+1,\nu}-f^{n,\nu}\right\|_{L^\infty([T_{\mathcal{K}_{n+1}},T_{\mathcal{K}_{n+1}}+3\cdot2^{-k_n}];L^1(\mathbb{T}^2))} \xrightarrow{n\to\infty}0.
		\end{equation}
		
		By \eqref{eqheatbound}, \eqref{eqfirstbound}, the proof of \eqref{eqfinaleq} will then be complete if also
		\begin{equation}\label{eqfinalbound}
			\sup_{\nu_{n+1} \le \nu\le \nu_n}\left\|f^{n+1,\nu}-f^{n,\nu}\right\|_{L^\infty([100-T_{\mathcal{K}_{n+1}}-3\cdot2^{-k_n},100-T_{\mathcal{K}_{n+1}}];L^1(\mathbb{T}^2))} \xrightarrow{n\to\infty}0.
		\end{equation}
		
		To this end denote by $g^{n,\nu}\in C^0([100-T_{\mathcal{K}_{n+1}}-3\cdot2^{-k_n},100-T_{\mathcal{K}_{n+1}}],L^1(\mathbb{T}^2))$ the solution to \eqref{eqADE} along $u_n$ on the time interval $[100-T_{\mathcal{K}_{n+1}}-3\cdot2^{-k_n},100-T_{\mathcal{K}_{n+1}}]$ with initial data $f^{n+1,\nu}(\cdot,100-T_{\mathcal{K}_{n+1}}-3\cdot2^{-k_n})\in L^\infty(\mathbb{T}^2)$.
		
		Then, arguing as in \eqref{eqlocalheatbound}, we have the same bound. For all $n\in\mathbb{N}$, for all $\nu>0$,
		\begin{equation*}
			\begin{gathered}
				\left\|f^{n+1,\nu}-g^{n,\nu}\right\|_{L^\infty([100-T_{\mathcal{K}_{n+1}}-3\cdot2^{-k_n},100-T_{\mathcal{K}_{n+1}}];L^1(\mathbb{T}^2))} \\ \le 2\left\|f_0 \right\|_{L^\infty} 2^{-\lfloor k_n/2 \rfloor} + C\left\|f_0\right\|_{L^\infty}\sqrt{\nu 2^{-k_n}}.
			\end{gathered}
		\end{equation*}
		
		Therefore, as before, we deduce that
		\begin{equation}\label{eqsecondbound}
			\sup_{\nu_{n+1} \le \nu\le \nu_n}\left\|f^{n+1,\nu}-g^{n,\nu}\right\|_{L^\infty([100-T_{\mathcal{K}_{n+1}}-3\cdot2^{-k_n},100-T_{\mathcal{K}_{n+1}}];L^1(\mathbb{T}^2))} \xrightarrow{n\to\infty}0.
		\end{equation}
		
		By the $L^p$-Inequality \eqref{eqlpineq}, and \eqref{eqinitheatbound2}, we have that, for all $n\in\mathbb{N}$, and for all $\nu>0$,
		\begin{gather*}
			\left\|g^{n,\nu}-f^{n,\nu}\right\|_{L^\infty([100-T_{\mathcal{K}_{n+1}}-3\cdot2^{-k_n},100-T_{\mathcal{K}_{n+1}}];L^1(\mathbb{T}^2))} \\ \le \left\|\left(f^{n+1,\nu}-f^{n,\nu}\right)(\cdot, T_{\mathcal{K}_{n+1}}+3\cdot2^{-k_n})\right\|_{L^1(\mathbb{T}^2)}.
		\end{gather*}
		
		So, by \eqref{eqfirstbound},
		\begin{equation*}
			\sup_{\nu_{n+1} \le \nu\le \nu_n}\left\|g^{n,\nu}-f^{n,\nu}\right\|_{L^\infty([100-T_{\mathcal{K}_{n+1}}-3\cdot2^{-k_n},100-T_{\mathcal{K}_{n+1}}];L^1(\mathbb{T}^2))}\xrightarrow{n\to\infty}0.
		\end{equation*}
		
		Therefore, by also \eqref{eqsecondbound}, we deduce \eqref{eqfinalbound}.
	\end{proof}
	
	\section*{Acknowledgements}
	Lucas Huysmans acknowledges support from the UK Engineering and Physical Sciences Research Council (EPSRC) under grant numbers EP/T517847/1 and EP/V52024X/1. The work of E.S.T. has benefited from the inspiring environment of the CRC 1114 ``Scaling Cascades in Complex Systems'', Project Number 235221301, Project C06, funded by Deutsche Forschungsgemeinschaft (DFG). The authors would like to thank the Isaac Newton Institute for Mathematical Sciences, Cambridge, for support and hospitality during the program ``Mathematical aspects of turbulence: where do we stand?'' where part of this work was undertaken and supported by EPSRC grant no.~EP/K032208/1.
	
	\bibliographystyle{apa}
	\bibliography{citations.bib}
	
\end{document}